\documentclass[12pt,reqno]{amsart}
\usepackage{amssymb}
\usepackage{mathtools}
\usepackage[dvipsnames]{xcolor}
\usepackage{hyperref}
\usepackage{dsfont}
\usepackage{enumitem}
\usepackage{dsfont}
\usepackage[pdftex]{graphicx}
\usepackage{caption}
\usepackage[labelfont=]{subcaption}
\usepackage{marginnote}
\usepackage{ragged2e}

\usepackage{tikz}
\usetikzlibrary{calc,decorations.pathreplacing,positioning}

\newtheorem{theorem}{Theorem}[section]
\newtheorem{definition}[theorem]{Definition}

\newtheorem{lemma}[theorem]{Lemma}

\newtheorem{remark}[theorem]{Remark}
\newtheorem{assumption}[theorem]{Assumption}

\newtheorem{observation}[theorem]{Observation}
\newtheorem*{claim*}{Claim}

\newtheorem*{remark*}{Remark}

\newcommand{\R}{\mathbb{R}}  
\newcommand{\Z}{\mathbb{Z}}
\newcommand{\N}{\mathbb{N}}

\newcommand*\defn[1]{\textit{#1}}
\newcommand*\interior[1]{#1^\circ}
\newcommand*\closure[1]{\overline{#1}}

\newcommand*\todo[2][]{{\color{red} \ifx\\#1\\ TODO\else TODO(#1)\fi : #2}}

\newcommand{\nocontentsline}[3]{}

\DeclareMathOperator{\ent}{ent}
\DeclareMathOperator{\Ent}{Ent}

\DeclareMathOperator{\Lip}{Lip}

\DeclareMathOperator{\Adm}{Adm}
\DeclareMathOperator{\diam}{diam}

\DeclarePairedDelimiter\ceil{\lceil}{\rceil}
\DeclarePairedDelimiter\floor{\lfloor}{\rfloor}

\let\oldtexttt\texttt
\renewcommand*\texttt[1]{\textnormal{\oldtexttt{#1}}}

\mathtoolsset{showonlyrefs}

\begin{document}

\title[A variational principle with minimal assumptions]
	{Deducing a variational principle with minimal \textit{a~priori}
	assumptions}

\author{Andrew Krieger}
\address{Department of Mathematics, University of California, Los Angeles}
\email{akrieger@math.ucla.edu}

\author{Georg Menz}
\address{Department of Mathematics, University of California, Los Angeles}
\email{gmenz@math.ucla.edu}

\author{Martin Tassy}
\address{Dartmouth College, Hanover}
\email{mtassy@math.dartmouth.edu}


\date{\today}

\begin{abstract}
We study the well-known variational and large deviation principle for graph homomorphisms
from~$\mathbb{Z}^m$ to~$\mathbb{Z}$.
We provide a robust method to deduce those principles
under minimal \textit{a priori} assumptions. The only ingredient specific to the model is a discrete Kirszbraun theorem i.e.~an extension theorem for graph homomorphisms. All other ingredients are of a general nature not specific to the model. They include elementary combinatorics, the compactness of Lipschitz functions and a simplicial Rademacher theorem. Compared to the literature, our proof does not need any other preliminary results like e.g.~concentration or strict convexity of the local surface tension. Therefore, the method is very robust and extends to more complex and subtle models, as e.g.~the homogenization of limit shapes or graph-homomorphisms to a regular tree.
\end{abstract}

\thanks{This research has been partially supported by NSF grant DMS-1712632.}

\maketitle


\tableofcontents

\section{Introduction}


Recently the study of limit shapes has attracted a lot of research.
Limit shapes appear in many models, including
domino tilings and dimer models
(e.g.~\cite{Kas63,CEP96,CKP01};
see Figures~\ref{f_limit_domino},~\ref{f_limit_ribbon}, and~\ref{f_limit_bar}),
polymer models,
lozenge tilings (e.g.~\cite{Des98,LRS01,Wil04}),
Ginzburg-Landau models (e.g.~\cite{DeGiIo00,FuOs04}),
Gibbs models (e.g.~\cite{She05}),
the Ising model (e.g.~\cite{DKS92,Cer06}),
asymmetric exclusion processes (e.g.~\cite{FS06}),
random matrices (e.g.~\cite{Wig59,KS99}),
sandpile models (e.g.\cite{LP08}),
the six vertex model (e.g.~\cite{BCG2016,CoSp16,ReSr16}),
and Young tableaux (e.g.~\cite{LS77,VK77,PR07}).
The appearance of limit shapes seems to be a universal phenomenon:
it has been rigorously proven for many models,
and it is strongly suggested by simulations for many additional models.\\

\begin{figure}[b]
	\centering
	\begin{subfigure}{0.3\textwidth}
		\includegraphics[width=\textwidth]{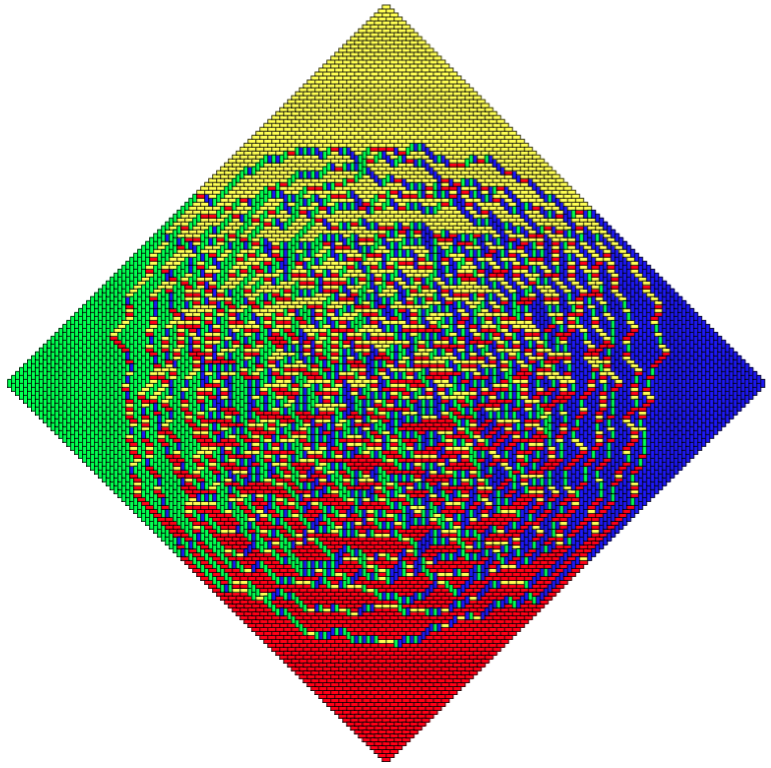}
		\caption{\RaggedRight Domino tiling (see~\cite{CKP01}).}
		\label{f_limit_domino}
	\end{subfigure}
	\begin{subfigure}{0.3\textwidth}
		\includegraphics[width=\textwidth]{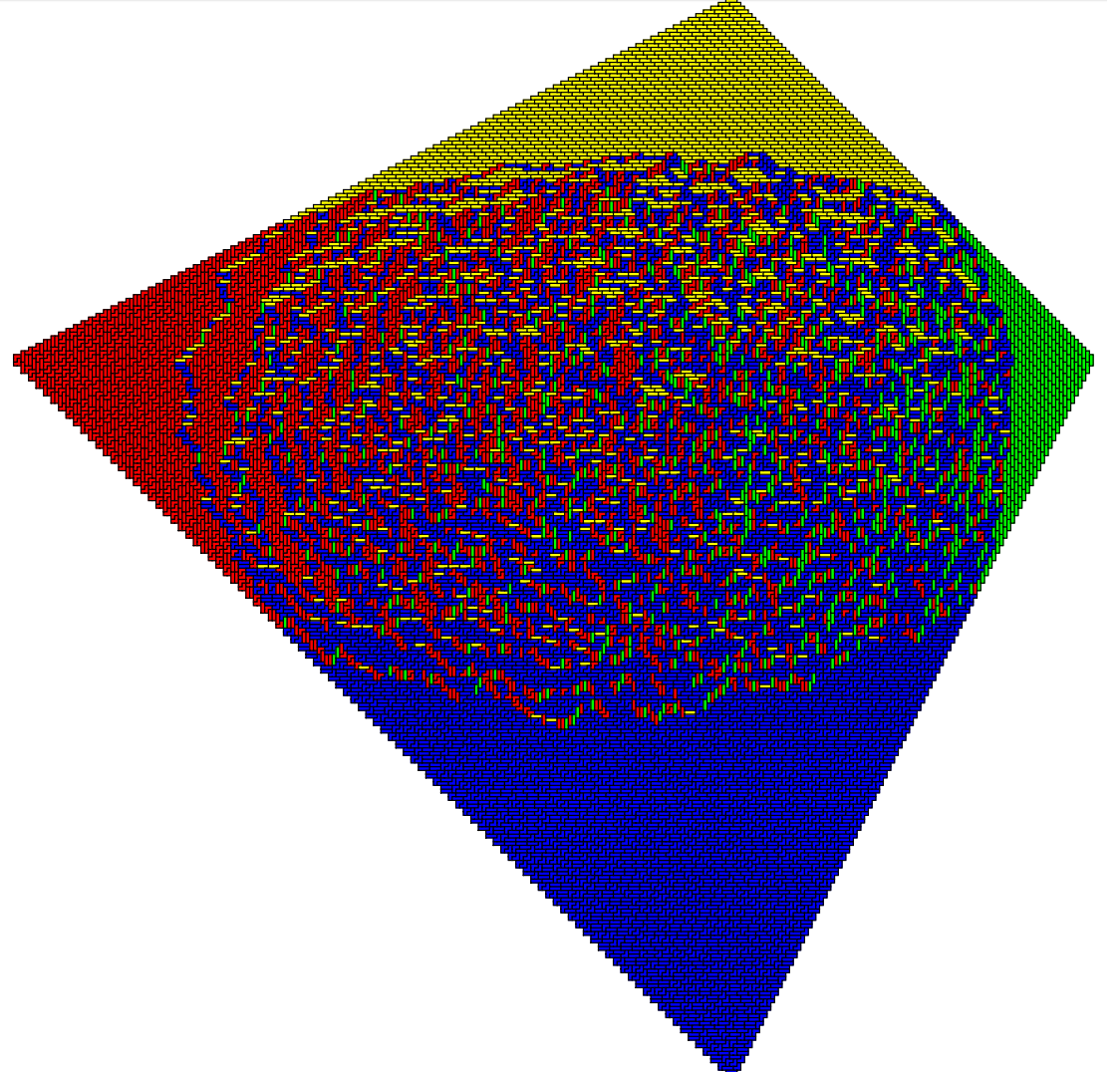}
		\caption{\RaggedRight Ribbon tiling (see~\cite{She02}).}
		\label{f_limit_ribbon}
	\end{subfigure}
	\begin{subfigure}{0.3\textwidth}
		\includegraphics[width=\textwidth]{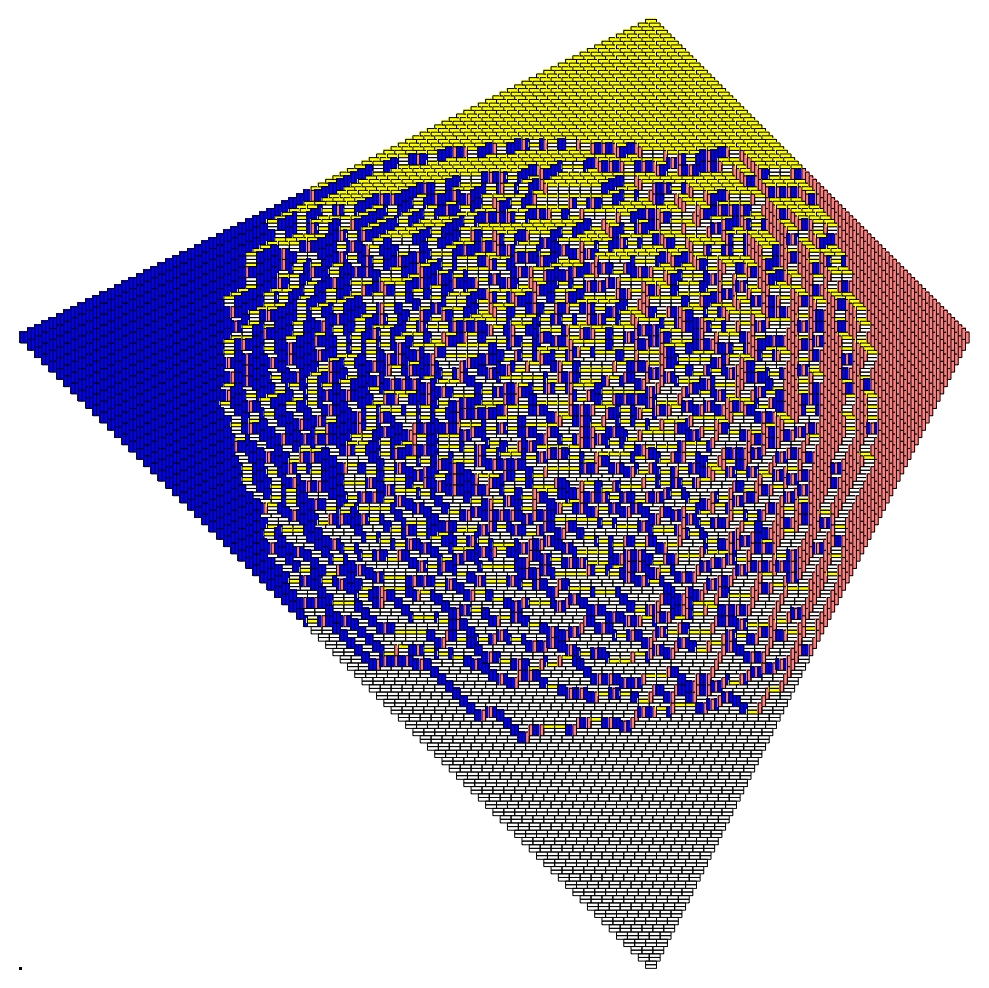}
		\caption{\RaggedRight Tiling by $3 \times 1$ bars
		(see~\cite{KeKe92}).}
		\label{f_limit_bar}
	\end{subfigure}
	\begin{subfigure}{0.3\textwidth}
		\includegraphics[width=\textwidth]{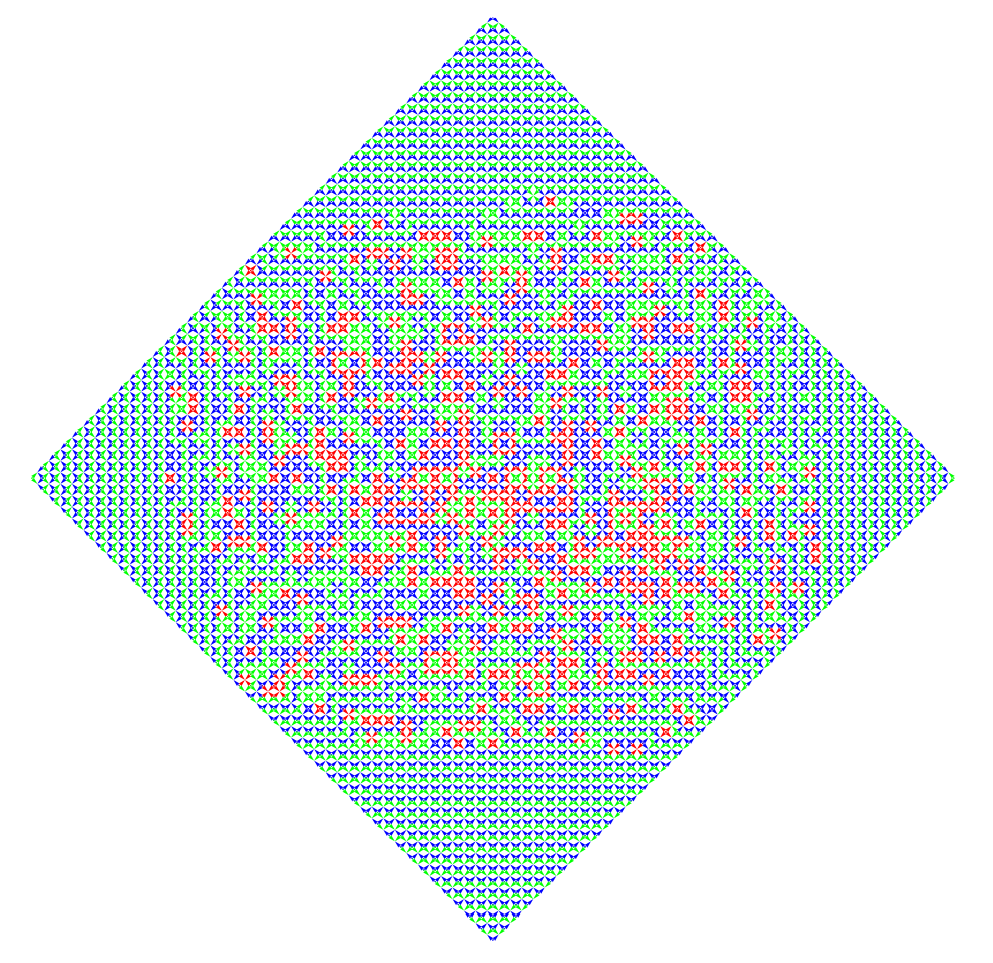}
		\caption{\RaggedRight Graph homomorphisms
		into the $3$-regular tree (see~\cite{MeTa16}).}
		\label{f_limit_3tree}
	\end{subfigure}
	\begin{subfigure}{0.3\textwidth}
		\includegraphics[width=\textwidth]{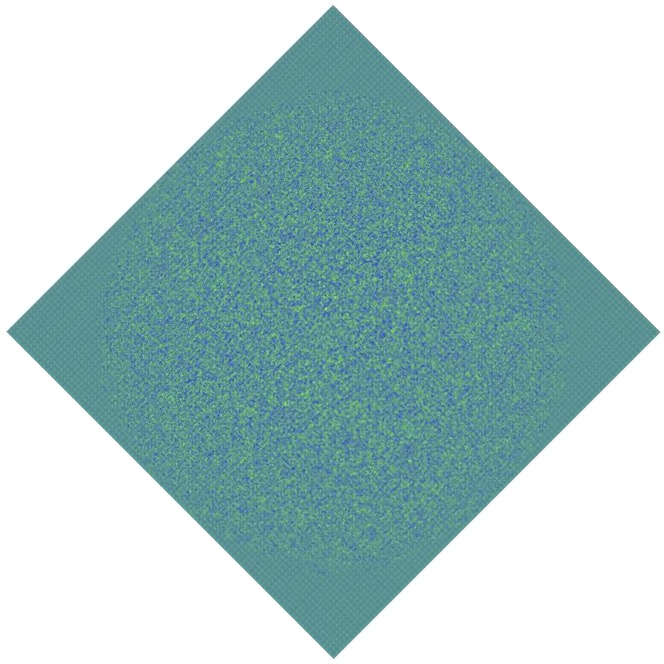}
		\caption{\RaggedRight Graph homomorphism into $\Z$
		(the model studied in this article).}
		\label{f_aztec_constant_2d}
	\end{subfigure}\mbox{$\qquad \qquad $}
\caption{Examples of limit shapes.%
}
\end{figure}

One can explain the appearance of limit shapes
with the help of three technical results.
Those are a profile theorem, a variational principle,
and a large deviations principle.
The profile theorem asymptotically counts the number of microscopic states,
that contribute to a particular macroscopic state, i.e.~a limiting profile.
Given a limiting profile $h$, we might write the profile theorem as
\[
	\Ent_n \bigl(
		\{ \text{microscopic states corresponding to $h$} \} \bigr)
	\approx \Ent(h) \,,
\]
where the microscopic entropy
\[
	\Ent_n(A) := - \frac{1}{n^m} \log |A|
\]
captures the number of microscopic states in a set $A$,
and the macroscopic entropy functional
is given by
\[
	\Ent(h) = \int \ent(\nabla h(x)) \, dx \,.
\]
The integrand $\ent(\nabla h(x))$ is a local quantity,
depending only on the gradient $\nabla h$ at $x$.
This article calls $\ent(\cdot)$ the local surface tension
(see Definition~\ref{d_loc_surf_tens} below).
See also Theorem~\ref{th_profile} for a more precise statement.\\

The variational principle asymptotically characterizes
the number of microscopic states, i.e.~the microscopic entropy~$\Ent_n$,
via a variational problem.
For large system sizes $n$, the microscopic entropy of the system $\Ent_n$
is given by minimizing the macroscopic entropy $\Ent(h)$
over all admissible limiting profiles $h$.
In formula, the variational principle states
(see Theorem~\ref{th_varnprin} below)
\[
	\Ent_n \approx \inf_h \Ent(h) \,.
\]

The large deviations principle complements the variational principle.
It characterizes the asymptotic fraction of microscopic states
in ``natural'' subsets of the microscopic state space,
such as (open or closed) balls around a limiting profile.
Approximately speaking, the large deviations principle states
that the fraction of microscopic states in a suitable subset $A$
decays exponentially in $n$,
with rate given by minimizing a rate function $I(h)$
over the limiting profiles corresponding to $A$
(see Theorem~\ref{th_ldp} below).
The rate function is, up to normalization,
the same as the macroscopic entropy $\Ent(h)$ above.
\\

When the macroscopic entropy has a unique minimizer,
the large deviations principle implies that almost all microscopic states
must approximate this entropy-minimizing profile.
Although the current article does not prove uniqueness of this minimizer,
we briefly discuss uniqueness in Section~\ref{ss_main_results},
after the statement of the large deviations principle (Theorem~\ref{th_ldp}).
\\

The current article continues the line of research started in~\cite{MeTa16},
which strives to develop a robust theory
of variational principles and limit shapes.
The intention of this article is not to provide new results
on subtle and technically challenging models
(as for example in~\cite{MeTa16}) but to provide and explain a simplified method
that only relies on minimal \textit{a priori} assumptions.
In order not to be distracted by unnecessary technical difficulties
we consider the most simple setting:
graph homomorphisms from~$\mathbb{Z}^m$ to~$\mathbb{Z}$.
\\

In more detail, we consider a macroscopic (continuum) domain $R \subset \R^m$,
and a sequence of microscopic (discrete) domains $R_n \subset \Z^m$
such that $\tfrac{1}{n} R_n \to R$ in the Hausdorff metric.
A microscopic state or a height function is a graph homomorphism
$h_{R_n}: R_n \to \Z$ for some $n$,
and a macroscopic state or asymptotic height function is a Lipschitz function
$h_R: R \to \R$.
In the two-dimensional case, this model is equivalent to the six-vertex model
with uniform weights (cf.~\cite{Henk77,ChPeShTa18}).
Full details of the model under study are given in Section~\ref{s_setting}.
\\

Our method emerges from distilling the core arguments
of~\cite{CKP01,She05,MeTa16}.
Compared to those works, our method does not rely on
explicit formulas for the local surface tension, strict convexity,
concentration inequalities, or the FKG inequality.
The robustness of this method is illustrated
in the companion article~\cite{KMT17}.
There, we show the homogenization of the variational principle
of graph homomorphisms to~$\mathbb{Z}$.
In homogenization, homomorphisms are not chosen according to the uniform measure
but instead certain heights are preferred or penalized according to a random
field.
Mathematically the height function is sampled from a Gibbs measure
with respect to a randomized Hamiltonian.
The limit shape may change drastically;
for example, when the random field is unbounded,
simulations show the formation of terraces.
See Figures~\ref{f_homog_2d} and~\ref{f_homog_3d} for examples.
We hope that the method outlined in this article
can serve as a guiding principle for deducing
the variational principle and related results
for more complex models.
\\

\begin{figure}[ht]
	\centering
	\begin{subfigure}{0.23\textwidth}
		\includegraphics[width=1.414214\textwidth]{random_field_constant.jpg}
		\caption{Constant field~$\omega$.}\label{f_aztec_constant}
	\end{subfigure} \mbox{ \color{white} $\quad$ test }
	\begin{subfigure}{0.23\textwidth}
		\includegraphics[width=\textwidth, angle=45]{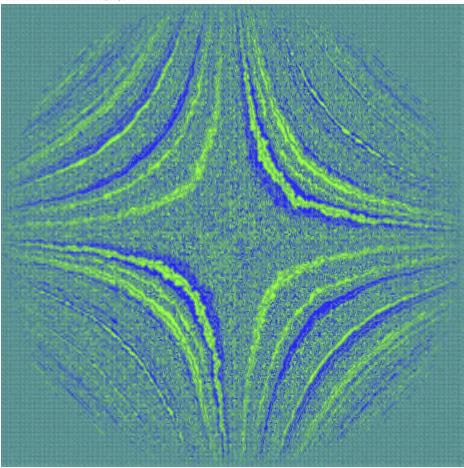}
		\caption{Bounded random field~$\omega$.}\label{f_aztec_uniform}
	\end{subfigure} \mbox{ \color{white} $\quad$ test}
	\begin{subfigure}{0.23\textwidth}
		\includegraphics[width=\textwidth, angle=45]{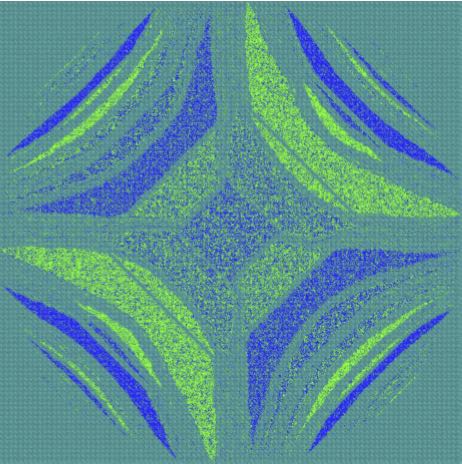}
		\caption{Unbounded random field~$\omega$.} \label{f_aztec_normal}
	\end{subfigure}
	\caption{Quenched Aztec diamonds, 2d perspective.}
	\label{f_homog_2d}
\end{figure}

\begin{figure}[ht]
    \centering
    \begin{subfigure}{0.4\textwidth}
        \includegraphics[width=\textwidth, angle=0]{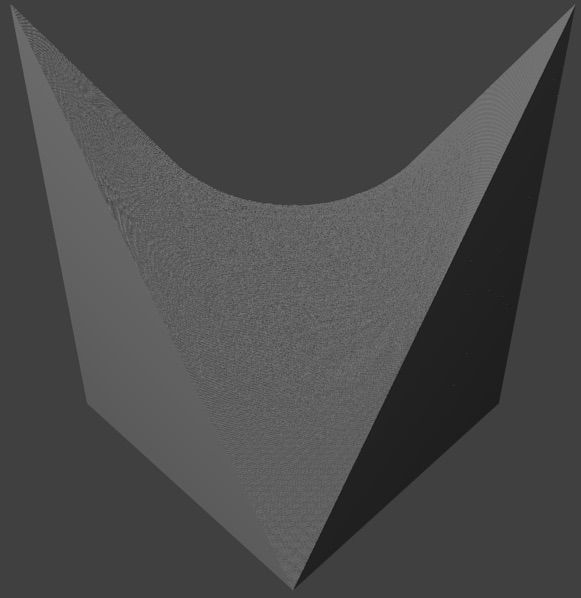}
\caption{Constant field~$\omega$, the same limit-shape as in Figure~\ref{f_aztec_constant}.}\label{f_aztec_constant_3d}
    \end{subfigure} \mbox{ \color{white} $\quad$ test }
\begin{subfigure}{0.4\textwidth}
        \includegraphics[width=\textwidth, angle=0]{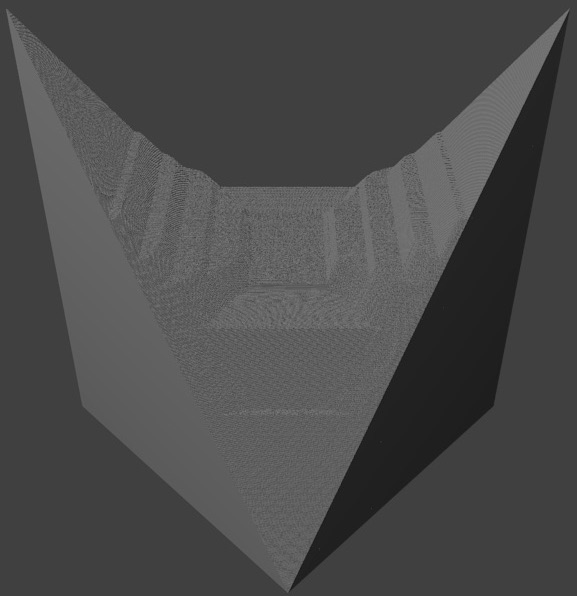}
 \caption{Unbounded random field~$\omega$, the same limit shape as in Figure~\ref{f_aztec_normal}.} \label{f_aztec_normal_3d}
 \end{subfigure}
\caption{Quenched Aztec diamonds, 3d perspective.}
\label{f_homog_3d}
 \end{figure}

As hinted above, the three main results of this article are
the profile theorem (Theorem~\ref{th_profile}),
the variational principle (Theorem~\ref{th_varnprin}),
and the large deviations principle (Theorem~\ref{th_ldp}).
The majority of the effort in this article
goes into proving the profile theorem.
The proof starts by proving the profile theorem in a special case
(where the domain $R$ is the union of simplices and the limiting profile $h_R$
is piecewise affine),
then bootstraps this result to the general case.
The main idea of each step in this proof is clear,
although some care is needed to account for all the details.
\\

We call attention to two ingredients in the proof.
The first ingredient is the simplicial Rademacher theorem,
so called because it approximates a Lipschitz function $h_R$
uniformly over a large portion of its domain
by a piecewise affine approximation $h_K$,
where the ``pieces'' on which $h_K$ is affine are simplices.
This approximation gives control over both the direct error $|h_K - h_R|$
and over the error in the derivatives $|\nabla h_K - \nabla h_R|$.
Compared to the classical Rademacher theorem
which states that $h_R$ is almost everywhere differentiable,
the simplicial Rademacher theorem is a surprisingly strong approximation result.
This seems like a standard result
but the authors have not found it stated in this form
in the random surfaces literature,
so we give details of the proof.
\\

In order to exploit the simplicial Rademacher theorem,
we need robustness of both the macroscopic and microscopic entropy,
under changes both to the limiting profile and to the domain.
Robustness of the macroscopic entropy follows from elementary analysis,
because the simplicial Rademacher approximation
has derivative $\nabla h_K$ close to $\nabla h_R$.
Robustness of the microscopic entropy rests largely upon the second ingredient
that we call attention to: a Kirszbraun theorem for graph homomorphisms
(see Theorem~\ref{th_kirszbraun}).
This theorem gives conditions under which a height function
$h_{R_n}: R_n \to \Z$ may be extended to a larger domain
$\tilde R_n \supseteq R_n$.
We expect that the main challenge in extending the method of this article
to other models will be proving a comparable extension theorem.
\\

The profile theorem can be used to prove the variational principle
and the large deviations principle.
Both proofs are similar, and rely on the local compactness
of the space of Lipschitz functions.
Moreover the two proofs are robust;
once the profile theorem is proven for a model
(with suitable macroscopic state space),
the variational principle and large deviations principle follow automatically.
We present the proof of the variational principle first and in greater detail.
For the large deviations principle we highlight the differences in proof,
and we also change notation (replacing symbols like $\Ent_n$ and $\Ent$),
in order to match the conventions of large deviations theory.
\\

\addtocontents{toc}{\protect\setcounter{tocdepth}{-1}}

\section*{Overview of remaining article}

In Section~\ref{s_setting}, we define the setting and formulate the main results of this article.
In Section~\ref{s_outline} we explain the main idea and the structure of the proofs of the main results. The details are then given in Sections~\ref{s_ent} through~\ref{s_ldp}.

\section*{Notation}

\begin{itemize}
\item $x$ and $y$ usually denote points in $\R^m$.
\item $z$ usually denotes a points in $\Z^m$.
\item For $x\in \R$, $\floor x$ denotes the largest integer $\le x$,
	and $\ceil x$ denotes the smallest integer $\ge x$.
\item For $x = (x_1, \dotsc, x_m) \in \R^m$,
	$\floor x := (\floor x_1, \dotsc, \floor x_m)$.
\item Given a set $A$ in some topological space,
	$\interior{A}$ and $\closure{A}$ denote the interior and closure of $A$
	respectively.
\item $R$ and $R_n$ are ``nice'' domains in $\R^m$ and $\Z^m$ respectively
	(see Assumption~\ref{a_domain} below).
\item $h_R: R \to \R$ is an asymptotic height function
	(i.e.~$1$-Lipschitz function).
\item $h_{R_n}: R_n \to \mathbb{Z}$ is a height function (i.e.~graph homomorphism).

\item $\theta(\delta)$ denotes a function
	with $\lim_{\delta \downarrow 0} \theta(\delta) = 0$.
\item $\theta_{a,b,c}(\delta)$ denotes a function
	with $\lim_{\delta \downarrow 0} \theta_{a,b,c}(\delta) = 0$,
	with rate of convergence depending only on the parameters $a,b,c$.
\item For a set~$A$ we denote with~$|A|$ either the cardinality of~$A$
	or the Lebesgue measure of~$A$.

\end{itemize}

\addtocontents{toc}{\protect\setcounter{tocdepth}{1}}

\section{Setting and main results} \label{s_setting}

In this section we formally describe the model under study,
and state the main results that we prove.
In describing the model we err on the side of verbosity and explicitness.
Some of the notations used are non-standard
(such as the $\theta$-notation for asymptotics
described in Section~\ref{ss_asymp}),
but these notations allow for relatively concise and (more importantly) precise
statements of the results and proofs to follow.

In Section~\ref{ss_givens} we will carefully introduce the basic model,
i.e.\@ height functions on ``nice'' subsets of $\Z^m$.
In Section~\ref{ss_canon_hf} we describe a canonical family of height functions.
In Section~\ref{ss_ent} we use these canonical height functions
to define the microscopic entropy,
then we go on to define the macroscopic entropy and surface tension.
In Section~\ref{ss_asymp} we introduce our asymptotic notation,
as mentioned above.
In Section~\ref{ss_main_results}, the main results of this article are stated.

\subsection{Objects of study} \label{ss_givens}

Given a graph $\Gamma = (V,E)$, we write ``$v_1 \sim v_2$'' if two vertices
$v_1, v_2 \in V$ are adjacent; i.e. if $\{v_1,v_2\} \in E$.
Given two graphs $\Gamma_1 = (V_1, E_1)$ and $\Gamma_2 = (V_2, E_2)$,
we recall that a \defn{graph homomorphism} is a function $\varphi: V_1 \to V_2$
such that whenever vertices $v, v' \in V_1$ are adjacent in $\Gamma_1$,
their images $\varphi(v), \varphi(v') \in V_2$ are adjacent in $\Gamma_2$.

In this article, we specialize to the case of graph homomorphisms from
certain subgraphs $R_n \subset \Z^m$ to a subgraph of $\Z$.
For $R_n \subset \Z^m$,
we write $R_n^c$ for the complement of $R_n$ in $\Z^m$,
and $\partial R_n := \{ z \in R_n \, | \, \exists z' \in R_n^c, \, z \sim z' \}$
for the (inner) boundary of $R_n$.

\begin{definition}[Height functions] \label{d_ht_func}
A \defn{height function} on $R_n$
is a graph homomorphism $h_{R_n}: R_n \to \Z$
that preserves parity, meaning that for $z = (z_1, \dotsc, z_m) \in \Z^m$,
\begin{equation} \label{e_ht_func_parity}
	h_{R_n}(z) \bmod 2 = z \bmod 2
	\enspace \biggl( := \biggl[ \sum_{i=1}^m z_i \biggr] \bmod 2 \biggr) \,.
\end{equation}

We call a height function $h_{\partial R_n}: \partial R_n \to \Z$,
defined on the boundary of $R_n$, a \defn{boundary height function}.
\end{definition}

We are interested in sequences of subgraphs $\{R_n: n \in \N\}$
that converge under a scaling limit to a ``nice'' region $R \subset \R^m$.
More specifically, we make the following assumptions.

\begin{assumption}[Assumptions on domains $R$ and $R_n$] \label{a_domain}
We assume that $R \subset \R^m$ is compact and connected
and that $R$ is the closure of its interior
(sets with the latter property are called \defn{regular closed sets};
see e.g.~\cite{StSe95}).

We assume that $R_n \subset \Z^m$
and $\partial R_n \subset R$ are connected as subgraphs of $\Z^m$,
and (for simplicity) we assume that $\tfrac{1}{n} R_n \subset R$.

We require that $\frac{1}{n} R_n \to R$ in the Hausdorff metric;
that is, the metric on $\mathcal{P}(\R^m) := \{A \subset \R^m\}$
defined by
\begin{equation} \label{e_hausdorff}
	d_H(A, B)
	:= \biggl( \: \sup_{x \in A} \inf_{y \in B} |x-y|_1 \biggr)
	\vee \biggl( \: \sup_{y \in B} \inf_{x \in A} |x-y|_1 \biggr) \,.
\end{equation}
\end{assumption}

\begin{remark}[On the choice of norm in~\eqref{e_hausdorff}]
\label{rem_hausdorff}
By equivalence of norms,
it does not matter which norm on $\R^m$ is used in~\eqref{e_hausdorff}.
Later, we will be interested primarily in the $\ell^1$ norm.
This is because the $\ell^1$ norm is the scaling limit of the graph distance
on $\Z^m$.
More precisely, if $x, x' \in \R^m$ and if $z_n, z_n' \in \Z^m$
satisfy $|\tfrac{1}{n} z_n - x|_1 < \tfrac{m}{n}$
and $|\tfrac{1}{n} z_n' - x'|_1 < \tfrac{m}{n}$,
then $\tfrac{1}{n} d_{\Z^m}(z_n, z_n') \to |x-x'|_1$,
where $d_{\Z^m}$ denotes the graph distance.
\end{remark}

For example, when $R$ is compact, convex polytope,
such as a hypercube or a simplex,
the sets $R_n := \{ z \in \Z^m \,|\, \tfrac{1}{n}z \in R \}$
satisfy Assumption~\ref{a_domain}.

\medskip

Just as the microscopic domains $R_n$ have a scaling limit,
so do the microscopic height functions $h_{R_n}: R_n \to \Z$.

\begin{definition}[Asymptotic height functions] \label{d_asymp_ht_func}
We call a function $h_R: R \to \R$ an \defn{asymptotic height function}
if $h_R$ is Lipschitz with Lipschitz constant at most $1$,
with respect to the $\ell^1$-norm on $\R^m$;
that is, if
\begin{equation} \label{e_asymp_ht_func}
	\operatorname{Lip}(h_R)
	:= \sup_{x \ne y \in R} \frac{|h_R(x) - h_R(y)|}{|x-y|_1}
	\le 1 \,.
\end{equation}
Likewise, if $h_{\partial R}: \partial R \to \R$ is $1$-Lipschitz
(with respect to the $\ell^1$-norm),
we call $h_{\partial R}$ an \defn{asymptotic boundary height function}.
\end{definition}

We assume that $h_{\partial R_n}: \partial R_n \to \Z$
are boundary height functions that converge (after rescaling) to an asymptotic
boundary height function $h_{\partial R}: \partial R \to \R$ in following sense:
for each $n$, let $d_n = d_H(\tfrac{1}{n}R_n, R)$.
Then, we say $\tfrac{1}{n} h_{\partial R_n} \to h_{\partial R}$ if and only if
\begin{equation} \label{e_bndy_fn_lim}
	\lim_{n \to \infty} \: \sup_{z \in \partial R_n} \:
		\sup_{\substack{
			x \in \partial R \\
			|x - \frac{1}{n} z|_1 \le d_n
		}} \:
			\biggl| \, \frac{1}{n} h_{\partial R_n}(z)
			- h_{\partial R}(x) \, \biggr| = 0 \,.
\end{equation}

\medskip

Now, we define a few families of height functions
and asymptotic height functions.
The sets of height functions from Definition~\ref{d_ht_func_sets} below
appear frequently in entropies,
and the sets of asymptotic height functions from
Definition~\ref{d_asymp_ht_func_sets} are important for the statement of the
variational principle (Theorem~\ref{th_varnprin}).

\begin{definition}[Sets of height functions] \label{d_ht_func_sets}
Let $R_n$ be a microscopic domain as above,
let $h_{R_n}: R_n \to \Z$ be a boundary height function,
and let $\delta > 0$.
We define:
\begin{align}
	M(R_n)
	&:= \bigl\{ h_{R_n}: R_n \to \R
	\, \bigl| \, \text{$h_{R_n}$ is a height function} \bigr\} \\
	M(R_n, h_{\partial R_n})
	&:= \bigl\{ h_{R_n} \in M(R_n)
	\, \big| \, h_{R_n}|_{\partial R_n} = h_{\partial R_n} \bigr\} \\
	M(R_n, h_{\partial R_n}, \delta)
	&:= \bigl\{ h_{R_n} \in M(R_n)
	\, \big| \, \sup_{z \in \partial R_n}
		|h_{R_n}(z) - h_{\partial R_n}(z)| < \delta n \bigr\} \,. \\
	B(R_n, h_R, \delta)
	&:= \bigl\{ h_{R_n} \in M(R_n)
	\, \big| \, \sup_{z \in R_n}
		|h_R(\tfrac{1}{n} z) - \tfrac{1}{n} h_{R_n}(z)| < \delta
	\bigr\} \,.
\end{align}

In the last definition, the expression ``$h_R(\tfrac{1}{n}z)$'' makes sense
because of the assumption that $\tfrac{1}{n}R_n \subset R$
in Assumption~\ref{a_domain}.
\end{definition}

\begin{definition}[Sets of asymptotic height functions]
\label{d_asymp_ht_func_sets}
Let $R \subset \R^m$ be a domain satisfying Assumption~\ref{a_domain},
let $h_{\partial R}: \partial R \to \R$
be an asymptotic boundary height function,
and let $\delta > 0$.
We define:
\begin{align}
	M(R)
	&:= \bigl\{ h_R: R \to \R \, \big| \,
		\text{$h_R$ is an asymptotic height function} \bigr\} \\
	M(R, h_{\partial R})
	&:= \bigl\{ h_R: R \to \R \, \big| \,
		h_R|_{\partial R} = h_{\partial R} \bigr\} \\
	M(R, h_{\partial R}, \delta)
	&:= \bigl\{ h_R: R \to \R \, \big| \,
		\forall x \in \partial R \,,\,
			|h_R(x) - h_{\partial R}(x)| \le \delta
	\bigr\} \,.
\end{align}
\end{definition}~

\subsection{Affine height functions}
\label{ss_canon_hf}

Affine height functions play an important role
in defining and studying the entropy of our model.
For an asymptotic height function $h_R: R \to \R$,
we mean by ``affine'' the usual property:
there exist $s \in [-1,1]^m$ and $b \in \R$ such that $h_R(x) = s \cdot x + b$.
The bounds on $s$ ensure that $h_R$
satisfies the Lipschitz property~\eqref{e_asymp_ht_func},
so all such functions are indeed asymptotic height functions
as per Definition~\ref{d_asymp_ht_func}.

On the microscopic domains $R_n$,
we consider best-possible approximations to affine functions.
Fix $s \in [-1,1]^m$ and $b \in \R$.
At a lattice point $z \in \Z^m$, we define
$h^{s \cdot x + b}_{R_n}(z)$ to be $s \cdot z + b$,
rounded to the nearest integer of correct parity
(see Figure~\ref{f_lin_ht_func}).
In the rest of this subsection, we formalize this definition,
verify that it actually does define a height function,
and check that it is consistent.

Let us introduce an auxiliary notation that is used only in this subsection.
Given a point $z = (z_1, \dotsc, z_m) \in \Z^m$,
we say $z$ has even or odd parity
as $(\sum_{i=1}^m z_i) \in \Z$ has even or odd parity respectively,
and we write $z \bmod 2$ for the parity of $z$.

Given $z \in \Z^m$ and $y \in \R$,
we write $[y]_{z \bmod 2}$ for the closest integer to $y$
that has parity $z \bmod 2$.
In case of a tie, i.e.\@ if $y$ is an integer that has opposite parity to $z$,
we arbitrarily choose to ``round up'' and set $[y]_{z \bmod 2} = y+1 \in \Z$.

For example, let $z = (1,2,3) \in \Z^3$ and $z' = (4,-6,7)$.
Then $z$ is an even point and $z'$ is an odd point.
So:
\begin{align}
		[5.4]_{z \bmod 2} &= 6 \,,
		&[-3]_{z \bmod 2} &= -2 \,, \\{}
		[5.4]_{z' \bmod 2} &= 5 \,,
		&[-3]_{z' \bmod 2} &= -3 \,. \\
\end{align}

Now, given $s \in [-1,1]^m$ and $b \in \R$,
we define the affine height functions $h^{s \cdot x + b}_{R_n}$ by
\begin{equation} \label{e_canon_hf_def}
	h_{R_n}^{s \cdot x + b}(z) := [s \cdot z + b]_{z \bmod 2} \, .
\end{equation}

Note that the symbol $x$ in the superscript of $h_{R_n}^{s \cdot x + b}$
is merely formal; ``$s \cdot x + b$'' should be read as
``the function mapping $x$ to $s \cdot x + b$''.
Moreover, the choice of domain $R_n$ in the subscript does not affect
the values of $h_{R_n}^{s \cdot x + b}$ at any point;
for any sets $A_n, B_n \subseteq \Z^m$ and any point $z \in A_n \cap B_n$,
one has $h_{A_n}^{s \cdot x + b}(z) = h_{B_n}^{s \cdot x + b}(z)$.
An example of a function $h_{R_n}^{s \cdot x + b}$
is provided in Figure~\ref{f_lin_ht_func}.

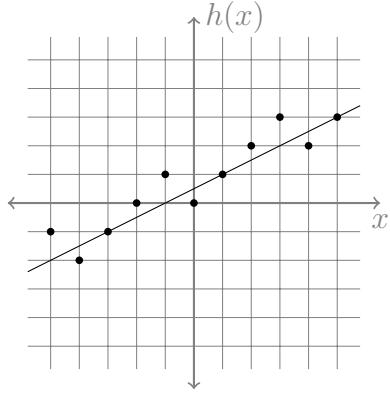
\begin{figure}
	\centering
	\begin{tikzpicture}[
		x=0.15in,y=0.15in,
		every circle/.style={radius=0.02in},
	]
		\draw[help lines] (-5.8,-5.8) grid[step=1] (5.8,5.8);
		\draw[<->,help lines,thick] (-6.5,0) -- (6.5,0)
			node [anchor=north] {$x$};
		\draw[<->,help lines,thick] (0,-6.5) -- (0,6.5)
			node [anchor=west] {$h(x)$};
		\path[draw] (-5.8,-2.4) -- (5.8, 3.4);
		\path[fill] (-5,-1) circle;
		\path[fill] (-4,-2) circle;
		\path[fill] (-3,-1) circle;
		\path[fill] (-2,0) circle;
		\path[fill] (-1,1) circle;
		\path[fill] (0,0) circle;
		\path[fill] (1,1) circle;
		\path[fill] (2,2) circle;
		\path[fill] (3,3) circle;
		\path[fill] (4,2) circle;
		\path[fill] (5,3) circle;
	\end{tikzpicture}
	\caption{An affine height function $h_{R_n}^{s \cdot x + b}$
	and the corresponding continuous affine function
	$x \mapsto s \cdot x + b$.
	Here $s = \tfrac{1}{2}$ and $b = \tfrac{1}{2}$.}
	\label{f_lin_ht_func}
\end{figure}

From the definition above, it is not clear that $h_{R_n}^{s \cdot x + b}$
are height functions.
This is the content of Lemma~\ref{l_lin_ht_func}.

\begin{lemma} \label{l_lin_ht_func}
Let $s \in [-1,1]^m$ and $b \in \R$.
For any adjacent points $z \sim z' \in \Z^m$,
the values $h_{R_n}^{s \cdot x + b}(z)$
and $h_{R_n}^{s \cdot x + b}(z')$ differ by exactly $1$.
\end{lemma}

\begin{proof}
From the definition of $h_{R_n}^{s \cdot x + b}$,
we note two inequalities:
\begin{equation} \label{e_lin_ht_func_z}
	|h_{R_n}^{s \cdot x}(z) - (s \cdot z + b)| \le 1 \,,
\end{equation}
and
\begin{equation} \label{e_lin_ht_func_zp}
	|h_{R_n}^{s \cdot x + b}(z') - (s \cdot z' + b)| \le 1 \,.
\end{equation}
Additionally, since $s \in [-1,1]^m$, we have
\begin{equation} \label{e_lin_ht_func_dot}
|(s \cdot z + b) - (s \cdot z' + b)| \le 1 \,.
\end{equation}

By the triangle inequality,
$|h_{R_n}^{s \cdot x + b}(z) - h_{R_n}^{s \cdot x + b}(z')| \le 3$.
We shall show that equality cannot hold.
Since the difference $h_{R_n}^{s \cdot x + b}(z) - h_{R_n}^{s \cdot x + b}(z')$
is obviously an odd integer, it will follow that the difference is $\pm 1$.

Suppose towards a contradiction that
\[
	\bigl| h_{R_n}^{s \cdot x + b}(z) - h_{R_n}^{s \cdot x + b}(z') \bigr|
	= 3 \,.
\]
Then~\eqref{e_lin_ht_func_z} and~\eqref{e_lin_ht_func_zp} must be equalities.
From the definition of $[\cdot]_{z \bmod 2}$,
necessarily then $s \cdot z + b$ is an integer
with parity opposite that of $z$,
and so
\[
	h_{R_n}^{s \cdot x + b}(z) = (s \cdot z + b) + 1 \,.
\]

Likewise
\[
	h_{R_n}^{s \cdot x + b}(z') = (s \cdot z' + b) + 1 \,.
\]

But then
\[
	\bigl| h_{R_n}^{s \cdot x + b}(z) - h_{R_n}^{s \cdot x + b}(z') \bigr|
	= \bigl| (s \cdot z + b + 1) - (s \cdot z' + b + 1) \bigr|
	\le 1 \,.
\]
\end{proof}

We end this section with the following lemma.
The conclusion~\eqref{e_canon_hf_ineq_conc} is exactly what is needed later
to apply the Kirszbraun theorem (Theorem~\ref{th_kirszbraun}):

\begin{lemma}[Inequality for $h_{R_n}^{s \cdot x + b}$]
\label{lem_canon_hf_ineq}
Let $s, s' \in [-1,1]^m$, $b, b' \in \R$, and $z, z' \in \Z^m$.
If
\begin{equation} \label{e_canon_hf_ineq_hypo}
	\bigl| (s \cdot z + b) - (s' \cdot z' + b') \bigr|
	\le |z - z'|_1 \,,
\end{equation}
then
\begin{equation} \label{e_canon_hf_ineq_conc}
	\bigl| h_{\{z\}}^{s \cdot x + b}(z)
		- h_{\{z'\}}^{s' \cdot x + b'}(z') \bigr|
	\le |z - z'|_1 \,.
\end{equation}
\end{lemma}

\begin{proof}
The proof is similar to that of Lemma~\ref{l_lin_ht_func}.
By the triangle inequality and~\eqref{e_canon_hf_ineq_hypo},
\begin{equation} \label{e_canon_hf_ineq_tri} \begin{aligned}
	\hskip3em&\hskip-3em
	\bigl| h_{\{z\}}^{s \cdot x + b}(z)
		- h_{\{z'\}}^{s' \cdot x + b'}(z') \bigr| \\
	&\le \bigl| h_{\{z\}}^{s \cdot x + b}(z) - (s \cdot z + b) \bigr|
	+ \bigl| (s \cdot z + b) - (s' \cdot z' + b') \bigr| \\
	& \qquad
	+ \bigl| (s' \cdot z' + b') -h_{\{z'\}}^{s' \cdot x + b'}(z') \bigr| \\
	&\le |z - z'|_1 + 2 \,.
\end{aligned} \end{equation}

Since $h_{\{z\}}^{s \cdot x + b}(z)$, $h_{\{z'\}}^{s' \cdot x + b'}(z')$
and $|z - z'|_1$ are all integers,
\begin{equation} \label{e_canon_hf_ineq_set}
	\bigl| h_{\{z\}}^{s \cdot x + b}(z)
		- h_{\{z'\}}^{s' \cdot x + b'}(z') \bigr|
	- |z - z'|_1
	\in \{\dotsc, -2, -1, 0, 1, 2\} \,.
\end{equation}

We want to prove that the left-hand side of~\eqref{e_canon_hf_ineq_set}
is $\le 0$.
By parity considerations it must be even,
and we need only prove it is $\ne 2$.
Assume for a contradiction that the left-hand side
of~\eqref{e_canon_hf_ineq_set} equals $2$.
Then equality holds in~\eqref{e_canon_hf_ineq_tri},
and in particular
\[
	\bigl| h_{\{z\}}^{s \cdot x + b}(z) - (s \cdot z + b) \bigr| = 1
	\quad \text{and} \quad
	\bigl| h_{\{z'\}}^{s' \cdot x + b'}(z') - (s' \cdot z' + b') \bigr| = 1
	\,.
\]

As in the proof of Lemma~\ref{l_lin_ht_func}, this implies that
\[
	h_{\{z\}}^{s \cdot x + b}(z) = (s \cdot z + b) + 1
	\quad \text{and} \quad
	h_{\{z'\}}^{s' \cdot x + b'}(z') = (s' \cdot z' + b') + 1 \,.
\]

Therefore
\[
	\bigl| h_{\{z\}}^{s \cdot x + b}(z)
		- h_{\{z'\}}^{s' \cdot x + b'}(z') \bigr|
	= \bigl| (s \cdot z + b) - (s' \cdot z' + b') \bigr|
	\le |z - z'|_1 \,.
\]

This is the desired contradiction, which completes the proof.
\end{proof}

\subsection{Entropies and surface tensions} \label{ss_ent}

In this section we make three more definitions
needed for our statement of the main results.
First, we define the microscopic entropy of a set of height functions.
More precisely, this is the Shannon entropy of the uniform distribution
over a finite set of height functions,
normalized by the size of their common domain,
and negated.
(The negative convention is chosen so that the surface tension $\ent(s)$,
defined later, is convex rather than concave.)
The microscopic entropy is essentially the same as
the specific free energy of~\cite{She05}.

\begin{definition}[Microscopic entropy]
Given a finite, non-empty set of height functions $A \subset M(R_n)$,
we define the \defn{microscopic entropy}
\begin{equation} \label{e_micro_ent}
	\Ent_{R_n}(A) := - \frac{1}{|R_n|} \ln |A| \,.
\end{equation}
\end{definition}

We observe that the microscopic entropy is translation invariant:

\begin{observation}[Translation invariance] \label{obs_trans_invar}
Let $h_{\partial R_n} \in M(\partial R_n)$, $h_R \in M(R)$,
$\delta  > 0$, and $c \in \R$.
Then:
\begin{equation} \begin{aligned}
	\Ent_{R_n} \bigl( M(R_n, h_{\partial R_n} + c) \bigr)
	&= \Ent_{R_n} \bigl( M(R_n, h_{\partial R_n}) \bigr) \,, \\
	\Ent_{R_n} \bigl( M(R_n, h_{\partial R_n} + c, \delta) \bigr)
	&= \Ent_{R_n} \bigl( M(R_n, h_{\partial R_n}, \delta) \bigr)
	\,, \text{and} \\
	\Ent_{R_n} \bigl( B(R_n, h_R + c, \delta) \bigr)
	&= \Ent_{R_n} \bigl( B(R_n, h_R, \delta) \bigr) \,,
\end{aligned} \end{equation}
where $h_{R_n} + c$ is the height function defined by
$(h_{R_n} + c)(z) = h_{R_n}(z) + c$ for $z \in R_n$,
and likewise $(h_R + c)(x) = h_R(x) + c$ for $x \in R$.
\end{observation}

All of these sets, except for $M(R_n)$, are finite,
because of the constraints they impose on $h_{R_n}$
and the Lipschitz property of $h_{R_n}$. In fact, we can say more.
Let us count $M(R_n, h_{\partial R_n})$,
for some boundary height function $h_{\partial R_n} \in M(\partial R_n)$.
The values of $h_R \in M(R_n, h_{\partial R_n})$ are fixed on $\partial R_n$,
and for each of the $\le |R_n|$ points $x$ in the interior of $R_n$,
there are at most $2$ admissible values for $h_R(x)$.
Therefore $|M(R_n, h_{\partial R_n})| \le 2^{|R_n|}$.
Similar logic holds for $M(R_n, h_\partial R_n, \delta)$
and $B(R_n, h_R, \delta)$. This leads to the following observation:

\begin{observation}[Boundedness of $\Ent$] \label{obs_ent_bdd}
Let $h_{\partial R_n} \in M(\partial R_n)$, $h_R \in M(R)$,
and $\delta > 0$. Then:
\begin{equation} \begin{aligned} \label{e_ent_bdd}
	\Ent_{R_n} \bigl( M(R_n, h_{\partial R_n}) \bigr)
	&\ge -\ln 2 \,, \\
	\Ent_{R_n} \bigl( M(R_n, h_{\partial R_n}, \delta) \bigr)
	&\ge - \ln 2 - \frac{\ln \lceil 2 \delta n \rceil}{|R_n|} \\
	&= - \ln 2 - \delta \, O(\tfrac{1}{n^{m-1}}) \,, \text{and} \\
	\Ent_{R_n} \bigl( B(R_n, h_R, \delta) \bigr)
	&\ge -\ln 2 - \frac{\ln \lceil 2 \delta n \rceil}{|R_n|} \\
	&= - \ln 2 - \delta \, O(\tfrac{1}{n^{m-1}}) \,.
\end{aligned} \end{equation}
\end{observation}

\medskip

Next, we define the local surface tension.
There are in general many equivalent definitions of surface tension
(see for example Chapter~6 of~\cite{She05}).
The following definition is easiest to work with for our purposes.

\begin{definition}[Local surface tension] \label{d_loc_surf_tens}
For $s \in [-1,1]^m$, the \defn{local surface tension} $\ent(s)$
is defined to be the limit
\begin{equation} \label{e_loc_surf_tens}
	\ent(s) := \lim_{n \to \infty} \ent_n(s) \,,
\end{equation}
where $\ent_n(s)$ is defined as
\begin{equation} \label{e_entn}
	\ent_n(s)
	:= \Ent_{Q_n}\bigl( M(Q_n, h_{\partial Q_n}^{s \cdot x + 0}) \bigr)
\end{equation}
and where $Q_n = [0,n)^m \cap \Z^m$
is the discrete hypercube of side length $n$.
\end{definition}

The limit~\eqref{e_loc_surf_tens} exists by standard subadditivity arguments;
we refer the interested reader to e.g.~\cite{Durrett}.
In fact, by translation invariance (see Observation~\ref{obs_trans_invar}),
we may replace $h_{\partial Q_n}^{s \cdot x + 0}$
by $h_{\partial Q_n}^{s \cdot x + b}$ in~\eqref{e_entn}, for any $b \in \R$.
Additionally, boundedness passes through the limit in~\eqref{e_loc_surf_tens}.
Therefore:

\begin{observation}[Boundedness of $\ent(s)$] \label{obs_surf_tens_bdd}
For any $s \in [-1,1]^m$,
\[
	-\ln 2 \le \ent(s) \le 0 \,.
\]
\end{observation}

Let us now define the macroscopic entropy.

\begin{definition}[Macroscopic entropy] \label{d_macro_ent}
Let $h_R: R \to \R$ denote an asymptotic height function .
The \defn{macroscopic entropy} $\Ent_R(h_R)$ is defined as
\begin{equation} \label{e_macro_ent}
	\Ent_R(h_R) := \int_R \ent(\nabla h(x)) \, dx \,.
\end{equation}
\end{definition}

\subsection{Asymptotic notation} \label{ss_asymp}

In this section we introduce a notation for asymptotic error.
Compared to the Landau big-$O$ notation,
our $\theta$-notation abstracts away
the rate of convergence of the error,
but makes explicit the dependence on parameters.
For this purpose we write $\theta_{\alpha}(\delta)$
for a family of unspecified functions,
parameterized by a symbol $\alpha$,
such that $\theta_\alpha(\delta) \to 0$
at a rate depending on the value of the parameter $\alpha$.
That is, for any $\varepsilon > 0$ and any admissible parameter value $\alpha$,
there exists $\delta_0 = \delta_0(\alpha) > 0$
such that $0 <\delta < \delta_0$ implies $\theta_\alpha(\delta) < \varepsilon$.

Extending the above notation, we frequently replace the single parameter
$\alpha$ by a list of parameters $\alpha, \beta, \gamma, \dotsc$.
For example, we might write an identity like
\begin{equation} \label{e_asymp_example} \begin{aligned}
	\min_{h_R \in M(R, h_{\partial R})} \Ent_R(h_R)
	&= \Ent_{R_n} \bigl( M(R_n, h_{\partial R_n}, \delta) \bigr) \\
	&\qquad + \theta_{m, R, h_{\partial R}, R_n, h_{\partial R_n}}(
		\delta) \\
	&\qquad+ \theta_{m, R, h_{\partial R}, R_n, h_{\partial R_n}, \delta}(
		\tfrac{1}{n}) \,.
\end{aligned} \end{equation}

The identity states that the two entropy terms on the first line
differ by a small amount;
the difference vanishes as $\delta$ and $\tfrac{1}{n}$ go to zero,
and the rate of convergence depends on several parameters.
The ``$\theta(\delta)$'' term depends on the parameters from the setting,
namely the ambient dimension $m$, the region $R$,
the height function $h_R$ of interest,
and the corresponding discrete objects $R_n$ and $h_{R_n}$.
The ``$\theta(\tfrac{1}{n})$'' term depends on these parameters
along with the value of $\delta$.
We find that listing out the setting parameters
$m$, $R$, $h_R$, $R_n$, and $h_{R_n}$ makes the expression harder to read.
So for the rest of the article we suppress these parameters from the subscripts
of $\theta$ terms.
Under this convention~\eqref{e_asymp_example} becomes:
\[ \begin{aligned}
	\min_{h_R \in M(R, h_{\partial R})} \Ent_R(h_R)
	&= \Ent_{R_n} \bigl( M(R_n, h_{\partial R_n}, \delta) \bigr) \\
	&\qquad + \theta(\delta)
	+ \theta_\delta(\tfrac{1}{n}) \,.
\end{aligned} \]

As mentioned above, the advantage of our $\theta$ notation is
that it abstracts away the exact rates of convergence,
but leaves explicit the dependencies between parameters.
For example, suppose we want to make the error in approximation
in \eqref{e_asymp_example} to be less than $\varepsilon$.
We should first choose $\delta$
so that (say) $\theta(\delta) < \tfrac{1}{2} \varepsilon$,
then choose $n$ depending on $\delta$
(and on the suppressed parameters $m$, $R$, etc.)
so that $\theta_\delta(\tfrac{1}{n}) < \tfrac{1}{2} \varepsilon$.

\subsection{Main results} \label{ss_main_results}

The main results of this article are
the profile theorem, the variational principle,
and the large deviations principle:

\begin{theorem}[Profile theorem] \label{th_profile}
Under the setting explained in Section~\ref{s_setting},
for any $h_R \in M(R)$, $\delta > 0$, and $n \in \N$,
\begin{align}\label{e_profile_theorem}
    	\Ent_R(h_R)
	= \Ent_{R_n} \bigl( B(R_n, h_R, \delta) \bigr)
		+ \theta_{h_R}(\delta) + \theta_{h_R,\delta}(\tfrac{1}{n}) \,.
\end{align}
\end{theorem}

The second main result is the variational principle:

\begin{theorem}[Variational principle] \label{th_varnprin}
Under the setting explained in Section~\ref{s_setting},
\begin{equation} \label{e_varnprin}
	\inf_{h_R \in M(R, h_{\partial R})} \Ent_R(h_R)
	= \Ent_{R_n} \bigl( M(R_n, h_{\partial R_n}, \delta) \bigr)
		+ \theta(\delta) + \theta_\delta(\tfrac{1}{n})
\end{equation}
for any $\delta > 0$ and $n \in \N$.
\end{theorem}

Finally, we prove a large deviations principle for the model.
We adopt the conventions of large deviations
(see for example \cite{DeZe09,RaSe15}).

\begin{theorem}[Large deviations principle] \label{th_ldp}
Consider the Polish space $M(R)$ of asymptotic height functions
(i.e.\@ Lipschitz functions with Lipschitz constant $ \le 1$),
endowed with the topology of uniform convergence (induced by the supremum norm).

For $\delta > 0$ and $n \in \N$,
define a probability measure $\mu_{\delta, n}$ on $M(R)$ by
\[
	\mu_{\delta,n}(A)
	:= \frac{1}{|M(R_n, h_{\partial R_n}, \delta)|}
		\, \bigl| \bigl\{ h_{R_n} \in M(R_n, h_{\partial R_n}, \delta)
			\,\big|\, \tilde h_{R_n} \in A \bigr\} \bigr| \,,
\]
where $\tilde h_{R_n}$ is the Lipschitz function given by rescaling
and interpolating $h_{R_n}$ so as to make it an asymptotic height function,
i.e.\@ for $z \in R_n$,
$\tilde h_{R_n}(\tfrac{1}{n}z) = \tfrac{1}{n} h_{R_n}(z)$.

The measures $(\mu_{\delta,n})_{\delta > 0,n \in \N}$
satisfy a large deviations principle with speed $r_{\delta, n} := |R_n|$
and tight rate function $I: M(R) \to [0, \infty]$ given by
\[
	I(h_R) := \begin{cases}
		\Ent_R(h_R) - E
		&\text{if $h_R|_{\partial R} = h_{\partial R}$} \,, \\
		\infty
		&\text{otherwise} \,.
	\end{cases}
\]
where $E := \inf_{h_R \in M(R, h_{\partial R})} \Ent_R(h_R)$.
More precisely, for any Borel subset $A \subset M(R)$,
\begin{equation} \label{e_ldp_lower_literal}
	- \inf_{h_R \in \interior{A}} I(h_R)
	\le \frac{1}{r_{\delta,n}} \log
		\varliminf_{\delta \to 0} \varliminf_{n \to \infty}
		\mu_{\delta,n}(A)
\end{equation}
and
\begin{equation} \label{e_ldp_upper_literal}
	\frac{1}{r_{\delta,n}} \log
		\varlimsup_{\delta \to 0} \varlimsup_{n \to \infty}
			\mu_{\delta,n}(A)
	\le - \inf_{h_R \in \closure{A}} I(h_R) \,,
\end{equation}
where as usual $\varliminf$ and $\varlimsup$
denote the limit inferior and superior respectively.
\end{theorem}

\begin{remark}
It is straightforward to reduce the double limits in Theorem~\ref{th_ldp}
to (single) sequential limits, which are more common in large deviations theory.
For example, one may choose any sequences
$(\delta_k)_{k \in \N}, (\varepsilon_k)_{k \in \N}$
such that $\delta_k \to 0$ and $\varepsilon_k \to 0$ as $k \to \infty$.
Then, choose $n_k$ large enough that
\[ \begin{aligned}
	\varliminf_{n \to \infty} \mu_{\delta_k, n}(A) - \varepsilon_k
	\le \mu_{\delta_k,n_k}(A)
	\le \varlimsup_{n \to \infty} \mu_{\delta_k, n}(A) + \varepsilon_k \,,
\end{aligned} \]
and define $\mu_k := \mu_{\delta_k,n_k}$ and $r_k := r_{\delta_k,n_k}$.
Then
\[
	-\inf_{h_R \in \interior{A}} I(h_R)
	\le \varliminf_{k \to \infty} \frac{1}{r_k} \log \mu_k(A)
	\le \varlimsup_{k \to \infty} \frac{1}{r_k} \log \mu_k(A)
	\le -\inf_{h_R \in \closure{A}} I(h_R) \,.
\]
\end{remark}

For the study of limit shapes, it is useful to prove two additional results:
existence and uniqueness of the minimizer of the rate function
from the large deviations principle, i.e.
there exists a unique $h_R^\texttt{min} \in M(R, h_{\partial R})$
such that
\[
	I(h_R^\texttt{min})
	= \inf_{h_R \in M(R, h_{\partial R})} I(h_R) \,.
\]

Indeed, this holds for the simple model studied in the current article.
See for example~\cite{She05} for proofs and discussion of these results.
Even in more subtle models,
the existence of the minimizer is often easy to show:
the proof is standard as long as the local surface tension is convex
and bounded below.
To show uniqueness is harder.
Uniqueness of the minimizer may be proved
using strict convexity of the local surface tension
see for example Proposition~4.5 of~\cite{DSSa08}.
We do not prove these results in the current article,
but rather focus on the variational principle and large deviations principle.

Once existence and uniqueness of the minimizer are established
(or in the language of the current model,
once it is known that the macroscopic entropy functional
admits a unique minimizing height function),
one can explain the appearance of a limit shape in the following way.
The set of asymptotic height functions that lie within distance $\varepsilon$
of this minimizer is an open ball in the space $M(R)$.
By applying the large deviations principle on the set-theoretical complement,
one sees that the percentage of microscopic height functions
in $M(R_n, h_{\partial R_n}, \delta)$ that do not lie $\varepsilon$-close
to the minimizer decays exponentially.
In other words, with high probability a randomly chosen height function
is close to the minimizer,
and therefore the minimizer is the limit shape.

\section{Outline and discussion of proof of main results} \label{s_outline}

In this section we briefly outline the proof of the main results
and summarize some key ideas.
Then we analyze the ingredients in the proof
with an eye toward extending the proof to other random surface models.

In Section~\ref{s_ent}, we provide auxiliary results
including basic properties of
the local surface tension and microscopic entropy.
A central ingredient of the overall argument is discussed in
Section~\ref{s_simple}. There we prove the profile theorem
in the special case of piecewise affine asymptotic height functions.
In Section~\ref{s_profile}, we extend the profile theorem
to general asymptotic height functions by an approximation argument,
yielding the first main result Theorem~\ref{th_profile}.
In Section~\ref{s_varnprin}, we use the profile theorem
and a compactness argument
to prove the variational principle Theorem~\ref{th_varnprin}.
The argument is based on compactness of
the space of asymptotic height functions
with fixed boundary values $M(R, h_{\partial R})$.
Finally in Section~\ref{s_ldp}, we extend the proof of the variational principle
in order to prove the large deviations principle Theorem~\ref{th_ldp}.

As one can see from this outline,
the main idea of the argument is to reduce the proof of the profile theorem from general domains and asymptotic height functions
to simpler domains and asymptotic height functions by an approximation argument.
This means that
the left-hand side of~\eqref{e_profile_theorem}, i.e.~the macroscopic entropy,
and the right-hand side, i.e.~the microscopic entropy,
must both be robust with respect to approximations.

The macroscopic entropy is robust because
$\ent(s)$ is bounded and uniformly continuous,
and Lipschitz functions can be approximated
very well by linear interpolations on a simplex domain (cf.~the simplicial Rademacher theorem Lemma~\ref{lem_approx_tri}).
This approximation lemma was formulated for two dimensions in~\cite{CKP01}.
The result is interesting in its own right
and for the convenience of the reader we state and prove it for arbitrary dimension in Section~\ref{s_profile}.

The microscopic entropy is robust under approximations
because the microscopic surface tension is very robust:
even with fluctuations in the boundary values and the geometry of the boundary,
one still gets the same limit in~\eqref{e_loc_surf_tens}.
This result is proved in Section~\ref{s_ent},
using a Kirszbraun theorem for graph homomorphisms stated below.
This theorem gives conditions under which a graph homomorphism can be extended
from a smaller domain to a larger domain.
This is a discrete analogue to the classical result~\cite{Kir34},
which deals with Lipschitz functions defined on subsets of $\R^d$.
We also note that more general forms of the Kirszbraun theorem for graph
homomorphisms are known, e.g.~\cite{ChPaTa18}.

\begin{theorem}[Kirszbraun theorem for height functions] \label{th_kirszbraun}
Let $\Lambda$ be a connected region of $\Z^m$,
let $S \subset \Lambda$,
and let $\overline h: S \to \Z$ be a graph homomorphism that preserves parity.
There exists a graph homomorphism $h: \Lambda \to \Z$
such that $h = \overline h$ on $S$ if and only if for all $x$, $y$ in $S$,
\begin{equation} \label{e_kirszbraun}
	|\overline h(x) - \overline h(y)| \le |x-y|_1 \,,
\end{equation}
where $|x-y|_1$ is the $\ell^1$-norm in $\Z^m$.
\end{theorem}

\begin{remark}
The parity condition is necessary in general;
consider for example the function $\overline h$ defined on $\{0,2\} \subset \Z$
by $\overline h(0) = 0$, $\overline h(2) = 1$.
The parity condition in Theorem~\ref{th_kirszbraun}
is the reason for the parity condition in Definition~\ref{d_ht_func}.

Two of the authors gave a proof of a more general version of this theorem
in~\cite{MeTa16} (see Theorem~4.1).
The proof is restated below for the reader's convenience.
This proof is also simplified by only addressing the model from this article,
where the height functions take values in $\Z$
rather than in a $d$-regular tree.
\end{remark}

\begin{proof}[Proof of Theorem~\ref{th_kirszbraun}]
Obviously if an extension $h$ of $\overline h$ exists,
then $\overline h$ satisfies~\eqref{e_kirszbraun}.
So, suppose instead that~\eqref{e_kirszbraun},
and let us prove that an extension $h$ exists.
For $y \in \Lambda$, set
\begin{equation} \label{e_kirszbraun_def}
	h(y)
	:= \max \bigl\{ \overline h(x) - |x-y|_1 \,\big|\, x \in S \bigr\} \,.
\end{equation}

We must check two things:
first, that $h(y) = \overline h(y)$ when $y \in S$,
and second, that $|h(y) - h(\tilde y)| = 1$
when $y \sim \tilde y$ are adjacent points in $\Lambda$.

To prove that $h|_S = \overline h$,
let $y \in S$ and consider any point $x \in S$.
By the Lipschitz property of $\overline h$,
\[
	\overline h(x) - \overline h(y)
	\le \bigl| \overline h(x) - \overline h(y) \bigr|
	\le |x - y|_1 \,,
\]
so $\overline h(x) - |x-y|_1 \le \overline h(y)$.
Therefore the maximum in~\eqref{e_kirszbraun_def}
is attained when $x=y$, so $h(y) = \overline h(y) + |y-y|_1 = \overline h(y)$.

To prove that $h$ is a graph homomorphism,
let $y \sim \tilde y$ be adjacent points in $\Lambda$,
and let $x, \tilde x$ be points in $S$
that attain the maximum in~\eqref{e_kirszbraun_def}
for $y, \tilde y$ respectively,
i.e.\@ $h(y) = \overline h(x) - |x-y|_1$
and $h(\tilde y) = \overline h(\tilde x) - |\tilde x - \tilde y|_1$.
Then
\begin{equation} \begin{aligned}
	h(y)
	&= \max \bigl\{ \overline h(z) + |z - y|_1 \,\big|\, z \in S \bigr\} \\
	&\ge \overline h(\tilde x) - |\tilde x - y| \\
	&\ge \overline h(\tilde x) - |\tilde x - \tilde y| - 1 \\
	&= h(\tilde y) - 1 \,,
\end{aligned} \end{equation}
and likewise $h(\tilde y) \ge h(y) - 1$.

For every $x \in S$,
the map $y \mapsto \overline h(x) + |x-y|_1$ preserves parity
(recall the assumption that $\overline h$ preserves parity),
and therefore so does $h$.
So $h$ is a parity-preserving map such that $|h(y) - h(\tilde y)| \le 1$
whenever $y$ and $\tilde y$ are neighbors.
This proves that $h$ is a graph homomorphism.
\end{proof}

Now, we describe further how to prove the central theorem of this article,
i.e.\@ the profile theorem
in the special case of piecewise affine height functions.
We derive the desired asymptotic equality by showing two inequalities.
One direction of the inequality arises by overcounting the number of height
functions that are close to the piecewise affine height profile;
the opposite direction arises by undercounting the same set.
In both directions, we subdivide the region into small blocks,
so that we can compare the entropy on each block to the local surface tension
(see Definition~\ref{d_loc_surf_tens} and Figure~\ref{f_decomp_tri}).
To overcount, we consider all choices of boundary values
on the boundaries of the blocks,
and for each boundary value function we count all possible extensions
into the interior of the blocks.
To undercount we have to use much smaller blocks,
with boundary values fixed to match the desired affine function exactly
(after rescaling, and up to rounding).
The details of the proof are given in Section~\ref{s_simple}.
The more difficult part of the proof is the overcounting argument,
which relies on robustness of the microscopic entropy.
We expect this to be a major source of difficulty
when adapting our methods to other models.

As one can see, the framework of this argument is quite general
and it can be adapted to more complicated models and settings.
For example, the model of graph homomorphism into the infinite $d$-regular tree,
studied by some of the current authors in~\cite{MeTa16},
is amenable to this approach.
Additionally, the authors have applied the current strategy
to $\Z$-valued homomorphisms sampled according to a random environment.
This means that the underlying combinatorial model
is the same as in the current article,
but in the definition of the microscopic entropy,
the (uniform) counting measure on $M(R_n, h_{\partial R_n})$
is replaced by a randomly perturbed measure.
The conclusion is a homogenized variational principle,
meaning that the microscopic entropy $\Ent_n(M(R_n, h_{\partial R_n}))$,
now a random variable depending on the realization of the environment,
converges in probability to the minimum of the macroscopic entropy,
which is still a deterministic quantity.
Furthermore, we hope the method applies to other height function models,
such as domino tilings (as studied in e.g.~\cite{CKP01}),
and perhaps even more general tilings (as in e.g.~\cite{She02,Thu90}).

\section{Microscopic entropy and surface tension} \label{s_ent}

In this section, we prove basic properties of the microscopic entropy
and local surface tension.
More precisely, we prove
that $\ent(s)$ is continuous (see Lemma~\ref{lem_cont}),
that $\ent_n(s) \to \ent(s)$ uniformly (see Lemma~\ref{lem_unif_conv}),
and that $\Ent_{R_n}$ is robust under small changes to boundary values
(see Lemma~\ref{lem_robust}).

All three of these proofs split into two cases:
values of the slope $s$ that are close to $1$
(that is, such that $|s|_\infty \ge 1 - \varepsilon$),
where there are comparatively few possible states because of the steep slope;
and slopes away from $1$ (i.e.~$|s|_\infty \le 1 - \varepsilon$),
where we can make arguments based on extending height functions
from one domain to another via the Kirszbraun theorem~%
(Theorem~\ref{th_kirszbraun}).

The first result we state is
about the microscopic entropy for slopes close to $1$.
This lemma is used in the remainder of the section
to handle the case of $s$ close to $1$.

\begin{lemma}[Microscopic entropy for slopes near 1] \label{lem_ent_near_1}
Let $\delta > 0$,
let $s \in [-1,1]^m$ with $|s|_\infty > 1 - \delta$,
and let $n \in \N$.
Consider any boundary height function
$h_{\partial Q_n} \in M(\partial Q_n)$
such that
\begin{equation} \label{e_ent_near_1_cond}
	\sup_{z \in \partial Q_n} \bigl| h_{\partial Q_n}(z)
		- h_{\partial Q_n}^s(z) \bigr|
	\le \delta n \,.
\end{equation}

Then,
\[
	\Ent_{Q_n} \bigl( M(Q_n, h_{\partial Q_n}) \bigr)
	= \theta(\delta) + \theta_{\delta}(\tfrac{1}{n}) \,.
\]
\end{lemma}

\begin{proof}
First, consider the one-dimensional case, i.e.~$m=1$.
Then the problem reduces to a simple calculation.
The main idea is that the large slope $s$ forces a height function
$h_{Q_n} \in M(Q_n, h_{\partial Q_n})$
to closely follow a line of slope $\pm 1$.
By counting the number of deviating edges
we overestimate the number of height functions.

Indeed, we assume without loss of generality that $\R^1 \ni s > 1 - \delta$
(the case $s < -(1-\delta)$ is symmetric).
We want to count height functions $h_{Q_n} \in M(Q_n, h_{\partial Q_n})$.
The line graph $Q_n = \{-n, -(n-1), \dotsc, n-1, n\}$
has $2n$ edges;
let us write $k$ for the number of edges on which $h_{Q_n}$ decreases
(see Figure~\ref{f_ht_func_near_1}).
Then the height difference $h_{Q_n}(n) - h_{Q_n}(-n)$
is exactly $(2n-k) - k$, which we simplify to $2(n-k)$.
By~\eqref{e_ent_near_1_cond},
we have (ignoring rounding errors)
\[ \begin{aligned}
	h_{Q_n}(n) - h_{Q_n}(-n)
	&= h_{\partial Q_n}^s(n) - h_{\partial Q_n}^s(-n) \\
	&= sn - s(-n) \\
	&> 2(1-\delta)n \,.
\end{aligned} \]

Therefore, $2(n-k) > 2(1-\delta)n$, so $k \le \delta n$.
It follows that
\[
	\bigl| M(Q_n, h_{\partial Q_n}) \bigr|
	\le \binom{2n}{k}
	\le \binom{2n}{\lceil \delta n \rceil} \,,
\]
and the limit
\[
	\lim_{n \to \infty} \, \frac{1}{2n+1}
		\log \binom{2n}{\lceil \delta n \rceil}
	= - \delta \log \delta - (1-\delta) \log (1-\delta)
	= \theta(\delta)
\]
is an easy calculation using Stirling's formula.

\begin{figure}
	\centering
	\begin{tikzpicture}[
		x=0.15in,y=0.15in,
		every circle/.style={radius=0.015625in},
	]
		\draw[help lines] (-5.8,-5.8) grid[step=1] (5.8,5.8);
		\draw[<->,help lines,thick] (-6.5,0) -- (6.5,0)
			node [anchor=north] {$x$};
		\draw[<->,help lines,thick] (0,-6.5) -- (0,6.5)
			node [anchor=west] {$h(x)$};
		\path[draw,fill] (-5.8,-3.8)
			-- (-5,-3) circle -- (-4,-2) circle -- (-3,-1) circle
			-- (-2,0) circle -- (-1,-1) circle -- (0,0) circle
			-- (1,1) circle -- (2,2) circle -- (3,3) circle
			-- (4,4) circle -- (5,5) circle -- (5.8,5.8);
	\end{tikzpicture}
	\caption{A one-dimensional height function with slope $s > 1-\delta$.
	Because the slope is close to $1$, there cannot be many edges
	along which $h(x)$ decreases.}
	\label{f_ht_func_near_1}
\end{figure}
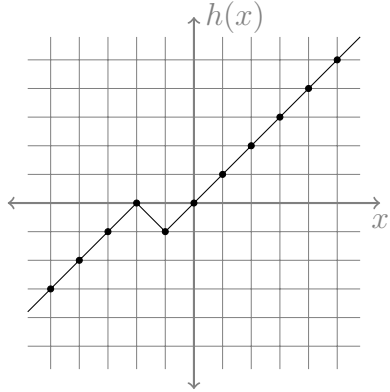

\medskip

For higher dimensions, we reduce to the one-dimensional case
by treating the hypercube $\{-n,\dotsc,n\}^m$ as the union
of $(2n+1)^{m-1}$ independent lines.
In so doing we overestimate $|M(Q_n, h_{\partial Q_n})|$,
because we relax the graph homomorphism condition between lines.
Thus
\[
	|M(Q_n, h_{\partial Q_n})|
	\le \binom{2n}{\lceil \delta n \rceil}^{(2n+1)^{m-1}} \,.
\]

Taking a logarithm and dividing by $-|Q_n| = -(2n+1)^m$ yields
\[
	\Ent_{Q_n} \bigl( M(Q_n, h_{\partial Q_n}) \bigr)
	\ge -\frac{1}{2n+1} \log \binom{2n}{\lceil \delta n \rceil} \,,
\]
which completes the proof (and in particular shows why the
$\theta$ error terms do not depend on the dimension $m$).
\end{proof}

While Lemma~\ref{lem_ent_near_1} deals with slopes $s$ close to $1$,
a different approach is needed for slopes away from $1$.
We use Theorem~\ref{th_kirszbraun}, which is a Kirszbraun theorem
for graph homomorphisms.
It gives a simple criterion for when a height function
can be extended to larger domain.
Lemma~\ref{lem_ent_kirsz} applies the Kirszbraun theorem
to derive entropy estimates.
In particular, for two box sizes $n < \hat n$, the lemma compares
$\Ent_{Q_n}(M(Q_n, h_{\partial {Q_n}}))$ and
$\Ent_{Q_{\hat n}}(M(Q_{\hat n}, h_{\partial {Q_{\hat n}}}))$.
The key idea is that any height function on the smaller box $Q_n$
can be extended to a height function on $Q_{\hat n}$,
respecting the boundary data $h_{\partial Q_{\hat n}}$.
Therefore (up to vanishing error terms),
$
	\Ent_{Q_{\hat n}}(M(Q_{\hat n}, h_{\partial {Q_{\hat n}}}))
	\le
	\Ent_{Q_n}(M(Q_n, h_{\partial {Q_n}}))
$.

The extension requires that the boundary data
$h_{\partial Q_n}$ and $h_{\partial Q_{\hat n}}$ be sufficiently similar.
In particular, we will assume that both boundary height functions are
close to linear height functions, with slopes $s$ and $\hat s$ respectively.
The parameter $\varepsilon$ quantifies how close $h_{\partial Q_n}$
and $h_{\partial Q_{\hat n}}$ are to their respective linear height functions.

We also require that the slopes $s$ and $\hat s$ be close to each other,
which is obviously necessary to apply the Kirszbraun theorem in our setting.
Finally, we require that the two boxes sizes $n$ and $\hat n$ be not too
different. In particular, we take $\hat n = (1 + \delta) n$,
where $\delta$ is a second approximation parameter.
$\delta$ also shows up in a few other bounds,
and in the conclusion of the lemma as a $\theta_m(\delta)$ error term.

This is not the simplest lemma of its kind that we could state,
nor is it the most general.
We choose to state these conditions because they are sufficient for our
applications in this section.
Moreover, they are necessary in the sense that simplifying any condition,
e.g.~by using only a single slope $s$ rather than two slopes,
or by using linear boundary height functions without than allowing $\varepsilon$
fluctuations,
would not suffice for our purposes.

\begin{lemma}[Entropy estimates from the Kirszbraun theorem]
\label{lem_ent_kirsz}
Let $\delta \in (0, \tfrac{1}{3})$,
$\varepsilon \in (0, \tfrac{\delta^2}{2+\delta}]$,
$n, \hat n \in \N$,
$s, \hat s \in [-1,1]^m$,
$h_{\partial Q_n} \in M(\partial Q_n)$,
and $h_{\partial Q_{\hat n}} \in M(\partial Q_{\hat n})$
satisfy:
\begin{itemize}
\item $\hat n = \ceil{(1+\delta) n}$ or $\hat n = \ceil{(1+\delta) n}+1$,
\item $|s|_\infty \le 1 - 3\delta$
	and $|s-\hat s|_\infty \le \tfrac{\delta^2}{1+\delta}$,
\item $\max_{z \in \partial Q_n} |h_{\partial Q_n}(z) - s \cdot z|
	\le \varepsilon n$, and
\item $\max_{z \in \partial Q_{\hat n}}
		|h_{\partial Q_{\hat n}}(z) - \hat s \cdot z|
	\le \varepsilon \hat n$.
\end{itemize}
See Figure~\ref{f_ent_kirsz_domains} for a partial illustration.
Then:
\begin{equation} \label{e_ent_kirsz}
	\Ent_{Q_{\hat n}} \bigl( M(Q_{\hat n}, h_{\partial Q_{\hat n}}) \bigr)
	\le \Ent_{Q_n} \bigl( M(Q_n, h_{\partial Q_n}) \bigr)
		+ \theta_m(\delta) + \theta_m(\tfrac{1}{n}) \,.
\end{equation}
\end{lemma}

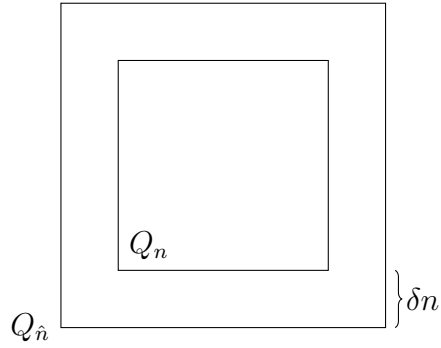
\begin{figure}
	\centering
	\begin{tikzpicture}[
		x=0.01in,y=0.01in,
	]
		\path[draw] (-55,-55)
			node [anchor=south west]{$Q_n$}
			-- (55,-55) -- (55,55) -- (-55,55) -- cycle;
		\path[draw] (-85,-85)
			node [anchor=east] {$Q_{\hat n}$}
			-- (85,-85) -- (85,85) -- (-85,85) -- cycle;
		\path[draw,decorate,decoration=brace]
			(85,-55) + (0.05in,0) --
			node [anchor=west] {$\delta n$}
			+ (0.05in,-30);
	\end{tikzpicture}
	\caption{Domains for Lemma~\ref{lem_ent_kirsz}.
	The smaller domain $Q_n$ is a hypercube of side length $2n+1$,
	and the larger domain $Q_{\hat n}$ has side length $2 \hat n + 1$,
	where $\hat n = (1+\delta)n$.
	}
	\label{f_ent_kirsz_domains}
\end{figure}

\begin{proof}[Proof of Lemma~\ref{lem_ent_kirsz}]
We apply the Kirszbraun theorem (Theorem~\ref{th_kirszbraun})
to construct an injection from $M(Q_n, h_{\partial Q_{\hat n}})$
into $M(Q_{\hat n}, h_{\partial Q_{\hat n}})$.
The existence of such an injection implies that
\[
	\bigl| M(Q_{\hat n}, h_{\partial Q_{\hat n}}) \bigr|
	\ge \bigl| M(Q_n, h_{\partial Q_n}) \bigr| \,
\]
so
\[
	\Ent_{Q_{\hat n}} \bigl( M(Q_{\hat n}, h_{\partial Q_{\hat n}}) \bigr)
	\le \frac{|Q_n|}{|Q_{\hat n}|}
		\Ent_{Q_n} \bigl( M(Q_n, h_{\partial Q_n}) \bigr)
	\,.
\]

Of course,
$\frac{|Q_n|}{|Q_{\hat n}|} = (\frac{n}{\hat n})^m + \theta_m(\tfrac{1}{n})$
and
\[
	\frac{n}{\hat n}
	= \frac{n}{(1 + \delta)n} + \theta(\tfrac{1}{n})
	= 1 + \theta(\delta) + \theta(\tfrac{1}{n}) \,.
\]
Since $\Ent_{Q_n}(M(Q_n, h_{\partial Q_n}))$ is bounded
(see Observation~\ref{obs_ent_bdd}),
the conclusion follows.

\smallskip

So, it remains to show that any height function
$h_{Q_n} \in M(Q_n, h_{\partial Q_n})$
can be extended to a height function in
$M(Q_{\hat n}, h_{\partial Q_{\hat n}})$.
We want to apply the Kirszbraun theorem.
The parity condition is automatic from our parity assumption
in the definition of height functions (see Definition~\ref{d_ht_func}).
We must verify inequality~\eqref{e_kirszbraun}.

Let $z \in \partial Q_n$ and $\hat z \in \partial Q_{\hat n}$.
By the triangle inequality,
\[ \begin{aligned}
	\hskip3em&\hskip-3em
	\bigl| h_{\partial Q_n}(z) - h_{\partial Q_{\hat n}}(\hat z) \bigr| \\
	&\le \bigl| h_{\partial Q_n}(z) - s \cdot z \bigr|
	+ \bigl| s \cdot (z - \hat z) \bigr| \\
	&\qquad + \bigl| (s - \hat s) \cdot \hat z \bigr|
	+ \bigl| \hat s \cdot \hat z - h_{\partial Q_{\hat n}}(\hat z) \bigr| \\
	&\le \varepsilon n + |s|_\infty |z - \hat z|_1 \\
	&\qquad + |s - \hat s|_\infty + \varepsilon \hat n \\
	&\le \varepsilon n + (1 - 3\delta) |z - \hat z|_1 \\
	&\qquad + \tfrac{\delta^2}{1+\delta} \bigl[ (1+\delta) n + 1 \bigr]
	+ \varepsilon \bigl[ (1+\delta) n + 1 \bigr] \\
	&= \varepsilon (2+\delta) n + (1-3\delta) |z - \hat z|_1
	+ \delta^2 n \\
	&\qquad + \tfrac{\delta^2}{1+\delta} + \varepsilon \\
	&\le \tfrac{\delta^2}{2+\delta} (2+\delta) n
	+ (1-3\delta) |z - \hat z|_1 + \delta^2 n \\
	&\qquad + \tfrac{\delta^2}{1+\delta} + \tfrac{\delta^2}{2+\delta} \\
	&\le 2 \delta^2 n + (1 - 3\delta) |z - \hat z|_1 + 2\delta^2 \,.
\end{aligned} \]

By definition, $|z - \hat z|_1 \ge \hat n - n \ge \delta n$,
so $2\delta^2 n \le 2\delta |z - \hat z|_1$.
Moreover, the leftover term $2 \delta^2$ is $\le 1$, so
\[
	\bigl| h_{\partial Q_n}(z) - h_{\partial Q_{\hat n}}(\hat z) \bigr|
	\le (1-\delta) |z - \hat z|_1 + 1 \,.
\]

We assumed that $\delta |z - \hat z|_1 \ge \delta^2 n > 1$,
so the right-hand side above is strictly less that $|z - \hat z|_1$.
Therefore the Kirszbraun theorem applies,
which completes this proof.
\end{proof}

Now, we may quickly state and prove a few useful properties
of the microscopic entropy and surface tension,
using Lemma~\ref{lem_ent_near_1} and Lemma~\ref{lem_ent_kirsz} for the proofs.

\begin{lemma}[Continuity of $\ent(s)$] \label{lem_cont}
The function $s \mapsto \ent(s)$, from $[-1,1]^m$ to $[-\log 2, 0]$,
is  continuous.
In fact, since the domain is compact, $s \mapsto \ent(s)$ is uniformly
continuous.
\end{lemma}

\begin{proof}[Proof of Lemma~\ref{lem_cont}]
First, if $|s|_\infty = 1$, Lemma~\ref{lem_ent_near_1} implies $\ent(s) = 0$.
As $|s|_\infty \to 1$, the same lemma implies
$\ent(s) = \theta(1 - |s|_\infty) \to 0$.
So, $s \mapsto \ent(s)$ is continuous at points $s$ with $|s|_\infty = 1$.

Suppose instead that $|s|_\infty < 1$.
In the language of Lemma~\ref{lem_ent_kirsz},
let $\delta < \tfrac{1}{3} \wedge (1 - |s|_\infty)$,
let $\varepsilon = \tfrac{\delta^2}{2+\delta}$,
and consider any $\hat s$
satisfying $|s - \hat s|_\infty < \tfrac{\delta^2}{1+\delta}$.
For any $n, \hat n$ as in Lemma~\ref{lem_ent_kirsz},
take $h_{\partial Q_n} = h_{\partial Q_n}^s$
and $h_{\partial Q_{\hat n}} = h_{\partial Q_{\hat n}}^{\hat s}$.
Then by Lemma~\ref{lem_ent_kirsz},
\[
	\ent_n(s)
	\ge \ent_{\hat n}(\hat s)
	+ \theta_m(\delta) + \theta_m(\tfrac{1}{n}) \,.
\]

Taking $n \to \infty$, we have
$\ent(s) \ge \ent(\hat s) + \theta_m(\delta)$;
taking $\delta \to 0$,
we conclude that $\ent(s) \ge \liminf_{\hat s \to s} \ent(\hat s)$.
By exchanging the role of $s$ and $\hat s$,
we conclude that $\ent(s) \le \limsup_{\hat s \to s} \ent(\hat s)$.
Therefore the function $s \mapsto \ent(s)$ is continuous.
\end{proof}

\medskip

\begin{lemma}[Uniform convergence of $\ent_n(s)$] \label{lem_unif_conv}
For a fixed dimension $m$,
the convergence of $\ent_n(s) \to \ent(s)$ is uniform in $s \in [-1,1]^m$.
In other words,
\[
	\ent(s) = \ent_n(s) + \theta_m(\tfrac{1}{n}) \,.
\]
\end{lemma}

\begin{proof}[Proof of Lemma~\ref{lem_unif_conv}]
Let $\varepsilon > 0$. We want to find $n_0$, depending only on $\varepsilon$
and $m$, such that $n \ge n_0$ implies $|\ent_n(s) - \ent(s)| < \varepsilon$
for any $s \in [-1,1]^m$.

By Lemma~\ref{lem_ent_near_1} there exists $\delta > 0$
such that $|s|_\infty \ge 1 - \delta$ implies $|\ent_n(s)| < \varepsilon$
for all $n$. This suffices to handle the case $|s|_\infty \ge 1 - \delta$.

For the remaining case of $|s|_\infty \le 1 - \delta$,
we rely on compactness of the space $[-1+\delta, 1-\delta]^m$.
By uniform continuity of $\ent(s)$ (see Lemma~\ref{lem_ent_kirsz},
there exists $\alpha > 0$ such that $|s_1 - s_2|_\infty \le \alpha$
implies $|\ent(s_1) - \ent(s_2)| < \tfrac{\varepsilon}{4}$.
Shrink $\delta$ if necessary so that $\tfrac{\delta^2}{1+\delta} \le \alpha$,
then shrink $\delta$ further so that the $\theta_m(\delta)$
term from~\eqref{e_ent_kirsz}, the conclusion of Lemma~\ref{lem_ent_kirsz},
is less than $\tfrac{\varepsilon}{4}$.

Choose a finite set of slopes
$s^{(1)}, \dotsc, s^{(k)} \in [-1+\delta,1-\delta]^m$
such that for every $s^\star \in [-1+\delta,1-\delta]^m$
there exists some $i=1, \dotsc, k$
with $|s^\star - s^{(i)}|_\infty \le \tfrac{\delta^2}{1+\delta}$.
Finally, choose $n_0$ large enough that
whenever $n \ge \tfrac{1}{1-\delta} n_0$,
the $\theta_m(\tfrac{1}{n})$ term from~\eqref{e_ent_kirsz}
is less than $\tfrac{\varepsilon}{4}$,
and for each $i=1, \dotsc, k$,
$|\ent_n(s^{(i)}) - \ent(s^{(i)})| < \tfrac{\varepsilon}{4}$.

Let $n \ge n_0$, let $s^\star \in [-1+\delta,1-\delta]^m$ be arbitrary,
and fix $i=1, \dotsc, k$
such that $|s^\star - s^{(i)}|_\infty < \tfrac{\delta^2}{1+\delta}$.
We apply Lemma~\ref{lem_ent_kirsz} twice.
First take $s = s^\star$, $\hat s = s^{(i)}$,
$h_{\partial Q_n} = h_{\partial Q_n}^s$,
and $h_{\partial Q_{\hat n}} = h_{\partial Q_{\hat n}}^{\hat s}$.
The conclusion is
\[
	\ent_{\hat n}(s^{(i)})
	\le \ent_n(s^\star) + \theta_m(\delta) + \theta_m(\tfrac{1}{n}) \,,
\]
and by our assumptions on $\delta$ and $n_0$ above,
the $\theta$ terms simplify to
\[
	\ent_{\hat n}(s^{(i)})
	\le \ent_n(s^\star) + \tfrac{2\varepsilon}{4} \,.
\]
By choice of $n_0$,
\[
	\ent_{\hat n}(s^{(i)}) \ge \ent(s^{(i)}) - \tfrac{\varepsilon}{4} \,,
\]
and by choice of $s^{(i)}$,
\[
	\ent(s^{(i)}) \ge \ent(s^\star) - \tfrac{\varepsilon}{4} \,.
\]
Combining the last three inequalities yields
\[
	\ent_n(s) \ge \ent(s) - \varepsilon \,.
\]

For the reverse inequality, choose $s = s^{(i)}$, $\hat s = s^\star$,
and exchange the role of $n$ and $\hat n$.
Repeating the work above, we deduce the inequality
\[
	\ent_n(s) \le \ent(s) + \varepsilon \,,
\]
which completes the proof of Lemma~\ref{lem_unif_conv}.
\end{proof}

\begin{lemma}[Robustness of $\Ent_{Q_n}$]
\label{lem_robust}
Let $n \in \N$, $\varepsilon \in (0, \tfrac{1}{27})$, and $s \in [-1,1]^m$.
Let $h_{\partial Q_n} \in M(\partial Q_n)$ be such that
\begin{equation} \label{e_robust_bndy_approx}
	\sup_{z \in \partial Q_n} \bigl| h_{\partial Q_n}(z)
		- h_{\partial Q_n}^s(z) \bigr|
	\le \varepsilon n \,.
\end{equation}

Then,
\begin{equation} \label{e_robust}
	\Ent_{Q_n} \bigl( M(Q_n, h_{\partial Q_n}) \bigr)
	= \ent(s) + \theta_m(\varepsilon)
		+ \theta_{m,\varepsilon}(\tfrac{1}{n}) \,.
\end{equation}
\end{lemma}

\begin{proof}
Suppose first that $|s|_\infty \ge 1 - \varepsilon^{1/2}$.
Then Lemma~\ref{lem_ent_near_1} applies to both $h_{\partial Q_n}$
and $h_{\partial Q_n}^s$, so
\[ \begin{aligned}
	\Ent_{Q_n} \bigl( M(Q_n, h_{\partial Q_n}) \bigr)
	&= 0 + \theta(\varepsilon) + \theta_\varepsilon(\tfrac{1}{n}) \\
	&= \ent_n(s) + \theta(\varepsilon) + \theta_\varepsilon(\tfrac{1}{n}) \\
	&= \ent(s) + \theta_m(\varepsilon)
		+ \theta_{m,\varepsilon}(\tfrac{1}{n}) \,,
\end{aligned} \]
where in the last line we used Lemma~\ref{lem_unif_conv} for uniform convergence
of $\ent_n(s) \to \ent(s)$, dependent only on dimension $m$.

So suppose instead that $|s|_\infty \le 1 - \varepsilon^{1/2}$.
Apply Lemma~\ref{lem_ent_kirsz} twice.
Both times take $\hat s = s$ and $\delta = \sqrt{3\varepsilon}$.
Note that then
$\varepsilon = \tfrac{\delta^2}{3} \le \tfrac{\delta}{2+\delta}$,
as required by Lemma~\ref{lem_ent_kirsz}.
In the first application of Lemma~\ref{lem_ent_kirsz}
take $h_{\partial Q_n} = h_{\partial Q_n}$
and $h_{\partial Q_{\hat n}} = h_{\partial Q_{\hat n}}^s$
so that
\[
	\ent_{\hat n}(s)
	\le \Ent_{Q_n} \bigl( M(Q_n, h_{\partial Q_n}) \bigr)
		+ \theta_m(\delta) + \theta_m(\tfrac{1}{n}) \,.
\]

The second time exchange $n$ and $\hat n$ to derive
\[
	\Ent_{Q_n} \bigl( M(Q_n, h_{\partial Q_n}) \bigr)
	\le \ent_{\hat n}(s) +\theta_m(\delta) + \theta_m(\tfrac{1}{n}) \,.
\]

Since $\delta$ is determined by $\varepsilon$,
we may replace $\delta$ by $\varepsilon$ in the $\theta$ terms above.
And as before, Lemma~\ref{lem_unif_conv} implies that
$\ent_{\hat n}(s) \to \ent(s)$ and $\ent_n(s) \to \ent(s)$ as $n \to \infty$,
at a rate depending only on the dimension and on $\delta$
(since $n$, $\hat n$ differ from $n$ by a factor of $(1+\delta)^{\pm 1}$).
Therefore
\[
	\Ent_{Q_n} \bigl( M(Q_n, h_{\partial Q_n}) \bigr)
	= \ent(s) + \theta_m(\varepsilon) + \theta_{m,\varepsilon}(\tfrac{1}{n})
\]
as claimed.
\end{proof}

\section{Profile theorem for piecewise affine functions} \label{s_simple}

In this section we prove a simpler version of the profile theorem,
restricted to the case where the domain $R$ is a finite union of simplices
and where the asymptotic height function $h_R$ is piecewise affine,
that is, affine when restricted to a single simplex.
On one hand this case is simple enough that we can prove the profile theorem
directly via over-\@ and under-counting arguments
(see the proof of Theorem~\ref{thm_complex} below).
On the other hand, this case is sufficiently powerful
to approximate general domains and height functions very well
(see the proof of Theorem~\ref{th_profile}
and especially Lemma~\ref{lem_approx_tri}).

We must impose some regularity assumption on the simplices chosen;
in particular we need the isoperimetric ratio to be bounded above
(that is, the surface area of a simplex must not be too large
in comparison to its volume).
For simplicity we restrict our attention to certain families of simplices.
Now let us introduce a standard notation describing these simplices.

In our exposition we follow~\cite{She05}.
For a point $w = (w_1, \dotsc, w_m) \in \R^m$,
we recall from our list of notations that
$\floor{w} := (\floor{w_1}, \dotsc, \floor{w_m}) \in \Z^m$.
For a typical point $w \in \R$, let $s(w)$
denote the permutation of $\{1, \dotsc, m\}$
which rank-orders the components of $w - \floor{w}$.
In particular,
\[
	w_{s(1)} - \floor{w_{s(1)}}
	> w_{s(2)} - \floor{w_{s(2)}}
	> \dotsb
	> w_{s(m)} - \floor{w_{s(m)}} \,.
\]
For example, consider the point $w = (1.1, -0.5, 2.3) \in \R^3$. Then
\[
	\floor w = (1, -1, 2)
	\quad \text{and} \quad
	w - \floor w = (0.1, 0.5, 0.3) \,.
\]
Since the first largest coordinate in $w - \floor w$ is at index $2$,
the second largest coordinate is at index $3$,
and the third largest (i.e.~the smallest) is at index $1$,
we have $s(w) = (2 \; 3 \; 1)$.

\begin{definition}[Simplices of scale $1$] \label{d_simpl_one}
Let $S_m$ denote the group of permutations on $\{1, \dotsc, m\}$.
For $v \in \Z^m$ and $s \in S_m$, we define $C(v,s)$
to be the closure of the set
\begin{equation} \label{e_c_v_s}
	\bigl\{ w \in \R^m \,\big|\, \floor{w} = v, \, s(w) = s \bigr\} \,.
\end{equation}
\end{definition}

A few members of the family $\{C(v,s) \,|\, v \in \Z^m, \, s \in S_m\}$
are illustrated in Figure~\ref{f_c_v_s} in the case of dimension $m=2$.
It is an elementary observation that the $m!$ simplices
$\{C(0, s) \,|\, s \in S_m\}$
tile the hypercube $[0,1]^m$,
i.e.~$\bigcup_{s \in S_m} C(0,s) = [0,1]^m$,
and any two simplices from $\{C(0,s) \,|\, s \in S_m\}$ only have at an
$(m-1)$-dimensional intersection.
It follows that, the infinite family $\{C(v,s) \,|\, v \in \Z^m, s \in S_m\}$
tiles $\R^m$.

Moreover, any two simplices $C(v_1, s_1)$ and $C(v_2, s_2)$ are isometric.
That is, there exists a distance-preserving bijection $f: \R^m \to \R^m$
such that $f(C(v_1, s_1)) = C(v_2, s_2)$.
This ensures that all the simplices $C(v,s)$ have the same isoperimetric ratio.
For our purposes we will also make reference to rescaled simplices.

\begin{definition}[Simplices of scale $\ell$] \label{d_simpl_ell}
For $\ell > 0$, $v \in \Z^m$, and $s \in S_m$,
we write
\[
	\ell C(v, s) := \{ \ell x \,|\, x \in C(v,s) \}
\]
for scaled copy of the simplex $C(v, s)$, scaled out from the origin.
\end{definition}

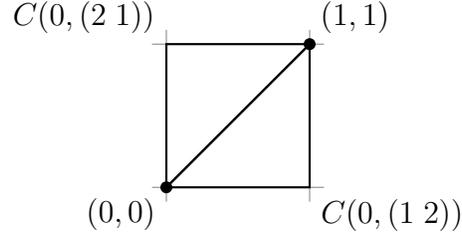
\begin{figure} \centering
\begin{tikzpicture}[x=0.75in,y=0.75in]
	\draw[step=1, help lines] (-0.1,-0.1) grid (1.1, 1.1);
	\fill (0,0) circle[radius=1in/32] (1,1);
	\draw (0,0) node [below left] {$(0,0)$};
	\fill (1,1) circle[radius=1in/32] (1,1);
	\draw (1,1) node [above right] {$(1,1)$};
	\draw [thick] (0,0) -- (1,1)
		-- (0,1) node [above left] {$C(0,(2 \; 1))$} -- cycle;
	\draw [thick] (0,0) -- (1,1)
		-- (1,0) node [below right] {$C(0, (1 \; 2))$} -- cycle;
\end{tikzpicture}
\caption{The two simplices in dimension $2$ that tile the unit square.
The simplex $C(0, (1 \; 2))$ is the closure of
the set of points $(x,y) \in [0,1]^2$ such that $x > y$,
and $C(0, (2 \; 1))$ is the closure of the points with $y > x$.
The other simplices $\{C(v,s) \,|\, v \in \Z^m, s \in S_2\}$
are translates of these two simplices.
}
\label{f_c_v_s}
\end{figure}

\begin{figure} \centering
\newcommand\scale{1}
\newcommand\sep{0.1}
\newcommand*\zero{0.05}
\newcommand*\one{0.95}
\newcommand\simplex[5][]{
	\path[#1] #2 -- #3 -- #4 -- cycle;
	\path[#1] #2 -- #3 -- #5 -- cycle;
	\path[#1] #2 -- #4 -- #5 -- cycle;
	\path[#1] #3 -- #4 -- #5 -- cycle;
}
\begin{tikzpicture}[
		x={(\scale*1in, \scale*cos(30)*1in)},
		y={(-\scale*1in,\scale*cos(30)*1in)},
		z={(0in,\scale*1in)}]
	\begin{scope}[shift={(0, \sep, -2*\sep)}]
		\simplex[draw=OliveGreen,fill=OliveGreen!25,fill opacity=0.25]
			{ (\zero,\zero,\zero) }{ (\zero,\one ,\zero) }
			{ (\one ,\one ,\zero) }{ (\one ,\one ,\one ) };
	\end{scope}
	\begin{scope}[shift={(\sep,0,-2*\sep)}]
		\simplex[draw=red,fill=red!25,fill opacity=0.25]
			{ (\zero,\zero,\zero) }{ (\one ,\zero,\zero) }
			{ (\one ,\one ,\zero) }{ (\one ,\one ,\one ) };
	\end{scope}
	\begin{scope}[shift={(\sep,-\sep,0)}]
		\simplex[draw=blue,fill=blue!25,fill opacity=0.25]
			{ (\zero,\zero,\zero) }{ (\one ,\zero,\zero) }
			{ (\one ,\zero,\one ) }{ (\one ,\one ,\one ) };
	\end{scope}
	\begin{scope}[shift={(-\sep,\sep,0)}]
		\simplex[draw=blue,fill=blue!25,fill opacity=0.25]
			{ (\zero,\zero,\zero) }{ (\zero,\one ,\zero) }
		{ (\zero,\one ,\one ) }{ (\one ,\one ,\one ) };
	\end{scope}
	\begin{scope}[shift={(0,-\sep,2*\sep)}]
		\simplex[draw=OliveGreen,fill=OliveGreen!25,fill opacity=0.25]
			{ (\zero,\zero,\zero) }{ (\zero,\zero,\one ) }
		{ (\one ,\zero,\one ) }{ (\one ,\one ,\one ) };
	\end{scope}
	\begin{scope}[shift={(-\sep,0,2*\sep)}]
		\simplex[draw=red,fill=red!25,fill opacity=0.25]
			{ (\zero,\zero,\zero) }{ (\zero,\zero,\one ) }
		{ (\zero,\one ,\one ) }{ (\one ,\one ,\one ) };
	\end{scope}
\end{tikzpicture}
\caption{Decomposition of a unit cube into $\{C(0,s) \,|\, s \in S_3\}$.
The simplices have been separated for a more clear figure.}
\label{f_c_v_s_3d}
\end{figure}
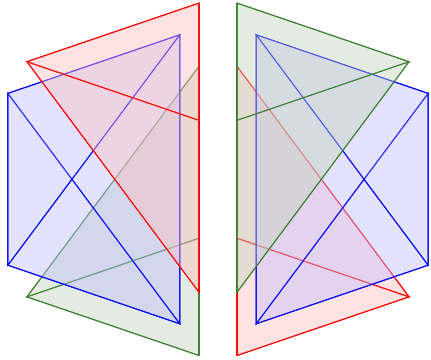
As before, we observe that for any $\ell > 0$,
the family $\{\ell C(0,s) \,|\, s \in S_m\}$ tiles the hypercube $[0,\ell]^m$.
Therefore again, $\{ \ell C(v,s) \,|\, v \in \Z^m, \, s \in S_m\}$
tiles $\R^m$.
To approximate a general domain $R$ that satisfies Assumption~\ref{a_domain},
we consider domains which are the union of simplices.

\begin{definition}[Simplex domain of scale $\ell$] \label{d_simpl_dom}
For $\ell > 0$, a \defn{simplex domain of scale $\ell$}
is a region $K \subset \R^m$ that is the union of finitely many simplices
of scale $\ell$.
We further require that simplex domains be connected,
so that a simplex domain $K$ automatically meets the requirements
from Assumption~\ref{a_domain}.
\end{definition}

For example, the union of the two simplices in Figure~\ref{f_c_v_s}
is a simplex domain of scale $1$.
It is clear that simplex domains can approximate more general
domains $R \subset \R^m$;
we make this observation more precise in Lemma~\ref{lem_approx_tri} below.
Now, let us formulate the main result of this section,
the simplicial profile theorem (Theorem~\ref{thm_complex}).
It is a special case of the profile theorem
for simplex domains and piecewise affine height functions;
cf.~the general profile theorem (Theorem~\ref{th_profile}).

\begin{theorem}[Simplicial profile theorem] \label{thm_complex}
Let $K = \Delta_1 \cup \dotsb \cup \Delta_k$ be a simplex domain
of scale $\ell$, in the sense of Definition~\ref{d_simpl_dom}.
Fix a height function $h_K \in M(K)$
such that each restriction $h_K|_{\Delta_j}$,
$j = 1, \dotsc, r$, is affine.
Let $\varepsilon > 0$, let $n \in \N$,
and let $K_n := \{z \in \Z^m \, | \, \frac{1}{n} z \in K\}$.
Then for any slope $s \in [-1,1]^m$,
\begin{equation} \label{e_complex_profile} \begin{aligned}
	\hskip3em&\hskip-3em
	\Ent_{K_n} \bigl( B(K_n, h_K, \varepsilon \ell) \bigr) \\
	&= \frac{1}{|K|} \int_K \ent(\nabla h_K(x)) \, dx
	+ \theta_m(\varepsilon) + \theta_{m,\varepsilon, \ell}(\tfrac{1}{n}) \,.
\end{aligned} \end{equation}
\end{theorem}

\begin{remark} \label{rem_complex_profile}
In reading the proof of Theorem~\ref{thm_complex} for the first time,
we encourage the reader to consider only a single simplex $\Delta$
rather than a simplex domain $K = \Delta_1 \cup \dotsb \cup \Delta_k$.
The key ideas are more clear when thinking about a single simplex.
In particular the simplex is decomposed into hypercubes two times,
using hypercubes of a different scale each time.
The two scales of hypercubes are illustrated in Figure~\ref{f_decomp_tri}.
One decomposition is used to overestimate the microscopic entropy
by undercounting the set $B(K_n, h_K, \varepsilon \ell)$.
The other is used to underestimate the entropy by overcounting the set.

In the more general case of a simplex domain
we still decompose twice,
using hypercubes of a different size each time.
A typical decomposition is illustrated in Figure~\ref{f_decomp_tris}.
In particular we keep only those hypercubes that lie inside a single simplex,
so that $h_K$ has a single, well-defined slope on each $Q_i$.
Both sides of~\eqref{e_complex_profile} are approximately additive over the
simplices, but we will not explicitly prove this result here,
nor do we rely on it.
\end{remark}

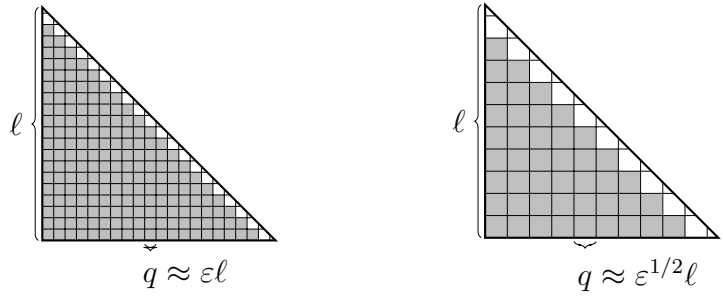
\begin{figure} \centering
\begin{subfigure}{0.4\textwidth}
\begin{tikzpicture}
	\newcommand*\width{3.1}
	\newcommand*\num{20.5}
	\pgfmathparse{\width/\num}
	\let\step\pgfmathresult
	\begin{scope}
		\clip (0,0) -- (\width,0) -- (0,\width) -- cycle;
		\fill[lightgray] (0,0) rectangle (\width,\width);
		\foreach \j in {0,...,\num}
			\fill[white] (\step*\j,{\step*(floor(\num)-\j-1)})
				rectangle +(\step,2*\step);
		\draw[step=\step,very thin] (0,0) grid (\width,\width);
	\end{scope}
	\draw[thick] (0,0) -- (\width,0) -- (0,\width) -- cycle;
	\draw[very thin,decorate,decoration=brace]
		({\step*(floor(\num/2)},-0.05) --
		node[below=2pt]
			{$q$\makebox[0pt][l]{${} \approx \varepsilon \ell$}}
		+(-\step,0);
	\draw[very thin,decorate,decoration=brace]
		(-0.05,0) -- node[left=2pt] {$\ell$} +(0,\width);
\end{tikzpicture}
\end{subfigure}\qquad%
\begin{subfigure}{0.4\textwidth}
\begin{tikzpicture}
	\newcommand*\width{3.1}
	\newcommand*\num{10.5}
	\pgfmathparse{\width/\num}
	\let\step\pgfmathresult
	\begin{scope}
		\clip (0,0) -- (\width,0) -- (0,\width) -- cycle;
		\fill[lightgray] (0,0) rectangle (\width,\width);
		\foreach \j in {0,...,\num}
			\fill[white] (\step*\j,{\step*(floor(\num)-\j-1)})
				rectangle +(\step,2*\step);
		\draw[step=\step,very thin] (0,0) grid (\width,\width);
	\end{scope}
	\draw[thick] (0,0) -- (\width,0) -- (0,\width) -- cycle;
	\draw[very thin,decorate,decoration=brace]
		({\step*(floor(\num/2)},-0.05) --
		node[below=2pt] {$q$\makebox[0pt][l]{%
			${} \approx \varepsilon^{1/2} \ell$}}
		+(-\step,0);
	\draw[very thin,decorate,decoration=brace]
		(-0.05,0) -- node[left=2pt] {$\ell$} +(0,\width);
\end{tikzpicture}
\end{subfigure}
\caption{
	Decomposition of a single simplex into hypercubes at two scales.
	In both images, the shaded squares
	are the $Q_i$ from the proof of Theorem~\ref{thm_complex}.
	The smaller squares on the left are used when undercounting the set
	$B(K_n, h_K, \varepsilon \ell)$
	and the larger squares on the right are used when overcounting this set.
}
\label{f_decomp_tri}
\end{figure}

\begin{figure}
\centering
\input{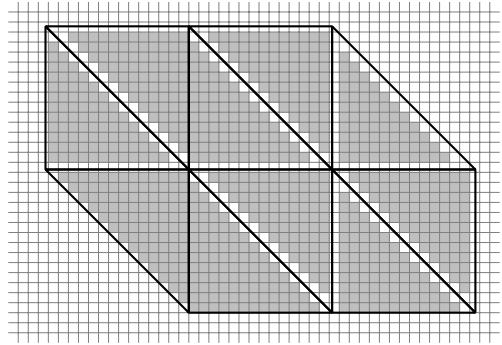}
\caption{
	In Theorem~\ref{thm_complex}
	$K$ may be a simplex domain rather than a single simplex.
	Then only hypercubes that lie inside one of the simplices
	are part of the collection $\{Q_i\}$.
	These hypercubes are shaded.
	Note that there are still two different scales of hypercubes used,
	as illustrated in Figure~\ref{f_decomp_tri},
	but only one scale is shown above.
	The set $G_n$ from later in the proof is the set of grid lines contained
	inside the simplex domain, and $U_n$ is the unshaded part of the simplex
	domain. $S_n$ is the union of $G_n$ and $U_n$.
}
\label{f_decomp_tris}
\end{figure}

\begin{proof}
As mentioned in Remark~\ref{rem_complex_profile},
we subdivide the region $K \subset \R^m$
into hypercubes $Q_i$ (for $i=1, \dotsc, k$) of equal side length $q$.
Two different values for the side length parameter $q$ are used at different
times.
The cubes $Q_i$ lie in a grid with their corners
on the rescaled lattice $q \Z^m \subset \R^m$.
The set $\{Q_1, \dotsc, Q_k\}$ enumerates all such hypercubes
that lie inside exactly one of the simplices $\Delta_1,\dotsc, \Delta_r$,
as illustrated in Figure~\ref{f_decomp_tris}.
This ensures that $h_K$ has constant derivative on $Q_i$,
which makes later arguments simpler.
For $i=1, \dotsc, r$ we choose $s_i \in [-1,1]^m$ and $b_i \in \R$
so that $h_K|_{Q,i} = h_{Q_i}^{s_i \cdot x + b_i}$.
Specifically, this means that
\[
	s_i := \nabla h_K(x_i)
	\quad \text{and} \quad
	b_i := h_K(x_i) - s_i \cdot x_i \,,
\]
for an arbitrarily chosen sample point $x_i$ from the interior of $Q_i$.

The hypercubes $Q_i$ induce a decomposition of the discrete set $K_n$
into subsets $Q_{i,n} := \{z \in \Z^m \, | \, \frac{1}{n}z \in Q_i\}$,
plus a negligible remainder $K_n \setminus \bigcup_{i=1}^r Q_{i,n}$.
This remainder is the unshaded part inside the triangles in
Figure~\ref{f_decomp_tris}.
We write $qn$ for the side length of the discrete hypercubes $Q_{i,n}$.
Technically, each $Q_{i,n}$ has an integer side length $q_{i,n} \in \Z$
that is equal to either $\floor{qn}$ or $\ceil{qn}$,
but for simplicity we elide this detail in the rest of the proof.

\medskip

Let us first sketch the main idea of the proof.
We start with the integral on the right-hand side of~\eqref{e_complex_profile}.
Since $h_K$ is piecewise affine,
the integral reduces to a finite sum
\begin{equation} \label{e_complex_profile_int} \begin{aligned}
	\frac{1}{|K|} \int_K \ent(\nabla h_K(x)) \, dx
	&= \sum_{i=1}^r \frac{1}{|K|} \int_{Q_i}
		\ent(\nabla h_K|_{Q_i}(x)) \, dx
	+ \theta_m(\tfrac{q}{\ell}) \\
	&= \sum_{i=1}^r \frac{|Q_i|}{|K|} \ent(s_i)
	+ \theta_m(\tfrac{q}{\ell}) \\
	&= \frac{1}{r} \sum_{i=1}^r \ent(s_i) + \theta_m(\tfrac{q}{\ell}) \,,
\end{aligned} \end{equation}
where we recall that $s_i = \nabla h_K(x_i)$ for $x_i \in Q_i$.
Both of the two values for the hypercube side length parameter $q$
are chosen so that $\theta_m(\tfrac{q}{\ell}) = \theta_m(\varepsilon)$.
The $\theta_m(\tfrac{q}{\ell})$ errors arise from the uncovered region
$K \setminus \bigcup_{i=1}^r Q_i$,
i.e.~the unshaded parts of the simplex domain in Figure~\ref{f_decomp_tris}.
Indeed, one simply compares the measure $|\Delta_j| = \frac{1}{m!} \ell^m$
against that of the smaller simplex $\Delta_j'$
with sides moved $\sqrt{m} q$ units inwards.
Any hypercube $Q_i$ that intersects $\Delta_j'$ must lie inside $\Delta_j$.
The $\theta_m(\tfrac{q}{\ell})$ error bound follows.

Now, we turn to the left-hand side of~\eqref{e_complex_profile}.
Our goal is to relate $\Ent_{K_n}(B(K_n, h_K, \varepsilon \ell))$
to the sum at the right-hand side of~\eqref{e_complex_profile_int}.
Towards this end, we will under- and over-count the set of height functions
$B(K_n, h_K, \varepsilon \ell)$,
in order to derive the over- and under-estimates
\begin{equation} \label{e_complex_profile_overest}
	\Ent_{K_n} \bigl( B(K_n, h_K, \varepsilon \ell) \bigr)
	\le \frac{1}{r} \sum_{i=1}^r \ent(s_i)
		+ \theta_m(\varepsilon)
		+ \theta_{m,\varepsilon,\ell}(\tfrac{1}{n})
\end{equation}
and
\begin{equation} \label{e_complex_profile_underest}
	\Ent_{K_n} \bigl( B(K_n, h_K, \varepsilon \ell) \bigr)
	\ge \frac{1}{r} \sum_{i=1}^r \ent(s_i)
		+ \theta_m(\varepsilon)
		+ \theta_{m,\varepsilon,\ell}(\tfrac{1}{n}) \,.
\end{equation}

Equations~\eqref{e_complex_profile_int},~\eqref{e_complex_profile_overest},
and~\eqref{e_complex_profile_underest},
together with the observation made above that
$\theta_m(\tfrac{q}{\ell}) = \theta_m(\varepsilon)$,
suffice to prove the theorem.
In order to prove~\eqref{e_complex_profile_overest},
we will undercount height functions in $B(K_n, h_K, \varepsilon \ell)$.
We choose $q \approx \varepsilon \ell$ and consider only height functions
that agree with $h_{K_n}$ on each boundary $\partial Q_{i,n}$.
These boundary data, together with the small size of $Q_{i,n}$,
ensure that $h_{K_n}$ satisfies the $\ell^\infty$ condition
for membership in $B(K_n, h_K, \varepsilon \ell)$.
Then, to prove~\eqref{e_complex_profile_underest},
we overcount height functions.
We choose a larger value $q \approx \varepsilon^{1/2} \ell$.
Any height function $h_n \in B(K_n, h_K, \varepsilon \ell)$,
when restricted to one of the boundary sets $\partial Q_{i,n}$
and rescaled appropriately,
fluctuates away from $h_K$ by at most
$\varepsilon \ell n = \varepsilon^{1/2} q n$.
When $\varepsilon$ is small, this allows us to compare the entropy on $Q_i$
to $\Ent(M(Q_{i,n}, h_{\partial Q_{i,n}}^{s_i})) = \ent_{qn}({s_i})$,
where $s_i = \nabla h_K(x_i)$ for $x_i \in Q_i$.

\medskip

Now, let us describe the undercounting argument in detail.
We seek to derive~\eqref{e_complex_profile_overest},
an overestimate of $\Ent_{K_n}(B(K_n, h_K, \varepsilon \ell))$,
by undercounting the set $B(K_n, h_K, \varepsilon \ell)$.
We take the side length of the hypercubes $\{Q_i\}$
to be $q = \frac{1}{4} \varepsilon \ell$.

We define an injection from the product set
$\prod_i M(Q_{i,n}, h_{\partial Q_{i,n}}^{s_i \cdot x + b_i})$
into the ball $B(K_n, h_K, \varepsilon \ell)$
in the natural way:
given a tuple of height functions
$h_{Q_{i,n}} \in M(Q_{i,n}, h_{\partial Q_{i,n}}^{s_i \cdot x + b_i})$
for $i=1, \dotsc, r$,
we define $h_{K_n}(z) := h_{Q_{i,n}}(z)$ if $z \in Q_{i,n}$.
For $z \in K_n \setminus \bigcup_{i=1}^r Q_{i,n}$,
if $\tfrac{1}{n} z \in \Delta_i$
we set $h_{K_n}(z) = h_{\{z\}}^{s_i \cdot x + b_i}(z)$.
It follows from Lemma~\ref{lem_canon_hf_ineq}
that this function $h_{K_n}$ is a height function,

Let us check that $h_{K_n} \in B(K_n, h_K, \varepsilon \ell)$.
For $z \in K_n \setminus \bigcup_{i=1}^r Q_{i,n}$,
the estimate
$|\tfrac{1}{n} h_{K_n}(z) - h_K(\tfrac{1}{n}z)| \le \varepsilon \ell$
is immediate from the definition of $h_{K_n}(z)$.
So, suppose that $z \in Q_{i,n}$ for some $i=1, \dotsc, r$,
and let $z' \in Q_{i,n}$ be a boundary point in $\partial Q_{i,n}$
that minimizes $\ell^1$ distance from $z$.
In particular, $|z - z'|_1 \le qn$, so
\[ \begin{aligned}
	\hskip3em&\hskip-3em
	\Bigl| \frac{1}{n} h_{K_n}(z)
		- h_K(\tfrac{1}{n} z) \Bigr|
	= \Bigl| \frac{1}{n} h_{Q_{i,n}}(z) - h_{Q_i}(\tfrac{1}{n} z) \Bigr| \\
	&\le \Bigl| \frac{1}{n} h_{Q_{i,n}}(z)
	- \frac{1}{n} h_{Q_{i,n}}(z') \Bigr|
	+ \Bigl| \frac{1}{n} h_{Q_{i,n}}(z') - h_{Q_i}(\tfrac{1}{n} z') \Bigr|
	\\
	&\qquad + \Bigl| h_{Q_i}(\tfrac{1}{n} z')
		- h_{Q_i}(\tfrac{1}{n} z) \Bigr| \\
	&\le \frac{1}{n} |z-z'|_1 + \frac{1}{n} + \frac{1}{n} |z-z'|_1 \\
	&\le 3 q < \varepsilon \ell \,,
\end{aligned} \]
at least for $n$ large enough that
$\tfrac{1}{n} \le q = \tfrac{1}{4} \varepsilon \ell$.
Therefore $h_{K_n} \in B(K_n, h_K, \varepsilon \ell)$ as desired.
Thus
\[
	\Bigl| \prod_{i=1}^r
		M(Q_{i,n}, h_{\partial Q_{i,n}}^{s_i \cdot x + b_i}) \Bigr|
	\le \bigl| B(K_n, h_K, \varepsilon \ell) \bigr| \,,
\]
so
\begin{equation} \label{e_prof_tri_overest}
	\sum_{i=1}^r \frac{|Q_{i,n}|}{|K_n|}
		\Ent_{Q_{i,n}} \bigl( M( Q_{i,n},
			h_{\partial Q_{i,n}}^{s_i \cdot x + b_i}) \bigr)
	\ge \Ent_{K_n} \bigl(
		B(K_n, h_K, \varepsilon \ell) \bigr) \,.
\end{equation}

Now,
\[
	\Ent_{Q_{i,n}} \bigl(
		M(Q_{i,n}, h_{\partial Q_{i,n}}^{s_i \cdot x + b_i}) \bigr)
	= \ent_{qn}(s_i)
\]
by translation invariance (see Observation~\ref{obs_trans_invar}).
Moreover, because $qn = \frac{1}{4} \varepsilon \ell n$,
we have $\ent_{qn}(s_i) \to \ent(s_i)$ as $n \to \infty$
at a rate dependent on $\varepsilon \ell$;
by Lemma~\ref{lem_unif_conv}, the convergence is uniform with respect to $s_i$.
In other words,
$\ent_{qn}(s_i) = \ent(s_i) + \theta_{m,\varepsilon, \ell}(\frac{1}{n})$.
Therefore, recalling~\eqref{e_prof_tri_overest}, we have
\begin{equation} \label{e_simplex_le_sum_ent}
	\Ent_{Kn} \bigl( B(K_n, h_K, \varepsilon \ell) \bigr)
	\le \sum_{i=1}^r \frac{|Q_{i,n}|}{|K_n|} \bigl(
		\ent(s_i) + \theta_{m,\varepsilon, \ell}(\tfrac{1}{n})
	\bigr) \,.
\end{equation}

Now, the difference between $\sum_{i=1}^r \frac{|Q_{i,n}|}{|K_n|}$
and $\tfrac{1}{r}$
is $\theta_m(\varepsilon) + \theta_{m,\varepsilon, \ell}(\frac{1}{n})$,
where the first error term accounts for the unshaded part of
Figure~\ref{f_decomp_tris},
and the second term is due to discretization effects.
Therefore~\eqref{e_simplex_le_sum_ent} simplifies to
\begin{equation} \label{e_prof_compl_under_sum_int}
	\Ent_{K_n} \bigl(
		B(K_n, h_K, \varepsilon \ell) \bigr)
	\le \frac{1}{r} \sum_{i=1}^r \ent(s_i)
	+ \theta_m(\varepsilon) + \theta_{m,\varepsilon, \ell}(\tfrac{1}{n}) \,,
\end{equation}
which is exactly the overestimate~\eqref{e_complex_profile_overest}.

\medskip

Now, we turn to~\eqref{e_complex_profile_underest}.
We will overcount $B(K_n, h_K, \varepsilon \ell)$
in order to underestimate the entropy
$\Ent_{K_n}(B(K_n, h_K, \varepsilon \ell))$.
We will take side length $q$ of the hypercubes $Q_i$ to be
$q = \varepsilon^{1/2} \ell$ for this part of the argument.

The basic idea is the following:
we choose a subset $S_n \subset K_n$,
and we only enforce the condition that
$|\tfrac{1}{n} h_{K_n}(x) - h_K(\tfrac{1}{n} x)| < \varepsilon \ell$
from Definition~\ref{d_ht_func_sets} on $S_n$ rather than on all of $K_n$.
$S_n$ is the complement of the (interiors of the) grid cells $Q_{i,n}$,
so for any fixed height values on $S_n$,
we can count the number of all extensions into the grid cells
using a sum of entropy over the cells.
There are many possible height values on $S_n$
that satisfy the $\varepsilon \ell$ error condition,
but ultimately not too many because $S_n$ is small (compared to $K_n$).

Let us provide more detail.
We define $S_n$ as follows.
Let $G_n$ denote the grid formed by the boundaries of $Q_{i,n}$,
i.e.~the part of the grid lines from Figure~\ref{f_decomp_tris}
that lies inside the simplex domain.
Let $U_n$ denote the points in $K_n$
that lie outside of any hypercube $Q_{i,n}$,
i.e.~the unshaded part of the simplex domain in Figure~\eqref{f_decomp_tris}.
Let $S_n := G_n \cup U_n$.
(As claimed, the complement $K_n \setminus S_n$
is the interior of the grid cells $Q_{i,n}$.)
Additionally, let $\Adm(S_n)$ denote the set
\begin{equation} \label{e_def_adm_s}
	\Adm(S_n)
	:= \bigl\{ \text{``admissible'' height functions on $S_n$} \bigr\} \,,
\end{equation}
where ``admissible'' means those height functions $h_{S_n} \in M(S_n)$
that admit an extension to a height function
in $B(K_n, h_K, \varepsilon \ell)$.

We claim that there is an injection from $B(K_n, h_K, \varepsilon \ell)$ into
\begin{equation} \label{e_complex_profile_underest_set}
	\biguplus_{h_{S_n} \in \Adm(S_n)}
		\prod_{i=1}^r M ( Q_{i,n}, h_{S_n}|_{\partial Q_{i,n}} ) \,,
\end{equation}
where ``$\biguplus$'' denotes the disjoint union
(so for distinct height functions $h_{S_n}$ and $\tilde h_{S_n}$ in $\Adm(S_n)$,
the product sets $\prod_1^r M(Q_{i,n}, h_{S_n}|_{\partial Q_{i,n}})$
and $\prod_1^r M(Q_{i,n}, \tilde h_{S_n}|_{\partial Q_{i,n}})$
are considered disjoint inside the set
from~\eqref{e_complex_profile_underest_set}).

Indeed, for any $h_{K_n} \in B(K_n, h_K, \varepsilon \ell)$,
the function $h_{S_n} := h_{K_n}|_{S_n}$ is by definition in $\Adm(S_n)$,
and $(h_{K_n}|_{Q_{i,n}})_{i=1}^r$ lies in the Cartesian product
from the right-hand side of~\eqref{e_complex_profile_underest_set}.
To see that this map is an injection, suppose that $h_{K_n}$
and $\tilde h_{K_n}$ map to the same point.
Then by definition of the (purported) injection,
$h_{K_n}|_{Q_{i,n}} = \tilde h_{K_n}|_{Q_{i,n}}$ for each hypercube $Q_{i,n}$.
Additionally, since the right-hand side
of~\eqref{e_complex_profile_underest_set} is a \emph{disjoint} union,
we have $h_{K_n}|_{S_n} = \tilde h_{K_n}|_{S_n}$.
Since $K_n = \bigcup_i Q_{i,n} \cup S_n$,
the two height functions $h_{K_n}$ and $\tilde h_{K_n}$ are identical.
Therefore, the map is an injection, and so
\[ \begin{aligned}
	\bigl| B(K_n, h_K, \varepsilon \ell) \bigr|
	&\le \sum_{h_{S_n} \in \Adm(S_n)}
		\prod_{i=1}^r \bigl|
			M(Q_{i,n}, h_{S_n}|_{\partial Q_{i,n}}) \bigr| \\
	&\le \bigl|\Adm(S_n) \bigr| \, \max_{h_{S_n} \in \Adm(S_n)}
		\prod_{i=1}^r \bigl|
			M(Q_{i,n}, h_{S_n}|_{\partial Q_{i,n}}) \bigr| \,.
\end{aligned} \]

Taking logarithms and multiplying by $-\frac{1}{|K_n|}$, we see that
\begin{equation} \label{e_complex_profile_underest_min} \begin{aligned}
	\hskip3em&\hskip-3em
	\Ent_{K_n} \bigl( B(K_n, h_K, \varepsilon \ell) \bigr) \\
	&\ge \min_{h_{S_n} \in \Adm(S_n)}
		\sum_{i=1}^r - \frac{1}{|K_n|} \log \bigl|
			M(Q_{i,n}, h_{S_n}|_{\partial Q_{i,n}}) \bigr| \\
	&\qquad - \frac{\log |\Adm(S_n)|}{|K_n|} \\
	&= \min_{h_{S_n} \in \Adm(S_n)}
		\sum_{i=1}^r \frac{|Q_{i,n}|}{|K_n|} \Ent_{Q_{i,n}} \bigl(
			M(Q_{i,n}, h_{S_n}|_{\partial Q_{i,n}}) \bigr) \\
	&\qquad - \frac{\log |\Adm(S_n)|}{|K_n|} \,.
\end{aligned} \end{equation}

Now, we are almost done.
We use three more asymptotic identities in the right-hand side
of~\eqref{e_complex_profile_underest_min}
to derive~\eqref{e_complex_profile_underest}.
First and simplest, since $
	\frac{|Q_{i,n}|}{|K_n|}
	= \frac{1}{r}
	+ \theta_m(\varepsilon)
	+ \theta_{m,\varepsilon,\ell}(\tfrac{1}{n})
$ and $\Ent_{Q_{i,n}}$ is bounded (Observation~\ref{obs_ent_bdd}),
we replace $\frac{|Q_{i,n}|}{|K_n|}$ by $\frac{1}{r}$
in~\eqref{e_complex_profile_underest_min}.

Second, we apply Lemma~\ref{lem_robust}
to replace $\Ent_{Q_{i,n}}(M(Q_{i,n}, h_{S_n}|_{\partial Q_{i,n}}))$
by $\ent(s_i)$.
We fix a height function $h_{S_n}$ that achieves the minimum,
then apply the lemma on each of the hypercubes $Q_{i,n}$.
We recall that the hypercubes have side length $qn = \varepsilon^{1/2} \ell n$.
Since $h_{S_n}$ is admissible
and since $h_K|_{Q_i} = h_{Q_i}^{s_i \cdot x + b_i}$,
$h_{S_n}$ is approximately affine, i.e.
\[
	\bigl| h_{S_n}(z) - h_{\partial Q_{i,n}}^{s_i \cdot x + b_i}(z) \bigr|
	\le \varepsilon \ell n
	= \varepsilon^{1/2} qn \,.
\]
So, Lemma~\ref{lem_robust} applies and yields
\[ \begin{aligned}
	\Ent_{Q_{i,n}} \bigl( M(Q_{i,n}, h_{S_n}|_{\partial Q_{i,n}}) \bigr)
	&= \ent(s_i) + \theta_m(\varepsilon^{1/2})
		+ \theta_{m,\varepsilon^{1/2}}(
			\tfrac{1}{\varepsilon^{1/2} \ell n}) \\
	&= \ent(s_i) + \theta_m(\varepsilon)
		+ \theta_{m,\varepsilon,\ell}(\tfrac{1}{n}) \,,
\end{aligned} \]
where the last line is just a matter of hiding the functions like
$\varepsilon^{1/2}$ inside our $\theta$-notation.
(See Section~\ref{ss_asymp} for the definition of $\theta$-notation.)

Finally, we claim that $
	\frac{1}{|K_n|} \log |\Adm(S_n)|
	= \theta_m(\varepsilon)
	+ \theta_{m,\varepsilon,\ell}(\tfrac{1}{n})
$,
where we recall that the set $\Adm(S_n)$ was defined in~\eqref{e_def_adm_s}.
To see this, fix a base point $z_0 \in S_n$.
There are at most $(2 \varepsilon n + 1)$ choices for $h_{S_n}(z_0)$,
by definition of $B(K_n, h_K, \varepsilon \ell)$.
Then, since $S_n \subset \Z^m$ is connected (in the sense of graph theory),
there are less than $2^{|S_n|}$ ways to extend $h_{S_n}$ to the rest of $S_n$.
So, we must estimate $|S_n|$.
We recall that $S_n = G_n \cup U_n$,
where $G_n$ is the grid and $U_n$ the unshaded region in
Figure~\ref{f_decomp_tris}.
Since $G_n$ grows like $n^{m-1}$ while $|K_n|$ grows like $n^m$,
we have $\frac{|G_n|}{|K_n|} = \theta_{m,\varepsilon,\ell}(\tfrac{1}{n})$.
Next, the part of $K \subset \R^m$ that lies outside of any hypercube $Q_i$,
that is,~the unshaded part of the simplex domain in Figure~\ref{f_decomp_tris},
is a $\theta_m(\varepsilon)$ fraction of the total volume of $K$.
Even with discretization errors, $
	\frac{|U_n|}{|K_n|}
	= \theta_m(\varepsilon) + \theta_{m,\varepsilon,\ell}(\tfrac{1}{n})
$.
Altogether,
\[ \begin{aligned}
	\frac{1}{|K_n|} \log |\Adm(S_n)|
	&\le \frac{1}{|K_n|} \log \biggl(
		(2\varepsilon \ell n + 1) 2^{|S_n|} \biggr) \\
	&= \frac{\log(2\varepsilon\ell n + 1)}{|K_n|}
	+ \frac{\log 2}{|K_n|} \Bigl( |G_n| +|U_n| \Bigr) \\
	&= \theta_m(\varepsilon) + \theta_{m,\varepsilon,\ell}(\tfrac{1}{n}) \,.
\end{aligned} \]

Applying the three asymptotic identities above
in~\eqref{e_complex_profile_underest_min}, we derive
\[
	\Ent_{K_n} \bigl( B(K_n, h_K, \varepsilon \ell) \bigr)
	\ge \frac{1}{r} \sum_{i=1}^r \ent(s_i)
	+ \theta_m(\varepsilon) + \theta_{m,\varepsilon,\ell}(\tfrac{1}{n}) \,,
\]
which is exactly~\eqref{e_complex_profile_underest}. This completes the proof.
\end{proof}

\section{Proof of the profile theorem} \label{s_profile}

In this section we extend Theorem~\ref{thm_complex},
the profile theorem for piecewise affine height functions on simplex domains,
to general asymptotic height theorem on general domains
(subject to Assumption~\ref{a_domain}, as always).

The proof is an approximation argument, and we will need some auxiliary results.
The most helpful is the simplicial Rademacher theorem,
which states that Lipschitz functions are well-approximated
by piecewise affine functions on a simplex domain.
The other auxiliary results are about robustness
of the microscopic and macroscopic entropies
under changes in the domain and in the asymptotic height profile.

The simplicial Rademacher theorem is a general fact about Lipschitz functions.
There is nothing particular to our setting, except for the use of our term
``asymptotic height function'' instead of ``Lipschitz function.''
Related results include~\cite{Sch14}, which extends Lemma~\ref{lem_approx_tri}
from Lipschitz functions to Sobolev functions,
but weakens the approximation somewhat and is therefore not suitable for our
purposes here.
The statement of the simplicial Rademacher theorem is
adapted from Lemma~2.2 of~\cite{CKP01},
and the proof is inspired by the proof there.

\begin{lemma}[Simplicial Rademacher theorem] \label{lem_approx_tri}
Let $R \subseteq \R^m$ be a region satisfying Assumption~\ref{a_domain},
and let $h_R \in M(R, h_{\partial R})$ be an asymptotic height function on $R$.
For any $\varepsilon > 0$
and any $\ell > 0$ sufficiently small (depending on $\varepsilon$),
we may choose a simplex domain
$K = \Delta_1 \cup \dotsb \cup \Delta_k  \subseteq R$ of scale $\ell$
(see Definition~\ref{d_simpl_dom})
and a piecewise affine asymptotic height function $h_K: K \to \R$
(that is, an asymptotic height function
such that each restriction $h_K|_{\Delta_i}: \Delta_i \to \R$ is affine)
that satisfy the following properties:
\begin{enumerate}[label={(\alph*)}]
\item \label{lem_approx_tri_volume}
	$|R \setminus K| < \varepsilon$,
	where $|\cdot|$ denotes the Lebesgue measure,
	and $d_H(K, R) < \varepsilon)$,
	where $d_H$ denotes Hausdorff metric;
\item \label{lem_approx_tri_approx}
	$\max_{x \in K} |h_K(x) - h_R(x)| < \tfrac{1}{2}\varepsilon \ell$; and
\item \label{lem_approx_tri_grad}
	on at least a $(1 - \varepsilon)$ fraction of the points in $K$
	(by Lebesgue measure), the gradients $\nabla h_K(x)$ and $\nabla h_R(x)$
	agree to within $\varepsilon$; more precisely,
	$
		\frac{1}{|K|} \lvert \{
			x \in K
			\mid |\nabla h_K(x) - \nabla h_R(x)|_2 \ge \varepsilon
		\} \rvert
		< \varepsilon
	$.
\end{enumerate}
\end{lemma}

\begin{remark}
We recall that the Rademacher theorem states that a Lipschitz function $h_R$
is differentiable almost everywhere.
However $\nabla h_R$ may be poorly behaved.
The Rademacher theorem gives no control over $\nabla h_R$,
and the Lipschitz property only implies boundedness of the derivative,
not regularity.
The simplicial Rademacher theorem provides an approximation both to $h_R$
and to its derivative.
Moreover the approximating function $h_K$ has a very simple derivative,
despite the potential wildness of $\nabla h_R$.
The cost is that $h_K$ only approximates $h_R$ well on a (large) portion
of the domain rather than almost everywhere,
but for our purposes this is a good trade-off.

In fact, it is not necessary that the function $h_R$ be Lipschitz.
Almost everywhere differentiability is sufficient.
\end{remark}

Before giving the proof of Lemma~\ref{lem_approx_tri},
we state and prove the following lemma about the standard simplices
from Definition~\ref{d_simpl_one}.

\begin{lemma} \label{lem_simpl_std_basis}
Let $\Delta$ be any of the simplices $C(v,s)$ for $v \in \Z^m$ and $s \in S_m$.
The $m+1$ vertices of $\Delta$ can be labelled $x^{(0)}, \dotsc, x^{(m)}$
in such a way that, for each $i=1, \dotsc, m$,
\[
	x^{(i)} - x^{(i-1)} = e^{(s(i))} \,,
\]
where for $1 \le j \le m$, $e^{(j)}$ denotes the $j$-th standard basis vector
(i.e., all entries of $e^{(j)}$ are $0$, except the $j$-th entry, which is $1$).
\end{lemma}

\begin{remark}
We encourage the reader to keep Figure~\ref{f_c_v_s_3d}
in mind (or better, in sight) while reading this proof.
\end{remark}

\begin{proof}
For simplicity, we assume without loss of generality that $v = 0$.
We use the permutation $s$ to define a path
between vertices of the simplex $C(0, s)$
starting at $(0, \dotsc, 0)$ and ending at $(1, \dotsc, 1)$.
To construct the path, first observe that
\[
	C(0, s)
	= \bigl \{x = (x_1, \dotsc, x_m) \in [0,1]^m
		\,\big|\, x_{s(i)} \ge x_{s(j)}
		\text{ for all $i < j$} \bigr\} \,.
\]
In other words, the $s(1)$-th component of $x$
must be greater than the $s(2)$-th,
which is greater than or equal to the $s(3)$-th, and so on.
The path travels from $(0,\dotsc,0)$
along the $s(1)$-th axis to  $e_{s(1)}$,
then parallel to the $s(2)$-th axis to $e_{s(1)} + e_{s(2)}$,
and so on up to $\sum_{i=1}^m e_i = (1, \dotsc, 1)$.
Numbering the vertices of the path from $x^{(0)}$ to $x^{(m)}$
proves the lemma.
\end{proof}

Now, we are ready for the proof of Lemma~\ref{lem_approx_tri}.

\begin{proof}[Proof of Lemma~\ref{lem_approx_tri}]
Let $\ell > 0$.
We choose the simplex domain
$K = \Delta_1 \cup \dotsb \cup \Delta_k$
such that $\{\Delta_1, \dotsc, \Delta_k\}$
enumerates all simplices of scale $\ell$ (cf.~Definition~\ref{d_simpl_ell})
that are contained in $R$.
We define the asymptotic height function $h_K \in M(K)$
to agree with $h_R$ on the vertices of the simplices in $K$,
and we extend $h_K$ into the rest of each $\Delta_i$ by linear interpolation.
We will show that, once $\ell$ is small enough,
properties~\ref{lem_approx_tri_volume},~\ref{lem_approx_tri_approx},
and~\ref{lem_approx_tri_grad} from Lemma~\ref{lem_approx_tri} all hold.

\medskip

First we prove~\ref{lem_approx_tri_volume}.
The fact that $|K|$ tends to $|R|$ as $\ell \to 0$
is elementary measure theory, and we omit the proof.

Recall from~\eqref{e_hausdorff} that we define the Hausdorff metric $d_H$
in terms of the $\ell^1$ metric on $\R^m$,
for reasons explained in Remark~\ref{rem_hausdorff}.
Therefore, for the second part of~\ref{lem_approx_tri_volume},
it suffices to show that
\[
	R
	\subset K + B_{\ell^1}(0, \varepsilon)
	:= \bigl\{ x + y \,\big|\, x \in K ,\, |y|_1 < \varepsilon \bigr\} \,.
\]

We do this by constructing a subset $R' \subset R$
such that $R \subset R' + B_{\ell^1}(0, \varepsilon)$
and $R' + B_{\ell^1}(0, \alpha) \subset R$ for some $\alpha < \varepsilon$.
The latter condition ensures that, for $\ell$ small enough, $R' \subset K$.
Indeed, the $\ell^1$-diameter of a simplex of scale $\ell$ is $\le m \ell$
(i.e.\@
$
	\diam_1 \Delta
	:= \max \{ |x-y|_1 \,|\, x,y \in \Delta \}
	\le m \ell
$%
), so as long as $\ell < \frac{\alpha}{m}$,
every point $x \in R'$ belongs to a simplex $\Delta_i$ of scale $\ell$
which is part of $K$.
Therefore $R' \subset K$, so $R \subset K + B_{\ell^1}(0, \varepsilon)$
as intended.

We proceed to construct $R'$.
Since $R$ is the closure of its interior,
$R \subset \bigcup_{x \in \interior{R}} B_{\ell^1}(x, \varepsilon)$.
By compactness, choose $x_1, \dotsc, x_k$ such that
$
	R
	\subset B_{\ell^1}(x_1, \varepsilon) \cup \dotsb
		\cup B_{\ell^1}(x_k, \varepsilon)
$.
For $i=1, \dotsc, k$, let
\[
	\alpha_i
	:= \max \bigl\{ \tilde \alpha > 0 \,\big|\,
		B_{\ell^1}(x_i, \tilde \alpha) \subset R \bigr\}
\]
Since $x_i \in \interior{R}$, each $\alpha_i > 0$;
since $R$ is compact, each $\alpha_i < \infty$.
Let $\alpha = \min \{\alpha_1, \dotsc, \alpha_k, \frac{\varepsilon}{2}\}$,
and set
\[
	R'
	= \bigl\{ x \in \interior{R} \,\big|\, B_{\ell^1}(x, \alpha) \subset R \} \,.
\]

By construction, $R' + B_{\ell^1}(0, \alpha) \subset R$,
and since $x_1, \dotsc, x_k \in R'$,
$R \subset R' + B_{\ell^1}(0, \varepsilon)$.
So $R'$ satisfies the relations claimed above,
which completes the proof of~\ref{lem_approx_tri_volume}.

\medskip

For later use we strengthen
the volume estimate from~\ref{lem_approx_tri_volume}.
Choose $\ell$ smaller so that
\begin{equation} \label{lem_approx_tri_strong_volume}
	|R \setminus K|
	< \varepsilon'
	:= \min \left\{ \frac{\varepsilon}{2},
		\frac{\varepsilon}{2|R|},
		\frac{1}{2} \right\} \,.
\end{equation}
In particular, this implies that
\begin{equation} \label{e_pf_approx_tri_k_rel_r}
	|R| <  \frac{1}{1 - \varepsilon'} \, |K| \,.
\end{equation}

\medskip

Let us describe the key idea used to prove~\ref{lem_approx_tri_approx}
and~\ref{lem_approx_tri_grad}.
We consider points $x$ where $h_R$ is differentiable,
and indeed where $h_R$ is locally approximated well
by its first-order Taylor polynomial.
Once the simplices $\Delta_i$ are small enough and contain a ``good'' point $x$,
the vertices all lie close to $x$,
so we can use the Taylor polynomial
to estimate the values of $h_R$ on the vertices.
This yields the proof of~\ref{lem_approx_tri_approx}
and~\ref{lem_approx_tri_grad}.

To be more precise, we define a set $S_{\rho_0}$ of ``good'' points.
Recall that the Lipschitz function $h_R$ is almost everywhere differentiable,
by the Rademacher theorem.
Consider any point $x \in R$ at which $\nabla h_R$ exists.
Define the Taylor polynomial
\[
	L_x(y) := h_R(x) + \nabla h_R(x) \cdot (y - x) \,.
\]
By the definition of differentiability,
\[
	\lim_{y \to x} \frac{|h_R(y) - L_x(y)|}{|x-y|_2} = 0 \,,
\]
so there exists $r_0(x) > 0$ such that,
for any $y \in R$ with $|y-x|_2 < r_0(x)$,
\begin{equation} \label{e_pf_approx_tri_diff}
	| h_R(y) - L_x(y) |
	< \Bigl( \frac{\varepsilon}{4\sqrt{m}}
		\, \wedge \, \frac{\varepsilon}{2m} \Bigr)
	\; |y-x|_2 \,.
\end{equation}
(Recall that $m \in \N$ is the dimension parameter;
we could replace the parenthesized expression by $\frac{\varepsilon}{4m}$,
but the expressions $\frac{\varepsilon}{4 \sqrt m}$
and $\frac{\varepsilon}{2m}$ are useful later.)
For $\rho > 0$, define the set $S_\rho \subset R$ by
\[
	S_\rho := \bigl\{ x \in R \,\big|\, r_0(x) \ge \rho \bigr\} \,.
\]
As $\rho \to 0$, the sets $S_\rho$ increase to the full-measure subset of $R$
on which $h_R$ is differentiable.
Therefore $|S_\rho| \to |R|$ as $\rho \to 0$,
and in particular, there exists $\rho_0 > 0$ such that
\begin{equation} \label{e_pf_approx_tri_srho_rel_r}
	|S_{\rho_0}| \ge \Bigl( 1- \frac{\varepsilon'}{2} \Bigr) \, |R| \,.
\end{equation}

We choose $\ell_0 \le \frac{\rho_0}{\sqrt{m}}$.
By the Pythagorean theorem (in $m$ dimensions),
if $x,y$ are two points that lie in a simplex $\Delta_i$
and if $x \in S_{\rho_0}$,
then $|x-y|_2 \le \sqrt{m} \ell < \rho_0 \le r(x)$.
Therefore by~\eqref{e_pf_approx_tri_diff},
\begin{equation} \label{e_pf_approx_tri_approx}
	\max_{y \in \Delta_i} \, | h_R(y) - L_x(y) |
	\; < \;
	\frac{\varepsilon\ell}{4}
		\, \wedge \, \frac{\varepsilon\ell}{2\sqrt m} \,.
\end{equation}

\medskip

There are two more steps to prove~\ref{lem_approx_tri_approx}.
First, under the assumption that $x \in \Delta_i \cap S_{\rho_0}$,
we have compared $h_R|_{\Delta_i}$
to the Taylor polynomial of $h_R$ centered at $x$;
we should also compare $h_K$ to the same polynomial.
Second, we show that at least $(1-\varepsilon)k$ of the simplices
have some intersection with $S_{\rho_0}$.
Then it is straightforward to complete the proof of~\ref{lem_approx_tri_approx}.

Regarding $h_K$, recall that on the vertices $y_0, \dotsc, y_m$ of $\Delta_i$,
$h_K$ agrees with $h_R$.
Therefore by~\eqref{e_pf_approx_tri_approx},
\[
	\max_{0 \le i \le m} | h_K(y_i) - L_x(y_i) |
	< \frac{\varepsilon\ell}{4} \,.
\]
The function $y \mapsto h_K(y) - (h_R(x) + \nabla h_R(x) \cdot(y-x))$ is affine,
so $y \mapsto |h_K(y) - (h_R(x) + \nabla h_R(x) \cdot(y-x))|$ is convex.
Hence
\begin{equation} \label{e_pf_approx_tri_maxk}
	\max_{y \in \Delta_i} \, | h_K(y) - L_x(y) |
	< \frac{\varepsilon\ell}{4} \,.
\end{equation}
Therefore, if a simplex $\Delta_i$ has any intersection with $S_{\rho_0}$,
then~$h_K$ satisfies the $L^\infty$ inequality from~\ref{lem_approx_tri_approx}
over $\Delta_i$.

Now, let $k_0$ denote the number of simplices that intersect $S_{\rho_0}$.
We claim that $k_0 \ge (1-\varepsilon') k$.
Of course, since $S_{\rho_0}$ has large measure,
it must intersect many of the simplices.
More precisely, because $(k - k_0)$ is the number of simplices that do not
intersect $S_{\rho_0}$,
\[
	|S_{\rho_0}| \le |R| - (k - k_0) |\Delta_1| \,.
\]
(Recall that the simplices of scale $\ell$ are isometric,
so $|\Delta_1| = \dotsb = |\Delta_k| = \frac{\ell^m}{m!}$.)
Therefore
\[ \begin{aligned}
	k - k_0
	&\le \frac{|R| - |S_{\rho_0}|}{|\Delta_1|} \\
	&\le \frac{\varepsilon'}{2} \frac{|R|}{|\Delta_1|}
	&\qquad&\textit{(By~\eqref{e_pf_approx_tri_srho_rel_r})} \\
	&\le \frac{\varepsilon'}{2}
		\frac{(1 - \varepsilon')^{-1}|K|}{|\Delta_1|}
	&\qquad&\textit{(By~\eqref{e_pf_approx_tri_k_rel_r})} \\
	&= \frac{\varepsilon'}{2} (1 - \varepsilon')^{-1} k
	&\qquad&\textit{(Since $K = \bigcup_{i=1}^k \Delta_i$)} \\
	&\le \varepsilon' k \,,
\end{aligned} \]
since $\varepsilon' \le \frac{1}{2}$.
So, both~\eqref{e_pf_approx_tri_approx} and~\eqref{e_pf_approx_tri_maxk}
apply on $k_0 \ge (1-\varepsilon')k$ of the simplices.
We throw away the ``bad'' simplices,
at the cost of increasing $|R \setminus K|$ by at most
\[
	\varepsilon' k |\Delta_i|
	\le \left( \frac{\varepsilon}{2|R|} \right) k
		\left( \frac{|R|}{k} \right)
	= \frac{\varepsilon}{2}
	\,.
\]
This is permissible by~\eqref{lem_approx_tri_strong_volume},
since $|R \setminus K|$ was previously $\le \frac{\varepsilon}{2}$.
We have therefore proven~\ref{lem_approx_tri_approx}.

\medskip

For~\ref{lem_approx_tri_grad},
we claim that if $x \in K \cap S_{\rho_0}$,
then $|\nabla h_R(x) - \nabla h_K(x)| < \varepsilon$.
By~\eqref{e_pf_approx_tri_k_rel_r} and~\eqref{e_pf_approx_tri_srho_rel_r},
$|K \cap S_{\rho}| > (1 - \varepsilon) |R|$,
so this will suffice to prove~\ref{lem_approx_tri_grad}.

Let $x \in K \cap S_{\rho_0}$,
and fix $i$ such that $x \in \Delta_i$.
We will use Lemma~\ref{lem_simpl_std_basis} in order
to make the calculations as concrete as possible.
In particular, we label the vertices of $\Delta_i$
as $y_0, y_1, \dotsc, y_m$ in such a way that
$y_i - y_{i-1} = \ell e_{s(i)}$,
where $e_{s(i)}$ is the $s(i)$-th standard basis vector
for some permutation $s \in S_m$ of $\{1, \dotsc, m\}$.
Then by~\eqref{e_pf_approx_tri_approx},
\[
	| h_R(y_i) - L_x(y_i) | < \frac{\varepsilon\ell}{2\sqrt m}
	\quad \text{and} \quad
	| h_R(y_{i-1}) - L_x(y_{i-1}) | < \frac{\varepsilon\ell}{2\sqrt m} \,,
\]
where we recall that $L_x(y) = h_R(x) + \nabla h_R(x) \cdot (y-x)$
is the first-order Taylor polynomial of $h_R$ at $x$.
Combining these two inequalities,
\[
	\bigl| \bigl( h_R(y_i) - h_R(y_{i-1}) \bigr)
		- \bigl( \nabla h_R(x) \cdot (y_i - y_{i-1}) \bigr) \bigr|
	< \frac{\varepsilon\ell}{\sqrt m} \,.
\]
Since $y_i - y_{i-1} = \ell e_{s(i)}$,
\begin{equation} \label{e_pf_approx_tri_deriv3}
	\biggl| \frac{h_R(y_i) - h_R(y_{i-1})}{|y_i - y_{i-1}|_2}
		- \nabla h_R(x) \cdot
			\frac{y_i - y_{i-1}}{|y_i - y_{i-1}|_2} \biggr|
	< \frac{\varepsilon}{\sqrt m} \,.
\end{equation}
Because $h_K$ is the linear interpolation of $h_R$
from the vertices $y_0, \dotsc, y_m$ to the rest of $\Delta_i$,
we see that the first term
on the left-hand side of~\eqref{e_pf_approx_tri_deriv3}
is
\[
	\frac{h_R(y_i) - h_R(y_{i-1})}{|y_i - y_{i-1}|_2}
	= \frac{h_K(y_i) - h_K(y_{i-1})}{|y_i - y_{i-1}|_2}
	= \frac{\partial h_K}{\partial x_{s(i)}}(x) \,.
\]
And of course, $(y_i - y_{i-1})/|y_i - y_{i-1}|_2 = e_{s(i)}$,
so the second term on the left-hand side of~\eqref{e_pf_approx_tri_deriv3} is
\[
	\nabla h_R(x) \cdot \frac{y_i - y_{i-1}}{|y_i - y_{i-1}|_2}
	= \frac{\partial h_R}{\partial x_{s(i)}}(x) \,.
\]
The last three equations hold for all $i=1, \dotsc, m$.
Therefore we may drop the permutation $s(i)$ from the partial derivatives
and conclude that, for every $i$,
\[
	\biggl| \frac{\partial h_K}{\partial x_i}(x)
		- \frac{\partial h_R}{\partial x_i}(x) \biggr|
	< \frac{ \varepsilon}{\sqrt m}\,.
\]
Thus
\[
	\bigl| \nabla h_K(x) - \nabla h_R(x) \bigr|_2
	< \varepsilon \,.
\]
\end{proof}

\medskip

The next three lemmas regard the robustness of the macroscopic entropy
and microscopic entropy to changes in domain and asymptotic height function
As seen in the simplicial Rademacher theorem (Lemma~\ref{lem_approx_tri}),
we will change both the domain and the asymptotic height function.
As long as these changes are small enough (in the appropriate senses),
these lemmas show that the macroscopic entropy and microscopic entropy
change by a small amount.

First, we deal with robustness of the macroscopic entropy.
Because $\Ent_R$ is an integral function with continuous and bounded integrand,
robustness with respect to changes in both domain and asymptotic height function
is easy to prove by standard analytic arguments.
The main requirement is control over the change in the derivative of
the asymptotic height function, as is provided by~\ref{lem_approx_tri_grad}
from the simplicial Rademacher theorem (see Lemma~\ref{lem_approx_tri}).

\begin{lemma}[Robustness of macroscopic entropy under approximations]
\label{lem_approx_macro_ent}
Let $\varepsilon > 0$,
and let $\tilde R \subseteq R \subset \R^m$ be sets
meeting the assumptions from Assumption~\ref{a_domain},
with $|R \setminus \tilde R| < \varepsilon$.
Let $h_{\tilde R} \in M({\tilde R})$ and $h_R \in M(R)$ be such that
\[
	\Bigl| \Bigl\{ x \in \tilde R \Bigm\vert
		|\nabla h_{\tilde R}(x) - \nabla h_R(x)|_2 \ge \varepsilon
	\Bigr\} \Bigr|
	< \varepsilon \,.
\]

Then,
\[
	\Ent_R(h_R) = \Ent_{\tilde R}(h_{\tilde R}) + \theta_m(\varepsilon) \,.
\]
\end{lemma}

\begin{proof}
Recall that $\Ent_R(h_R) = \frac{1}{|R|} \int_R \ent(\nabla h_R(x)) \, dx$
and $
	\Ent_{\tilde R}(h_{\tilde R})
	= \frac{1}{|\tilde R|} \int_{\tilde R}
		\ent(\nabla h_{\tilde R}(x)) \, dx$.
Split $R$ into three parts.

The set $
	\{x \in \tilde R \,\vert\,
		|\nabla h_{\tilde R}(x) - \nabla h_R(x)|_2 \ge \varepsilon\}
$
has measure less than $\varepsilon$ by hypothesis.
Since $\ent(s)$ is bounded (see Observation~\ref{obs_surf_tens_bdd}),
the contribution of the points in this set to
$\Ent_{\tilde R}(h_{\tilde R})$ is within $\theta(\varepsilon)$
of the contribution to $\Ent_R(h_R)$.

Likewise, the set $R \setminus \tilde R$ has measure at most $\varepsilon$,
so the contribution to $\Ent_R(h_R)$ is $\theta(\varepsilon)$.
Of course, this set does not contribute to $\Ent_{\tilde R}(h_{\tilde R})$.

Finally, for the remaining points $x$,
$|\nabla h_{\tilde R}(x) - \nabla h_R(x)|_2 < \varepsilon$.
Since $\ent(s)$ is uniformly continuous on its domain $s \in [-1,1]^m$,
we have
$|\ent(\nabla h_{\tilde R}(x)) - \ent(\nabla h_R(x))| < \theta_m(\varepsilon)$.
Since the integrands differ by at most $\theta_m(\varepsilon)$
and since the integrals are normalized
by $\frac{1}{|R|}$ and $\frac{1}{|\tilde R|}$,
the contribution from this third part of the domain
is also $\theta_m(\varepsilon)$.
\end{proof}

\medskip

Now, we turn to the microscopic entropy.
Here it is easier to record two separate robustness results.
The first is robustness with respect to changes in the asymptotic height
function, and the second is robustness with respect to changes in domain.
Robustness with respect to changes in the asymptotic height function
comes immediately from the definition of the balls $B(R_n, h_R, \varepsilon)$.

\begin{lemma}[Robustness of microscopic entropy under change in profile]
\label{lem_approx_micro_ent_ahf}
Let $\varepsilon > 0$ and $n \in \N$.
Let $R \subset \R^m$ satisfy Assumption~\ref{a_domain},
and let $R_n \subset \Z^m$ satisfy $\tfrac{1}{n} R_n \subset R$.
Let $h_R, \tilde h_R \in M(R)$ be two asymptotic height functions
such that $\max_{x \in R} |h_R(x) - \tilde h_R(x)| \le \varepsilon$.
Then,
\[
	\Ent_{R_n} \bigl( B(R_n, h_R, 2 \varepsilon) \bigr)
	\le \Ent_{R_n} \bigl( B(R_n, \tilde h_R, \varepsilon) \bigr) \,.
\]
\end{lemma}

\begin{proof}
It suffices to notice that
$B(R_n, h_R, 2 \varepsilon) \supseteq B(R_n, \tilde h_R, \varepsilon)$.
This follows from the triangle inequality:
for any $h_{R_n} \in B(R_n, \tilde h_R, \varepsilon)$ and any $z \in R_n$,
\[ \begin{aligned}
	\bigl| \tfrac{1}{n} h_{R_n}(z) - h_R(\tfrac{1}{n}z) \bigr|
	&\le \bigl| \tfrac{1}{n} h_{R_n}(z) - \tilde h_R(\tfrac{1}{n}z) \bigr|
	+ \bigl| \tilde h_R(\tfrac{1}{n}z) - h_R(\tfrac{1}{n}z) \bigr| \\
	&\le \varepsilon + \varepsilon \,.
\end{aligned} \]
\end{proof}

\medskip

More care is needed to state and prove robustness of the microscopic entropy
with respect to changes in domain.
The main idea is straightforward.
Given two microscopic domains $\tilde R_n \subset R_n$,
we will consider the extension map from
$B(\tilde R_n, h_R, \varepsilon)$ to $B(R_n, h_R, \varepsilon)$
and the restriction map in the opposite direction.
So long as every height function on the smaller domain admits an extension,
we have $|B(\tilde R_n, h_R, \varepsilon)| \le |B(R_n, h_R, \varepsilon)|$.
In the opposite direction, the restriction map is not generally an injection
but the pre-images are not too large; at most $2^N$ height functions on $R_n$
restrict to the any specific height function on $\tilde h_R$,
where $N = |R_n| \setminus |\tilde R_n|$.

Most of the complications arise in the extension step.
Our primary extension result,
namely the Kirszbraun theorem (Theorem~\ref{th_kirszbraun}),
is insufficient.
It states that a height function
$h_{\tilde R_n} \in B(\tilde R_n, h_R, \varepsilon)$
admits an extension to $R_n$,
but that extension is not necessarily in $B(R_n, h_R, \varepsilon)$.
There are two ways forward: to prove a stronger extension theorem specialized
to the problem under consideration,
or to leverage the Lipschitz property to control the extension.
For greater generality, we prefer the second method.
However there are a few difficulties:
the Kirszbraun theorem is subtle when the asymptotic height profile has
$|\nabla h_K|_\infty$ in part of the region,
and the extension cannot generally be kept
within distance $\varepsilon$ of $h_R$.
This leads to the following somewhat complex formulation.

\begin{lemma}[Robustness of microscopic entropy under domain approximations]
\label{lem_approx_micro_ent_dom}
Let $\varepsilon \in (0, 1]$ and $n \in \N$ with $n \ge \tfrac{1}{\varepsilon}$.
Let $\tilde R \subset R \subset \R^m$ and $\tilde R_n \subset R_n \subset \Z^m$
satisfy these assumptions:
\begin{align}
	\tfrac{1}{n} R_n &\subset R \,, &
	\tfrac{1}{n} \tilde R_n &\subset \tilde R \,, \\
	d_H(\tfrac{1}{n} R_n, R) &= \theta(\varepsilon) \,, &
	d_H(\tfrac{1}{n} \tilde R_n, \tilde R) &= \theta(\varepsilon) \,, \\
\end{align}
Additionally, assume that
\[
	\frac{|R|}{|\tilde R|} = 1 + \theta(\varepsilon) \,.
\]

Let $h_R \in M(R)$ be an asymptotic height function
with $\Lip(h_R) \le 1 - c \varepsilon$ for some fixed $c \in (0, 1]$.
Then,
\[ \begin{aligned}
	\hskip3em&\hskip-3em
	\Ent_{\tilde R_n} \bigl( B(\tilde R_n, h_R, \varepsilon \bigr)
	+ \theta(\varepsilon) + \theta_\varepsilon(\tfrac{1}{n}) \\
	&\le \Ent_{R_n} \bigl( B(R_n, h_R, \varepsilon) \bigr) \\
	&\le \Ent_{\tilde R_n} \bigl(
		B(\tilde R_n, h_R, \tfrac{c}{3} \varepsilon^2) \bigr)
	+ \theta(\varepsilon) + \theta_\varepsilon(\tfrac{1}{n}) \,.
\end{aligned} \]
\end{lemma}

\begin{remark}
The assumptions in Lemma~\ref{lem_approx_micro_ent_dom} quantify the
imprecise statements that $\tilde R$, $\tfrac{1}{n} R_n$, and $\tfrac{1}{n} R_n$
respectively approximate $R$, $R$, and $\tilde R$ from inside.

If we take the simplicial approximation $K$ from Lemma~\ref{lem_approx_tri}
to be $\tilde R$, and its discretization
$K_n := \{ z \in \Z^m \,|\, \tfrac{1}{n} z \in K \}$ to be $\tilde R_n$,
and if we recall the Assumption~\ref{a_domain} about $R$ and $R_n$,
then the hypotheses of Lemma~\ref{lem_approx_micro_ent_dom} are satisfied,
and moreover, we may replace all instance of $\theta(\varepsilon)$
in the conclusion by $\theta(\varepsilon \ell)$.
\end{remark}

\begin{proof}
First, we prove the following inequalities:
\begin{equation} \label{e_approx_micro_ent_undercount}
	\bigl| B(R_n, h_R, \varepsilon) \bigr|
	\le (2^{|R_n|})^{\theta(\varepsilon) + \theta_\varepsilon(\tfrac{1}{n})}
		\bigl| B(\tilde R_n, h_R, \varepsilon) \bigr|
\end{equation}
and
\begin{equation} \label{e_approx_micro_ent_overcount}
	\bigl| B(R_n, h_R, \varepsilon) \bigr|
	\ge \bigl| B(\tilde R_n, h_R, \tfrac{c}{3} \varepsilon^2) \bigr| \,.
\end{equation}

For~\eqref{e_approx_micro_ent_undercount},
we note that every height function
$h_{R_n} \in B(R_n, h_R, \varepsilon)$
restricts to a height function
$h_{R_n}|_{\tilde R_n} \in B(\tilde R_n, h_R, \varepsilon)$.
Then, we claim that
\begin{equation} \label{e_approx_micro_ent_rn}
	|R_n| n^{-d}
	= \bigl( 1 + \theta(\varepsilon) \bigr) |R|
\end{equation}
and
\begin{equation} \label{e_approx_micro_ent_tilde_rn}
	|\tilde R_n| n^{-d}
	= \bigl( 1 + \theta(\varepsilon) \bigr) |\tilde R| \,.
\end{equation}
To justify~\eqref{e_approx_micro_ent_rn} we argue as follows.
Consider the continuum region
$R_n^\square := \bigcup_{z \in \R_n} ([0,\tfrac{1}{n}]^d + \tfrac{1}{n} z)$,
i.e.\@ the union of hypercubes of side length~$\tfrac{1}{n}$
translated by the points in $\tfrac{1}{n} R_n$.
Clearly $R_n^\square$ has Lebesgue measure equal to $|R_n| n^{-d}$,
and (like $\tfrac{1}{n} R_n$)
satisfies $d_H(R_n^\square, R) = \theta(\varepsilon)$.
This implies~\eqref{e_approx_micro_ent_rn}.
Equation~\eqref{e_approx_micro_ent_tilde_rn} is analogous.
Further arithmetic yields the equation
\begin{equation} \label{e_approx_micro_ent_ratio}
	\frac{|R_n|}{|\tilde R_n|}
	= \bigl( 1 + \theta(\varepsilon) \bigr)
		\bigl( 1 + \theta_\varepsilon(\tfrac{1}{n}) \bigr) \,,
\end{equation}
and then
\[
	|R_n \setminus \tilde R_n|
	= |R_n|
	\bigl( \theta(\varepsilon) + \theta_\varepsilon(\tfrac{1}{n}) \bigr) \,.
\]

Therefore the restriction map from $B(R_n, h_R, \varepsilon)$
to $B(\tilde R_n, h_R, \varepsilon)$ is at most
$(2^{|R_n|})^{\theta(\varepsilon) + \theta_\varepsilon(\tfrac{1}{n})}$-to-$1$.
Inequality~\eqref{e_approx_micro_ent_undercount} follows immediately.

\medskip

Now let us turn to~\eqref{e_approx_micro_ent_overcount}.
We want an injection
from $B(\tilde R_n, h_R, \tfrac{c}{3} \varepsilon^2)$
into $B(R_n, h_R, \varepsilon)$.
Fix a function
$h_{\tilde R_n} \in B(\tilde R_n, h_R, \tfrac{c}{3} \varepsilon^2)$;
we will construct an extension $h_{R_n} \in M(R_n)$.

Let
\[
	R_n'
	:= \big\{ z \in R_n \,\big|\,
		d(z, \tilde R_n) > \tfrac{\varepsilon}{3} n \bigr\} \,,
\]
where
\[
	d(z, \tilde R_n)
	:= \min_{\tilde z \in \tilde R_n} |z - \tilde z|_1 \,.
\]
For $z \in R_n'$, we arrange for
the extension to satisfy $|h_{R_n}(z) - n h_R(\tfrac{1}{n}z)| \le 1$.
When $n h_R(\tfrac{1}{n}z)$ is not an integer,
or is an integer but has the same parity as $z$,
this inequality uniquely determines the value of $h_{R_n}(z)$.
In the remaining case, there are two candidate values;
we arbitrarily choose to ``round down'' to the lower value.
Later it is important that we consistently round down (or up).

Let us check the hypotheses of the Kirszbraun theorem.
If $\tilde z \in \tilde R_n$ and if $z \in R_n'$,
then $|\tilde z - z|_1 > \tfrac{\varepsilon}{3} n$.
Therefore
\[ \begin{aligned}
	\hskip3em&\hskip-3em
	\bigl| h_{\tilde R_n}(\tilde z) - h_{R_n}(z) \bigr| \\
	&\le \bigl| h_{\tilde R_n}(\tilde z)
		- n h_R(\tfrac{1}{n} \tilde z) \bigr|
	+ n \bigl| h_R(\tfrac{1}{n} \tilde z) - h_R(\tfrac{1}{n} z) \bigr| \\
	&\qquad + \bigl| n h_R(\tfrac{1}{n} z) - h_{R_n}(z) \bigr| \\
	&\le \tfrac{c}{3} \varepsilon^2 n
		+ (1 - c \varepsilon) |\tilde z - z|_1
		+ 1 \\
	&< c \varepsilon |\tilde z - z|_1
		+ (1 - c \varepsilon) |\tilde z - z|_1
		+ 1 \\
	&< |\tilde z - z|_1 + 1 \\
	&\le |\tilde z - z|_1 \,.
\end{aligned} \]

The argument for points $z_1, z_2 \in R_n'$
is similar to the arguments made in Section~\ref{ss_canon_hf}.
By the triangle inequality,
$|h_{R_n}(z_1) - h_{R_n}(z_2)| \le |z_1 - z_2|_1 + 2$.
Equality holds only if both $nh_R(\tfrac{1}{n} z_1)$
and $nh_R(\tfrac{1}{n} z_2)$ are integers of the same parity as $z_1$ and $z_2$,
respectively.
In this case $h_{R_n}$ is rounded down at both points,
so the Kirszbraun inequality is still satisfied.

So, there exists an extension $h_{R_n}$ of $h_{\tilde R_n}$
such that $|h_{R_n}(z) - h_R(\tfrac{1}{n}z)| \le 1$ for $z \in R_n'$.
We claim that $h_{R_n} \in B(R_n, h_R, \varepsilon)$.
Since $\tfrac{c}{3} \varepsilon^2 \le \varepsilon$
and since $1 \le \varepsilon n$,
it suffices to consider points $z \in R_n \setminus R_n'$.
Fix such a $z$.
By the definition of $R_n'$, there exists $\tilde z \in \tilde R_n$
such that $|\tilde z - z| \le \tfrac{\varepsilon}{3} n$.
Note that $\tfrac{c}{3} \varepsilon^2 \le \tfrac{\varepsilon}{3}$,
since $c, \varepsilon \le 1$.
By the Lipschitz property of $h_R$ and $h_{R_n}$,
\[ \begin{aligned}
	h_{R_n}(z)
	&\le h_{R_n}(\tilde z) + \tfrac{\varepsilon}{3} n \\
	&= h_{\tilde R_n}(\tilde z) + \tfrac{\varepsilon}{3} n \\
	&\le n h_R(\tfrac{1}{n} \tilde z) + \tfrac{\varepsilon}{3} n
		+ \tfrac{c}{3} \varepsilon^2 n \\
	&\le n h_R(\tfrac{1}{n} \tilde z) + \tfrac{2\varepsilon}{3} n \\
	&\le n h_R(\tfrac{1}{n} z) + \varepsilon n \,.
\end{aligned} \]

By symmetry, $h_{R_n}(z) \ge n h_R(\tfrac{1}{n} z) - \varepsilon n$,
and so $h_{R_n} \in B(R_n, h_R, \varepsilon)$.
This extension process defines an injection
from $B(\tilde R_n, h_R, \tfrac{c}{3} \varepsilon^2)$
into $B(R_n, h_R, \varepsilon)$,
which proves~\eqref{e_approx_micro_ent_overcount}.

\medskip

Finally, we derive the conclusion from~\eqref{e_approx_micro_ent_undercount}
and~\eqref{e_approx_micro_ent_overcount} by taking logarithms and normalizing,
using~\eqref{e_approx_micro_ent_ratio} to account for the difference in
normalizing factors $-\frac{1}{|R_n|}$ and $-\frac{1}{|\tilde R_n|}$.
\end{proof}

\medskip

Now, let us prove the profile theorem (Theorem~\ref{th_profile}).
The main idea is straightforward:
we approximate $h_R$ by a piecewise affine function
(given by the simplicial Rademacher theorem, i.e.~Lemma~\ref{lem_approx_tri}),
for which we have already proven the simplicial profile theorem
(Theorem~\ref{thm_complex}).
Then we use the robustness results (Lemma~\ref{lem_approx_macro_ent},
Lemma~\ref{lem_approx_micro_ent_ahf}, and Lemma~\ref{lem_approx_micro_ent_ahf})
to deduce the profile theorem for $h_R$.
However, in order to apply Lemma~\ref{lem_approx_micro_ent_ahf}
we must first reduce to the case where the Lipschitz constant
$\Lip(h_R) := \inf \{\lambda > 0 \,|\, \text{$h_R$ is $\lambda$-Lipschitz} \}$
is strictly less than $1$.

\begin{proof}[Proof of Theorem~\ref{th_profile}]
First we reduce to the case where $\Lip(h_R) \le 1 - c \delta$,
for a constant $c > 0$ depending only on the domain $R$.
Then, we reduce to the piecewise affine case of Theorem~\ref{thm_complex}.

\medskip
\textbf{Reduction to $\Lip(h_R) \le 1 - c \delta$.}
By translation invariance, we may assume that there exists $x_0 \in R$
with $h_R(x_0) = 0$.
Set $c = \frac{1}{2 \diam_1 R} \wedge 1$, and define
\[
	\tilde h_R := (1 - c \delta) h_R \,.
\]

Then $\Lip(\tilde h_R) \le 1 - c \delta$,
and for all $x \in R$, both $|h_R(x) - \tilde h_R(x)| \le \tfrac{\delta}{2}$
and $|\nabla h_R(x) - \nabla \tilde h_R(x)| \le c \delta$.
Assume that the conclusion holds for $\tilde h_R$, i.e.~
\begin{equation} \label{e_profile_reduxlip_assn}
	\Ent_R(\tilde h_R)
	= \Ent_{R_n} \bigl( B(R_n, \tilde h_R, \delta) \bigr)
		+ \theta(\delta) + \theta_\delta(\tfrac{1}{n}) \,.
\end{equation}

We make two calculations. First,
\begin{align}
	\Ent_R(h_R)
	&\le \Ent_R(\tilde h_R) + \theta(\delta)
	&&\textit{(by Lemma~\ref{lem_approx_macro_ent})} \\
	&\le \Ent_{R_n} \bigl( B(R_n, \tilde h_R, 2 \delta) \bigr)
		+ \theta(\delta) + \theta_\delta(\tfrac{1}{n})
	&&\textit{(by~\eqref{e_profile_reduxlip_assn})} \\
	&\le \Ent_{R_n} \bigl( B(R_n, h_R, \delta) \bigr)
		+ \theta(\delta) + \theta_\delta(\tfrac{1}{n})
	&&\textit{(by Lemma~\ref{lem_approx_micro_ent_ahf})} \,.
\end{align}

Second,
\begin{align}
	\Ent_R(h_R)
	&\ge \Ent_R(\tilde h_R) + \theta(\delta)
	&&\textit{(by Lemma~\ref{lem_approx_macro_ent})} \\
	&\ge \Ent_{R_n} \bigl( B(R_n, \tilde h_R, \tfrac{1}{2} \delta) \bigr)
		+ \theta(\delta) + \theta_\delta(\tfrac{1}{n})
	&&\textit{(by~\eqref{e_profile_reduxlip_assn})} \\
	&\ge \Ent_{R_n} \bigl( B(R_n, h_R, \delta) \bigr)
		+ \theta(\delta) + \theta_\delta(\tfrac{1}{n})
	&&\textit{(by Lemma~\ref{lem_approx_micro_ent_ahf})} \,.
\end{align}

So once we prove~\eqref{e_profile_reduxlip_assn}
with the extra hypothesis that $\Lip(\tilde h_R) \le 1 - c\delta$,
the general result follows.

\medskip

\textbf{Reduction to piecewise linear height functions.}
We will apply Lemma~\ref{lem_approx_tri}
to derive a simplex domain $K$ and a piecewise linear height function $h_K$
approximating $R$ and $h_R$, then appeal to Theorem~\ref{thm_complex}.
In so doing we introduce two parameters:
$\varepsilon$, which controls how well $K$ approximates $R$,
and $\ell$, which controls the size of the simplices in $K$.
There are a few important properties of $\varepsilon$ and $\ell$.
First, $\delta = \varepsilon \ell$,
so there is actually only one degree of freedom.
Second, $\ell$ must be chosen to be sufficiently small,
as is required by the simplicial Rademacher theorem
(see Lemma~\ref{lem_approx_tri}).
Third, as $\delta \to 0$ we must have $\varepsilon \to 0$,
so that $\theta(\varepsilon) = \theta(\delta)$.

Let us describe explicitly how we choose $\varepsilon$ and $\ell$
satisfying these constraints.
We fix a sequence $\varepsilon_k \downarrow 0$
(e.g.\@ $\varepsilon_k = \tfrac{1}{k}$),
and for each $k$ set
\[
	\ell_k
	:= \tfrac{1}{2} \sup \bigl\{ \ell > 0 \,\big|\,
		\text{Lemma~\ref{lem_approx_tri} applies
			with $\varepsilon = \varepsilon_k$} \bigr\} \,.
\]

We call attention to the fact that Lemma~\ref{lem_approx_tri}
is monotone in $\ell$.
In particular, for any $\ell \le \ell_k$,
the conclusion of the lemma holds for $(\varepsilon_k, \ell)$.
Now, let
\[
	\delta_k :=
	\varepsilon_k \ell_k
	\; \wedge \;
	\tfrac{1}{2} \delta_{k-1} \,.
\]

The sequence $\delta_k$ is decreases to $0$,
so $\bigcup_{k=1}^\infty (\delta_{k+1}, \delta_k]$
is a non-trivial half-open interval with left endpoint at $0$.
We assume that $\delta$ lies in this interval.
Fix $k$ such that $\delta \in (\delta_{k+1}, \delta_k]$,
and set $\varepsilon = \varepsilon_k$
and $\ell = \tfrac{\delta}{\varepsilon}$.
Then $\delta = \varepsilon \ell$ by construction,
and as noted above $\ell$ is small enough
that the simplicial Rademacher theorem (see Lemma~\ref{lem_approx_tri}) applies.
As $\delta \to 0$ necessarily $k \to \infty$,
so $\varepsilon \to 0$ as desired.
Therefore this choice of $\varepsilon$ and $\ell$ satisfies our criteria.
Per Lemma~\ref{lem_approx_tri}, the corresponding simplex domain $K$
and piecewise affine asymptotic height function $h_K$ satisfy
\begin{equation} \label{e_profile_simpl_dom}
	|R \setminus K| < \varepsilon
	\quad \text{and} \quad
	d_H(K, R) < \varepsilon \,,
\end{equation}
\begin{equation} \label{e_profile_simpl_func}
	\max_{x \in K} |h_K(x) - h_R(x)| < \varepsilon \ell = \delta \,
\end{equation}
and
\begin{equation} \label{e_profile_simpl_grad}
	\frac{1}{|K|} \bigl\{ x \in K \,\big|\,
		|\nabla h_K(x) - \nabla h_R(x)|_2 \ge \varepsilon \bigr\}
	< \varepsilon \,.
\end{equation}

All that is left is to apply
the simplicial profile theorem (see Theorem~\ref{thm_complex})
and the robustness results (see Lemma~\ref{lem_approx_macro_ent},
Lemma~\ref{lem_approx_micro_ent_ahf}, and Lemma~\ref{lem_approx_micro_ent_ahf}).
In one direction, we have
\begin{align}
	\Ent_R(h_R)
	&\le \Ent_K(h_K) + \theta(\varepsilon)
	&&\textit{(Lemma~\ref{lem_approx_macro_ent}
		and \eqref{e_profile_simpl_grad})} \\
	&\le \Ent_{K_n} \bigl( B(K_n, h_K, \varepsilon \ell) \bigr)
	&&\textit{(Theorem~\ref{thm_complex})} \\
	&\quad+ \theta(\varepsilon) + \theta_{\varepsilon,\ell}(\tfrac{1}{n}) \\
	&\le \Ent_{K_n} \bigl( B(K_n, h_R, \tfrac{1}{2} \varepsilon \ell) \bigr)
	&&\textit{(Lemma~\ref{lem_approx_micro_ent_ahf}
		and~\eqref{e_profile_simpl_func})} \\
	&\quad+ \theta(\varepsilon) + \theta_{\varepsilon,\ell}(\tfrac{1}{n}) \\
	&\le \Ent_{R_n} \bigl( B(R_n, h_R, \tfrac{1}{2} \varepsilon \ell) \bigr)
	&&\textit{(Lemma~\ref{lem_approx_micro_ent_dom}
		and \eqref{e_profile_simpl_dom})} \\
	&\quad+ \theta(\varepsilon) + \theta_{\varepsilon,\ell}(\tfrac{1}{n}) \\
	&= \Ent_{R_n} \bigl( B(R_n, h_R, \tfrac{1}{2} \delta) \bigr)
	&&\textit{(choice of $\varepsilon$ and $\ell$)} \\
	&\quad+ \theta(\delta)
		+ \theta_\delta(\tfrac{1}{n}) \,.
\end{align}

By taking $\delta' = \tfrac{1}{2} \delta$, this yields
\begin{equation} \label{e_profile_simpl_le}
	\Ent_R(h_R)
	\le \Ent{R_n} \bigl( B(R_n, h_R, \delta') \bigr)
		+ \theta(\delta') + \theta_{\delta'}(\tfrac{1}{n}) \,.
\end{equation}

\smallskip

In the other direction,
\begin{align}
	\Ent_R(h_R)
	&\ge \Ent_K(h_K) + \theta(\varepsilon)
	&&\textit{(Lemma~\ref{lem_approx_macro_ent}
		and \eqref{e_profile_simpl_grad})} \\
	&\ge \Ent_{K_n} \bigl( B(K_n, h_K, \varepsilon \ell) \bigr)
	&&\textit{(Theorem~\ref{thm_complex})} \\
	&\quad+ \theta(\varepsilon) + \theta_{\varepsilon,\ell}(\tfrac{1}{n}) \\
	&\ge \Ent_{K_n} \bigl( B(K_n, h_R, 2 \varepsilon \ell) \bigr)
	&&\textit{(Lemma~\ref{lem_approx_micro_ent_ahf}
		and~\eqref{e_profile_simpl_func})} \\
	&\quad+ \theta(\varepsilon) + \theta_{\varepsilon,\ell}(\tfrac{1}{n}) \\
	&\ge \Ent_{R_n} \bigl(
		B(R_n, h_R, \tfrac{6}{c} (\varepsilon \ell)^{1/2}) \bigr)
	&&\textit{(Lemma~\ref{lem_approx_micro_ent_dom}
		and \eqref{e_profile_simpl_dom})} \\
	&\quad+ \theta(\varepsilon) + \theta_{\varepsilon,\ell}(\tfrac{1}{n}) \\
	&= \Ent_{R_n} \bigl( B(R_n, h_R, \tfrac{6}{c} \delta^{1/2}) \bigr)
	&&\textit{(choice of $\varepsilon$ and $\ell$)} \\
	&\quad+ \theta(\delta)
		+ \theta_\delta(\tfrac{1}{n}) \,.
\end{align}

By taking $\delta'' = \tfrac{6}{c} \delta^{1/2}$, this yields
\begin{equation} \label{e_profile_simpl_ge}
	\Ent_R(h_R)
	\le \Ent{R_n} \bigl( B(R_n, h_R, \delta'') \bigr)
		+ \theta(\delta') + \theta_{\delta''}(\tfrac{1}{n}) \,.
\end{equation}

\smallskip

Combining~\eqref{e_profile_simpl_le} and~\eqref{e_profile_simpl_ge}
completes the proof of the profile theorem.
\end{proof}

\section{Proof of the variational principle} \label{s_varnprin}

Besides the profile theorem (Theorem~\ref{th_profile}),
the proof of the variational principle (Theorem~\ref{th_varnprin})
relies on compactness of the space of asymptotic height functions.
For robustness, we give a proof that does not assume
that the macroscopic entropy functional admits a minimum.
Note that the existence of such a minimizer is standard
as soon as the local surface tension is convex and bounded below;
see for example Section~2 of~\cite{CKP01} or~\cite{She05}.
However, for greater generality we work with
the infimum of the macroscopic entropy
and we do not assume that a minimizer exists.
At any rate, it will be necessary to deal with infima (rather than minima)
later when proving the large deviations principle.

\begin{proof}[Proof of Theorem~\ref{th_varnprin}]
First, we shall prove that
\begin{equation} \label{e_variational_overest}
	\inf_{h_R \in M(R, h_{\partial R})} \Ent_R(h_R)
	\ge \Ent_{R_n} \bigl( M(R_n, h_{\partial R_n}, \delta) \bigr)
	+ \theta(\delta) + \theta_\delta(\tfrac{1}{n})
\end{equation}
via undercounting the number of height functions
in $M(R_n, h_{\partial R_n}, \delta)$.
The strategy is simple: we only count those height functions
that are close to a ``near-minimizer'' of the  macroscopic entropy.
If we assume that a minimizer exists,
i.e.\@ that there exists $h_R^\texttt{min} \in M(R, h_{\partial R})$
such that
\[
	\Ent_R(h_R^\texttt{min})
	= \inf_{h_R \in M(R, h_{\partial R})} \Ent_R(h_R) \,,
\]
then the following proof suffices.
For any $\delta > 0$ and $n \in \N$,
the Definition~\ref{d_ht_func_sets} implies that
\[
	B(R_n, h_R^\texttt{min}, \delta)
	\subseteq M(R_n, h_{\partial R_n}, 2\delta) \,.
\]
It follows immediately that
\[
	\Ent_{R_n} \bigl( B(R_n, h_R^\texttt{min}, \delta) \bigr)
	\ge \Ent_{R_n} \bigl( M(R_n, h_{\partial R_n}, 2\delta) \bigr) \,,
\]
so after applying the profile theorem and replacing $2\delta$ by $\delta$,
\[
	\Ent_{R_n} \bigl( M(R_n, h_{\partial R_n}, \delta) \bigr)
	\le \inf_{h \in M(R_, h_{\partial R})} \Ent_R(h)
	+ \theta(\delta) + \theta_\delta(\tfrac{1}{n}) \,.
\]

However, as mentioned above we want to give a proof that does not rely
on the existence of a minimizer.
This idea is also important for proving the large deviation principle below
(see the paragraphs following~\eqref{e_ldp_lower} below).
The first step is to replace $h_R^\texttt{min}$
by a sequence of approximations, say $h_R^{(k)}$ satisfying
\[
	\Ent_R(h_R^{(k)})
	\le \inf_{h_R \in M(R, h_{\partial R})} \Ent(h_R) + \tfrac{1}{k}
	\quad \text{for $k \in \N$} \,.
\]

Also let $\theta^{(k)}(\delta)$ and $\theta_\delta^{(k)}(\tfrac{1}{n})$
denote the $\theta$ terms from the profile theorem (Theorem~\ref{th_profile})
for the height function $h_R^{(k)}$.
At this point one may be tempted to simply take the limit $k \to \infty$
for fixed $\delta$ and $n$.
The problem is that the sequence $\theta^{(k)}(\delta)$
is not necessarily controlled as $k$ goes to infinity,
and could in general diverge for any fixed $\delta > 0$,
and likewise for $\theta_\delta^{(k)}(\frac{1}{n})$.
To correct this, we proceed as follows.

Let $\delta_0 = +\infty$ and $n_0 = 0$.
For $k \ge 1$, choose $\delta_k$ such that
\[
	0 < \delta_k \le \tfrac{1}{2} \delta_{k-1}
	\quad \text{and} \quad
	\theta^{(k)}(\delta_k) \le \frac{1}{k} \,.
\]

Now let $\delta > 0$ and $n \in \N$ be given.
Fix $k$ such that $\delta \in (\delta_k, \delta_{k-1}]$.
Note that this is possible since $\delta_k \le \tfrac{1}{2} \delta_{k-1}$
forces $\delta_k \to 0$ as $k \to \infty$.
Since $\delta_k < \delta$, we have
\[
	\Ent_{R_n} \bigl( B(R_n, h_{\partial R_n}, \delta) \bigr)
	\le \Ent_{R_n} \bigl( B(R_n, h_{\partial R_n}, \delta_k) \bigr) \,.
\]

By the profile theorem (Theorem~\ref{th_profile}) applied to $h_R^{(k)}$,
\[
	\Ent_{R_n} \bigl( B(R_n, h_{\partial R_n}, \delta_k) \bigr)
	\le \Ent_R(h_R^{(k)}) + \theta^{(k)}(\delta_k)
		+ \theta_{\delta_k}^{(k)}(\tfrac{1}{n}) \,.
\]
By choice of $h_R^{(k)}$, we have
\[
	\Ent_R(h_R^{(k)})
	\le \inf_{h_R \in M(R, h_{\partial R})} \Ent_R(h_R) + \frac{1}{k} \,.
\]
By choice of $\delta_k$,
\[
	\theta^{(k)}(\delta_k) \le \frac{1}{k} \,,
\]
Finally, since $k$ and $\delta_k$ are determined from $\delta$,
\[
	\theta_{\delta_k}^{(k)}(\tfrac{1}{n})
	= \theta_\delta(\tfrac{1}{n})
	\quad \text{and} \quad
	\frac{1}{k} = \theta(\delta) \,.
\]
Putting it all together, we have
\begin{equation} \begin{aligned}
	\Ent_{R_n} \bigl( B(R_n, h_{\partial R_n}, \delta) \bigr)
	&\le \inf_{h_R \in M(R, h_{\partial R})} \Ent_R(h_R)
		+ \frac{2}{k} + \theta_{\delta_k}^{(k)}(\tfrac{1}{n}) \\
	&= \inf_{h_R \in M(R, h_{\partial R})} \Ent_R(h_R)
		+ \theta(\delta) + \theta_\delta(\tfrac{1}{n}) \,.
\end{aligned} \end{equation}

\medskip

Now we prove the reverse inequality, namely
\begin{equation} \label{e_variational_underest}
	\inf_{h_R \in M(R, h_{\partial R})} \Ent_R(h_R)
	\le \Ent_{R_n} \bigl( M(R_n, h_{\partial R_n}, \delta) \bigr)
		+ \theta(\delta) + \theta_\delta(\tfrac{1}{n}) \,.
\end{equation}

Let $\varepsilon > 0$.
For each $h_R \in M(R, h_{\partial R}, 2\delta)$,
by the profile theorem (Theorem~\ref{th_profile})
\begin{equation} \label{e_variational_limit} \begin{aligned}
	\Big| \Ent_{R_n} \bigl( B(R_n, h_R, \delta) \bigr) - \Ent_R(h_R) \Bigr|
	\,&<\, \theta(\delta) + \theta_\delta(\tfrac{1}{n}) \,.
\end{aligned} \end{equation}
For each $h_R \in M(R, h_{\partial R})$,
fix $\eta(h_R) > 0$ such that
the $\theta(\delta)$ term in~\eqref{e_variational_limit} satisfies
\begin{equation} \label{e_variational_def_eta}
	\theta(\eta(h_R)) \le \frac{\varepsilon}{2} \,.
\end{equation}

Recall from Definition~\ref{d_asymp_ht_func_sets} that
\begin{equation} \label{e_variational_compact} \begin{aligned}
	\hskip3em&\hskip-3em
	M(R, h_{\partial R}, 2\delta) \\
	&:= \bigl\{ h_R \in \Lip(R) \, \big| \, \forall x \in \partial R, \,
		\, | h(x) - h_{\partial R}(x) | \le 2 \delta \bigr\} \,.
\end{aligned} \end{equation}
This set is compact as an easy consequence of the Arzel\`a--Ascoli theorem.
Choose $h_R^{(1)} \in M(R, h_{\partial R})$ such that
\[
	\Ent_R(h_R^{(1)})
	\le \inf_{h_R \in M(R, h_{\partial R})} \Ent_R(h_R) + \varepsilon \,,
\]
and pick $h_R^{(2)}, \dotsc, h_R^{(k)}$
so that the union $\bigcup_{i=1}^k B(R, h_R^{(i)}, \eta(h_R^{(i)}))$
covers $M(R, h_{\partial R}, 2\delta)$,
where $B(R, h_R^{(i)}, \eta(h_R^{(i)}))$
is the set of asymptotic height functions $h_R \in M(R)$
that are uniformly within distance $\eta(h_R^{(i)})$ of $h_R^{(i)}$.
Note that the number $k$ of sets in this cover depends only on $\delta$.
We abbreviate $\eta_i := \eta(h_R^{(i)})$.
Moreover, we fix $n_i \in \N$ such that for all $n \ge n_i$,
the $\theta_\delta(\tfrac{1}{n})$ from~\eqref{e_variational_limit} satisfies
\begin{equation} \label{e_variational_def_n}
	\theta_{\eta_i}(\tfrac{1}{n}) \le \varepsilon \,.
\end{equation}

We use the cover of the set $M(R, h_{\partial R}, 2\delta)$
to cover the set of height functions $M(R_n, h_{\partial R_n}, \delta)$.
Indeed, consider an arbitrary height function
$h_{R_n} \in M(R_n, h_{\partial R_n}, \delta)$.
After rescaling and interpolating (via the classical Kirszbraun theorem),
we identify $h_{R_n}$ with a continuous function
in $M(R, h_{\partial R}, 2\delta)$.
Under this identification,
\[
	M(R_n, h_{\partial R_n}, \delta)
	\subseteq M(R, h_{\partial R}, 2\delta) \\
	\subseteq \bigcup_{i=1}^k B(R, h_R^{(i)}, \eta_i) \,.
\]
This means that for any discrete height function
$h_{R_n} \in M(R_n, h_{\partial R_n}, \delta)$
with continuous (rescaled) interpolation
$\tilde h_{R_n} \in M(R, h_{\partial R}, 2\delta)$
(note that $\delta$ increases to $2\delta$ from discretization errors),
there is $i \in \{1, \dotsc, k\}$ such that
$\sup_{x \in R} |\tilde h_{R_n}(x) - h_R^{(i)}(x)| < \eta_i$.
By Definition~\ref{d_ht_func_sets}, it follows that

\begin{equation} \label{e_variational_cover}
	M(R_n, h_{\partial R_n}, \delta)
	\subseteq \bigcup_{i=1}^k B(R_n, h_R^{(i)}, \eta_i) \,.
\end{equation}
Hence,
\begin{equation} \label{e_variational_sum} \begin{aligned}
	\Ent_{R_n} \bigl( M(R_n, h_{\partial R_n}, \delta) \bigr)
	&\ge \Ent_{R_n} \left( \,
		\bigcup_{i=1}^k \, B(R_n, h_R^{(i)}, \eta_i)
	\right) \\
	&\ge -\frac{1}{|R_n|} \ln \left(
		\sum_{i=1}^k \, \bigl| B(R_n, h_R^{(i)}, \eta_i) \bigr|
	\right)
	\,.
\end{aligned} \end{equation}

Let us estimate $|B(R_n, h_R^{(i)}, \eta_i)|$.
Assuming that $n$ is larger than the constants $n_1, \dotsc, n_k$,
then for all $i=1, \dotsc, k$,
\begin{align*}
	\hskip3em&\hskip-3em
	\Ent_{R_n} \bigl( B(R_n, h_R^{(i)}, \eta_i) \bigr) \\
	&\ge \Ent_R ( h_R^{(i)} ) - 2\varepsilon
	&&\textit{(By \eqref{e_variational_limit},
		\eqref{e_variational_def_eta},
		and \eqref{e_variational_def_n})} \\
	&\ge \inf_{h_R \in M(R, h_{\partial R})} \Ent_R(h_R) - 2 \varepsilon \\
	&\ge \Ent_R ( h_R^{(1)} ) - 3\varepsilon \\
	&\ge \Ent_{R_n} \bigl( B(R_n, h_R^{(1)}, \eta_1) \bigr)
		- 5 \varepsilon
	&&\textit{(By \eqref{e_variational_limit},
		\eqref{e_variational_def_eta},
		and \eqref{e_variational_def_n})} \,.
\end{align*}
In other words,
\[
	\bigl| B(R_n, h_R^{(i)}, \eta_i ) \bigr|
	\le \bigl| B(R_n, h_R^{(1)}, \eta_1) \bigr|
		e^{5 \varepsilon |R_n|} \,.
\]

We apply this last estimate in~\eqref{e_variational_sum} to derive
\begin{equation} \label{e_variational_concl_eps} \begin{aligned}
	\Ent_{R_n} \bigl( M(R_n, h_{\partial R_n}, \delta) \bigr)
	&\ge -\frac{1}{|R_n|} \ln \biggl(
		k
		\bigl| B(R_n, h_R^{(1)}, \eta_1) \bigr|
		e^{5 \varepsilon |R_n|}
	\biggr) \\
	&= -\frac{1}{|R_n|} \ln \bigl| B(R_n, h_R^{(1)}, \eta_1) \bigr|
		- \frac{\ln k}{|R_n|} - 5\varepsilon \\
	&= -\frac{1}{|R_n|} \ln \bigl| B(R_n, h_R^{(1)}, \eta_1) \bigr|
		- \theta_\delta(\tfrac{1}{n}) - 5\varepsilon \\
	&= -\inf_{h_R \in M(R, h_{\partial R})} \Ent_R(h_R)
		- \theta_\delta(\tfrac{1}{n}) - 6\varepsilon \,.
\end{aligned} \end{equation}
Here, note that since $k$ depends only on $\delta$,
$\tfrac{k}{|R_n|} = \theta_\delta(\tfrac{1}{n})$.
Because $\varepsilon > 0$ was arbitrary,
this yields the desired estimate~\eqref{e_variational_underest}.
\end{proof}

\section{Large deviations principle} \label{s_ldp}

In this section we prove Theorem~\ref{th_ldp}, the large deviations principle.
For the reader's convenience, we recall the following definitions
from the statement of the theorem in Section~\ref{ss_main_results}.
For $\delta > 0$, $n \in \N$, and $h_R \in M(R)$:
\[ \begin{aligned}
	\mu_{\delta,n} &:= \frac{1}{|M(R_n, h_{\partial R_n}, \delta)|}
		\, \bigl| \bigl\{ h_{R_n} \in M(R_n, h_{\partial _n}, \delta)
			\,\big|\, \tilde h_{R_n} \in A \bigr\} \bigr| \,, \\
	r_{\delta,n} &:= |R_n| \,, \\
	I(h_R) &:= \begin{cases}
		\Ent_R(h_R) - E
		&\text{if $h_R|_{\partial R} = h_{\partial R}$} \,, \\
		\infty
		&\text{otherwise} \,,
	\end{cases}
\end{aligned} \]
where $E = \inf_{h_R \in M(R, h_{\partial R})} \Ent_R(h_R)$.

The proof of the large deviations principle that we give here
is based on the proof of the variational principle,
Theorem~\ref{th_varnprin}, given in Section~\ref{s_varnprin}.
We encourage the reader to read Section~\ref{s_varnprin} first.

\begin{proof}[Proof of Theorem~\ref{th_ldp}]

First, we prove the LDP lower bound~\eqref{e_ldp_lower_literal}, i.e.
\[
	- \inf_{h_R \in \interior{A}} I(h_R)
	\le \varliminf_{\delta \to 0} \varliminf_{n \to \infty}
		\frac{1}{r_{\delta,n}} \log \mu_{\delta,n}(A) \,.
\]

Without loss of generality we may assume that $A$ is open.
We may assume also that $\inf_{h_R \in A} I(h_R) < \infty$,
or else~\eqref{e_ldp_lower_literal} is trivial.
By using these assumptions
and replacing the symbols $\mu_{\delta,n}$, $r_{\delta,n}$, and $I(h_R)$
by their definitions,
\eqref{e_ldp_lower_literal} simplifies to
\begin{equation} \begin{aligned}
	\hskip3em&\hskip-3em
	-\inf_{h_R \in A}
		\Bigl( \Ent_R(h_R) - E \Bigr) \\
	&\le \varliminf_{\delta \to 0} \varliminf_{n \to \infty}
			\frac{\bigl|\bigl\{
				h_{R_n} \in M(R_n, h_{\partial R_n}, \delta)
				\big| \tilde h_{R_n} \in A
			\bigr| \bigr\}}{|M(R_n, h_{\partial R_n}, \delta)|}
		\Biggr) \,.
\end{aligned} \end{equation}

Simplifying further by multiplying by $-1$ and using our definition of the
microscopic entropy~\eqref{e_micro_ent},
it suffices to prove
\begin{equation} \label{e_ldp_lower_precancel} \begin{aligned}
	\hskip3em&\hskip-3em
	\inf_{h_R \in A} \Ent_R(h_R) - E \\
	&\ge \varlimsup_{\delta \to 0} \varlimsup_{n \to \infty}
		\Ent_{R_n} \bigl( \bigl\{
			h_{R_n} \in M(R_n, h_{\partial R_n}, \delta)
			\big| \tilde h_{R_n} \in A \bigr\} \bigr) \\
	&\qquad - \varliminf_{\delta \to 0} \varliminf_{n \to \infty}
		\Ent_{R_n} \bigl( M(R_n, h_{\partial R_n}, \delta) \bigr) \,.
\end{aligned} \end{equation}

By the variational principle (Theorem~\ref{th_varnprin}),
\[
	\varliminf_{\delta \to 0} \varliminf_{n \to \infty}
		\Ent_{R_n} \bigl( M(R_n, h_{\partial R_n}, \delta) \bigr)
	= E \,.
\]
After cancelling the corresponding terms in~\eqref{e_ldp_lower_precancel},
and after replacing $\varlimsup$ by our preferred $\theta$ asymptotics,
it suffices to show that
\begin{equation} \label{e_ldp_lower} \begin{aligned}
	\inf_{h_R \in A} \Ent_R(h_R)
	&\ge \Ent_{R_n} \bigl( \bigl\{
		h_{R_n} \in M(R_n, h_{\partial R_n}, \delta)
		\big| \tilde h_{R_n} \in A \bigr\} \bigr) \\
	&\qquad + \theta_A(\delta) + \theta_{A,\delta}(\tfrac{1}{n}) \,.
\end{aligned} \end{equation}

Note the analogy between~\eqref{e_ldp_lower}
and inequality~\eqref{e_variational_overest}
from the proof of the variational principle.
Indeed, we prove~\eqref{e_ldp_lower} in a similar manner
to~\eqref{e_variational_overest}.
We fix a sequence of asymptotic height function~$h_R^{(k)} \in A$
that saturates the infimum;
for concreteness, let us take
\[
	\Ent_R(h_R^{(k)})
	\le \inf_{h_R \in A} \Ent_R(h_R) + \frac{1}{k} \,.
\]

Write $\theta^{(k)}(\delta)$ and $\theta_\delta^{(k)}(\tfrac{1}{n})$
for the error terms from the profile theorem for $h_R^{(k)}$.
Choose a decreasing sequence $\delta_k$ such that
$\delta_k \le \tfrac{1}{2} \delta_{k-1}$
and such that $\theta^{(k)}(\delta_k) \le \frac{1}{k}$.
Given $\delta > 0$, choose $k$ such that $\delta \in (\delta_k, \delta_{k-1}]$;
then by the profile theorem applied to $h_R^{(k)}$,
\[ \begin{aligned}
	\hskip3em&\hskip-3em
	\Ent_{R_n} \bigl( \bigl\{
		h_{R_n} \in M(R_n, h_{\partial R_n}, \delta)
		\,\big|\, \tilde h_{R_n} \in A \bigr\} \bigr) \\
	&\le \Ent_{R_n} \bigl( B(R_n, h_R^{(k)}, \delta_k) \bigr) \\
	&\le \inf_{h_R \in A} \Ent_R(h_R) + \frac{2}{k}
		+ \theta_{\delta_k}^{(k)}(\tfrac{1}{n}) \\
	&= \inf_{h_R \in A} \Ent_R(h_R)
		+ \theta(\delta) + \theta_\delta(\tfrac{1}{n}) \,.
\end{aligned}\]

\medskip

Now, we turn to the LDP upper bound~\eqref{e_ldp_upper_literal}, i.e.
\[
	\varlimsup_{\delta \to 0} \varlimsup_{n \to \infty}
		\frac{1}{r_{\delta,n}} \log \mu_{\delta,n}(A)
	\le - \inf_{h_R \in \closure{A}} I(h_R) \,,
\]

We observe that $(\mu_{\delta,n})_{\delta,n}$ is exponentially tight,
i.e.~that for every $b \in (0, \infty)$,
there exists $K_b \subset M(R)$ such that
\[
	\varlimsup_{\delta \to 0} \varlimsup_{n \to \infty}
		\frac{1}{r_{\delta,n}} \log \mu_{\delta,n}(K_b^c)
	\le -b \,.
\]
Indeed, we may take $K_b$ to be the closure of $M(h_R, h_{\partial R}, 1)$,
independent of $b$.
For $\delta < \tfrac{1}{3}$ and $n$ large enough that
\[
	\max_{z \in \partial R_n} \bigl|
		\tfrac{1}{n}h_{\partial R_n}(z)
			- h_{\partial R} \bigl( \tfrac{1}{n}z \bigr)
	\bigr|
	\le \tfrac{1}{3} \,,
\]
any $h_{R_n} \in M(R_n, h_{\partial R}, \delta)$
satisfies $\tilde h_{R_n} \in M(h_R, h_{\partial R}, 1)$\
by the triangle inequality,
so $\mu_{\delta,n}(K_b^c) = 0$.
By the general theory of large deviations, exponential tightness implies that
it is sufficient prove the upper bound~\eqref{e_ldp_upper_literal}
for compact sets $A \subset M(R)$.

If $\inf_{h_R \in A} I(h_R) = \infty$,
then every height function in $A$
differs from $h_{\partial R}$ at some point on the boundary.
In fact by compactness, there exists $\delta_0$ such that
for every $h_R \in A$,
$\sup_{x \in \partial R} |h_{\partial R}(x) - h_R(x)| \ge \delta_0$.
Clearly, as in the proof of exponential tightness above,
this implies that
$\{ h_{R_n} \in M(R_n, h_{\partial R_n}, \delta) \,|\, \tilde h_{R_n} \in A \}$
is empty once $\delta$ is small enough and $n$ large enough.
For all such $\delta,n$ we have $\mu_{\delta,n}(A) = 0$
and~\eqref{e_ldp_upper_literal} follows.

It remains to prove the upper bound~\eqref{e_ldp_upper_literal}
when $\inf_{h_R \in A} I(h_R) < \infty$ and $A$ is compact.
Just like for the lower bound before,
we reduce to proving the following inequality:
\begin{equation} \label{e_ldp_upper} \begin{aligned}
	\inf_{h_R \in A} \Ent_R(h_R)
	&\le \Ent_{R_n} \bigl( \bigl\{
		h_{R_n} \in M(R_n, h_{\partial R_n}, \delta)
		\big| \tilde h_{R_n} \in A \bigr\} \bigr) \\
	&\qquad + \theta_A(\delta) + \theta_{A,\delta}(\tfrac{1}{n}) \,.
\end{aligned} \end{equation}

We will closely follow the proof of~\eqref{e_variational_underest}
from Theorem~\ref{th_varnprin}.
Let $\varepsilon > 0$,
and choose $h_R^{(1)}$
such that
\begin{equation} \label{e_ldp_upper_inf}
	\Ent_R(h_R^{(1)}) \le \inf_{h_R \in A} \Ent_R(h_R) + \varepsilon \,.
\end{equation}
As in~\eqref{e_variational_cover},
choose $h_R^{(2)}, \dotsc, h_R^{(k)}$ such that
\[
	A \subset \bigcup_{i=1}^k B(R, h_R^{(i)}, \eta_i) \,,
\]
where $\eta_1, \dotsc, \eta_k$ are chosen so that for each $i$,
the $\theta(\delta)$ term from the profile theorem for $h_R^{(i)}$
satisfies $\theta(\eta_i) \le \varepsilon$.
Exactly as in the proof of Theorem~\ref{th_varnprin}
(see~\eqref{e_variational_sum}),
\[ \begin{aligned}
	\hskip3em&\hskip-3em
	\Ent_{R_n} \bigl( \bigl\{
		h_{R_n} \in M(R_n, h_{\partial R_n}, \delta)
		\,\big|\, \tilde h_{R_n} \in A \bigr\} \bigr) \\
	&\ge -\frac{1}{|R_n|} \ln
		\sum_{i=1}^k \, \bigl| B(R_n, h_R^{(i)}, \eta_i) \bigr| \,.
\end{aligned} \]
From this we deduce the analogue of~\eqref{e_variational_concl_eps}, namely
\[ \begin{aligned}
	\hskip3em&\hskip-3em
	\Ent_{R_n} \bigl( \bigl\{
		h_{R_n} \in M(R_n, h_{\partial R_n}, \delta)
		\,\big|\, \tilde h_{R_n} \in A \bigr\} \bigr) \\
	&\ge \Ent_R(h_R^{(1)})
		+ \theta_{A,\delta}(\tfrac{1}{n}) - 5\varepsilon \\
	&\ge \inf_{h_R \in A} \Ent_R(h_R)
		+ \theta_{A,\delta}(\tfrac{1}{n}) - 6\varepsilon \,.
\end{aligned} \]
\end{proof}

\section*{Acknowledgment}
The authors want to thank
Tim Austin, Marek Biskup, Thomas Liggett, Igor Pak, Greta Panova
for the fruitful discussions and helpful comments.
This research has been partially supported by NSF grant DMS-1712632.

\bibliographystyle{alpha}
\bibliography{bib}

\begin{thebibliography}{CPST18}

\bibitem[BCG16]{BCG2016}
Alexei Borodin, Ivan Corwin, and Vadim Gorin.
\newblock Stochastic six-vertex model.
\newblock {\em Duke Math. J.}, 165(3):563--624, 2016.

\bibitem[CEP96]{CEP96}
Henry Cohn, Noam Elkies, and James Propp.
\newblock Local statistics for random domino tilings of the {A}ztec diamond.
\newblock {\em Duke Math. J.}, 85(1):117--166, 1996.

\bibitem[Cer06]{Cer06}
Rapha\"el Cerf.
\newblock {\em The {W}ulff Crystal in {I}sing and Percolation Models}, volume
  1878 of {\em Lecture Notes in Mathematics}.
\newblock Springer-Verlag, Berlin, 2006.

\bibitem[CKP01]{CKP01}
Henry Cohn, Richard Kenyon, and James Propp.
\newblock A variational principle for domino tilings.
\newblock {\em J. Amer. Math. Soc.}, 14(2):297--346 (electronic), 2001.

\bibitem[CPST18]{ChPeShTa18}
Nishant Chandgotia, Ron Peled, Scott Sheffield, and Martin Tassy.
\newblock Delocalization of uniform graph homomorphisms from $\mathbb{Z}^2$ to
  $\mathbb{Z}$.
\newblock {\em arXiv:1810.10124 [math.PR]}, 2018.

\bibitem[CPT18]{ChPaTa18}
Nishant Chandgotia, Igor Pak, and Martin Tassy.
\newblock {K}irszbraun-type theorems for graphs.
\newblock {\em J. Comb. Theory Ser. B}, 2018.

\bibitem[CS16]{CoSp16}
Filippo Colomo and Andrea Sportiello.
\newblock Arctic curves of the six-vertex model on generic domains: The tangent
  method.
\newblock {\em J. Stat. Phys.}, 164(6):1488--1523, 2016.

\bibitem[Des98]{Des98}
Nicolas Destainville.
\newblock Entropy and boundary conditions in random rhombus tilings.
\newblock {\em J. Phys. A}, 31(29):6123--6139, 1998.

\bibitem[DGI00]{DeGiIo00}
Jean-Dominique Deuschel, Giambattista Giacomin, and Dmitry Ioffe.
\newblock Large deviations and concentration properties for {$\nabla\phi$}
  interface models.
\newblock {\em Probab. Theory Relat. Fields}, 117(1):49--111, 2000.

\bibitem[DKS92]{DKS92}
Roland~Lvovich Dobrushin, Roman Koteck{\`y}, and Senya Shlosman.
\newblock {\em {W}ulff Construction}, volume 104 of {\em Translations of
  Mathematical Monographs}.
\newblock American Mathematical Society, Providence, RI, 1992.
\newblock A global shape from local interaction, Translated from the Russian by
  the authors.

\bibitem[DSS08]{DSSa08}
Daniela De~Silva and Ovidiu Savin.
\newblock Minimizers of convex functionals arising in random surfaces.
\newblock {\em Duke Math. J.}, 151, 10 2008.

\bibitem[Dur10]{Durrett}
Rick Durrett.
\newblock {\em Probability: Theory and Examples}.
\newblock Cambridge Series in Statistical and Probabilistic Mathematics.
  Cambridge University Press, Cambridge, fourth edition, 2010.

\bibitem[DZ09]{DeZe09}
Amir Dembo and Ofer Zeitouni.
\newblock {\em Large Deviations Techniques and Applications}.
\newblock Stochastic Modelling and Applied Probability. Springer Berlin
  Heidelberg, 2009.

\bibitem[FO04]{FuOs04}
Tadahisa Funaki and Hirofumi Osada.
\newblock {\em Stochastic Analysis on Large Scale Interacting Systems}.
\newblock Advanced studies in pure mathematics. Mathematical Society of Japan,
  2004.

\bibitem[FS06]{FS06}
Patrik~L. Ferrari and Herbert Spohn.
\newblock Scaling limit for the space-time covariance of the stationary totally
  asymmetric simple exclusion process.
\newblock {\em Commun. Math. Phys.}, 265(1):1--44, 2006.

\bibitem[Kas63]{Kas63}
Pieter~W Kasteleyn.
\newblock Dimer statistics and phase transitions.
\newblock {\em J. Math. Phys.}, 4:287--293, 1963.

\bibitem[Kir34]{Kir34}
Moj\.zesz~D Kirszbraun.
\newblock {\"U}ber die zusammenziehende und {L}ipschitzsche {T}ransformationen.
\newblock {\em Fundam. Math.}, 22(1):77--108, 1934.

\bibitem[KK92]{KeKe92}
Claire Kenyon and Rick Kenyon.
\newblock Tiling a polygon with rectangles.
\newblock {\em Proc. of 33rd Fundamentals of Computer Science (FOCS)}, pages
  610--619, 1992.

\bibitem[KMT17]{KMT17}
Anndrew Krieger, Georg Menz, and Martin Tassy.
\newblock A quenched variational principle for discrete random maps.
\newblock {\em arXiv:1710.11330}, 2017.

\bibitem[KS99]{KS99}
Michael K.-H. Kiessling and Herbert Spohn.
\newblock A note on the eigenvalue density of random matrices.
\newblock {\em Commun. Math. Phys.}, 199(3):683--695, Jan 1999.

\bibitem[LP08]{LP08}
Lionel Levine and Yuval Peres.
\newblock Strong spherical asymptotics for rotor-router aggregation and the
  divisible sandpile.
\newblock {\em Potential Anal.}, 30(1):1, 2008.

\bibitem[LRS01]{LRS01}
Michael Luby, Dana Randall, and Alistair Sinclair.
\newblock {M}arkov chain algorithms for planar lattice structures.
\newblock {\em SIAM J. Comput.}, 31(1):167--192, 2001.

\bibitem[LS77]{LS77}
Benjamin~F Logan and Larry~A Shepp.
\newblock A variational problem for random {Y}oung tableaux.
\newblock {\em Adv. in Math.}, 26(2):206--222, 1977.

\bibitem[MT16]{MeTa16}
Georg Menz and Martin Tassy.
\newblock A variational principle for a non-integrable model.
\newblock {\em arXiv:1610.08103 [math.PR]}, 2016.

\bibitem[PR07]{PR07}
Boris Pittel and Dan Romik.
\newblock Limit shapes for random square {Y}oung tableaux.
\newblock {\em Adv. in Appl. Math.}, 38(2):164--209, 2007.

\bibitem[RAS15]{RaSe15}
Firas Rassoul-Agha and Timo Seppäläinen.
\newblock {\em A Course on Large Deviations with an Introduction to {G}ibbs
  Measures}.
\newblock American Mathematical Society, Providence, RI, 05 2015.

\bibitem[RS18]{ReSr16}
Nicolai Reshetikhin and Ananth Sridhar.
\newblock Limit shapes of the stochastic six vertex model.
\newblock {\em Commun. Math. Phys.}, 363(3):741--765, 2018.

\bibitem[Sch14]{Sch14}
Jean~Van Schaftingen.
\newblock Approximation in {S}obolev spaces by piecewise affine interpolation.
\newblock {\em J. Math. Anal. Appl.}, 420(1):40--47, 2014.

\bibitem[She01]{She02}
Scott Sheffield.
\newblock Ribbon tilings and multidimensional height functions.
\newblock {\em Trans. Am. Math. Soc.}, 354, 08 2001.

\bibitem[She05]{She05}
Scott Sheffield.
\newblock Random surfaces.
\newblock {\em Ast\'erisque}, (304):vi+175, 2005.

\bibitem[SS95]{StSe95}
Lynn~Arthur Steen and J~Arthur Seebach.
\newblock {\em Counterexamples in Topology}.
\newblock Dover books on mathematics. Dover Publications, 1995.

\bibitem[Thu90]{Thu90}
William~P. Thurston.
\newblock {C}onway's tiling groups.
\newblock {\em Am. Math. Monthly}, 97(8):757--756, September 1990.

\bibitem[vB77]{Henk77}
Henk van Beijeren.
\newblock Exactly solvable model for the roughening transition of a crystal
  surface.
\newblock {\em Phys. Rev. Lett.}, 38(18), 1977.

\bibitem[VK77]{VK77}
Anatolii~Moiseevich Ver{\v{s}}ik and Sergei~V. Kerov.
\newblock Asymptotic behavior of the {P}lancherel measure of the symmetric
  group and the limit form of {Y}oung tableaux.
\newblock {\em Dokl. Akad. Nauk SSSR}, 233(6):1024--1027, 1977.

\bibitem[Wig59]{Wig59}
Eugene~Paul Wigner.
\newblock Statistical properties of real symmetric matrices with many
  dimensions.
\newblock In {\em 4th Can. Math. Congress (Banff 1957)}, pages 174--184. Univ.
  Toronto Press, 1959.

\bibitem[Wil04]{Wil04}
David~Bruce Wilson.
\newblock Mixing times of lozenge tiling and card shuffling {M}arkov chains.
\newblock {\em Ann. Appl. Probab.}, 14(1):274--325, 02 2004.

\end{thebibliography}

\end{document}